\documentclass[11pt,a4paper]{article}
\usepackage[a4paper]{geometry}
\usepackage{amssymb,latexsym,amsmath,amsfonts,amsthm}
\usepackage{graphicx}
\usepackage{epstopdf}
\usepackage{float,amsmath,amssymb,mathrsfs,bm,multirow,graphics}
\usepackage{tikz}
\usepackage{pgflibraryshapes}
\usetikzlibrary{arrows,decorations.markings}
\usepackage{subfig}
\usepackage[percent]{overpic}
\usepackage{fullpage}
\usepackage{comment,verbatim}
\usepackage{hyperref}
\usepackage{color}
\usepackage{mathrsfs}
\usepackage[title,titletoc]{appendix}
\usepackage{ulem}
\usepackage{enumitem}
\usepackage{mathabx}
\DeclareMathOperator{\diag}{diag}
\DeclareMathOperator*{\Res}{Res}

\DeclareMathOperator{\CHF}{CHF}
\DeclareMathOperator{\tr}{Tr}
\DeclareMathOperator{\Var}{Var}
\newcommand{\Bes}{\mathrm{Bes}}

\newcommand{\C}{\mathbb{C}}

\renewcommand{\Re}{\mathrm{Re}\,}
\renewcommand{\Im}{\mathrm{Im}\,}

\renewcommand{\vec}{\mathbf}
\newcommand{\ud}{\,\mathrm{d}}

\newcommand{\what}{\widehat}

\newcommand{\msf}{\mathsf}
\newcommand{\ii}{\mathrm{i}}

\newcommand{\Boh}{\mathcal{O}}

\newtheorem{theorem}{Theorem}[section]
\newtheorem{lemma}[theorem]{Lemma}
\newtheorem{proposition}[theorem]{Proposition}

\newtheorem{corollary}[theorem]{Corollary}

\newtheorem{rhp}[theorem]{RH problem}

\newcommand{\tac}{\mathrm{htac}}
\newcommand{\HM}{\mathrm{HM}}

\theoremstyle{definition}

\theoremstyle{remark}

\newtheorem{remark}[theorem]{Remark}

\numberwithin{equation}{section}

\begin{document}
	\title{Large gap asymptotics of the hard edge tacnode process}
	
	\author{Junwen Liu\footnotemark[1],\quad Luming Yao\footnotemark[2], \quad Lun Zhang\footnotemark[3]}
	
	\renewcommand{\thefootnote}{\fnsymbol{footnote}}
	\footnotetext[1]{School of Mathematical Sciences, Fudan University, Shanghai 200433, P. R. China. E-mail: junwenliu21\symbol{'100}m.fudan.edu.cn}
	\footnotetext[2]{Institute for Advanced Study, Shenzhen University, Shenzhen 518060, P. R. China. E-mail: lumingyao\symbol{'100}szu.edu.cn}
 \footnotetext[3]{School of Mathematical Sciences, Fudan University, Shanghai 200433, P. R. China. E-mail: lunzhang\symbol{'100}fudan.edu.cn}
	
	\date{\today}
	
	\maketitle

\begin{abstract}
A special type of geometric situation in ensembles of non-intersecting paths occurs when the non-intersecting trajectories are required to be nonnegative so that the limit shape becomes tangential to the hard-edge $0$. The local fluctuation is governed by the universal hard edge tacnode process, which also arises from some tiling problems. It is the aim of this work to explore the integrable structure and asymptotics for the gap probability of the hard edge thinned/unthinned tacnode process over $(0,s)$. We establish an integral representation of the gap probability in terms of the Hamiltonian associated with a system of differential equations. With the aids of some remarkable differential identities for the Hamiltonian, we  are able to derive the associated large gap asymptotics, up to and including the constant term in the thinned case. Some applications of our results are discussed as well. 
\end{abstract}

\tableofcontents
	
\section{Introduction}
	
	Since Dyson's seminal work on non-intersecting Brownian motions and the evolution of eigenvalues of the Gaussian unitary ensemble \cite{Dyson}, ensembles of non-intersecting paths have appeared in a variety of physical, combinatorial and probabilistic models. Some examples include the polynuclear growth models \cite{John03,PS}, directed polymers \cite{Comet} and random tilings of various domains \cite{AV23,John06,John18}, just to mention a few. These connections are particularly helpful in establishing edge fluctuations of limit shapes, which give rise to universal limit laws related to random matrix theory and the KPZ universality class.
	
	A special type of geometric situation occurs when the boundaries of the limit shape touch each other creating a tacnode point. The fluctuation around the touching point will lead to a critical determinantal point process called the tacnode process. By using different approaches, the correlation kernel of tacnode process was given in the context of non-intersecting random walks on $\mathbb{Z}$ in \cite{AFV13}, of non-intersecting Brownian motions in \cite{DKZ11,FV12,John13,LW16}, and of double Aztec diamonds in \cite{ACJV15,AJV14}.
	
	The tacnode geometry appears also if the non-intersecting trajectories are required to be nonnegative so that the limit shape becomes tangential to the hard-edge $0$. This hard edge tacnode process was first reported in the setting of non-intersecting squared Bessel process \cite{Del13}. More precisely, let $X(t)$ be a squared Bessel process with parameter $\alpha>-1$, i.e., a diffusion process on the positive real axis with transition probability density given by (cf.~\cite{AndreiNPaavo,
		KonigOconnell})
	\begin{align}
		p_t^{(\alpha)}(x,y)&=\frac{1}{2t}\left(\frac{y}{x}\right)^{\frac{\alpha}{2}}e^{-\frac{x+y}{2t}}
		I_{\alpha}\left(\frac{\sqrt{xy}}{t}\right), &\qquad x,y>0,\label{p_t alpha} \\
		p_t^{(\alpha)}(0,y)&=\frac{y^{\alpha}}{(2t)^{\alpha+1}\Gamma(\alpha+1)}e^{-\frac{y}{2t}},
		&\qquad y>0,\label{p_t alpha 0}
	\end{align}
	where
	\begin{equation}
		I_{\alpha}(z)=\sum_{k=0}^{\infty}\frac{(z/2)^{2k+\alpha}}{k!\Gamma(k+\alpha+1)}
	\end{equation}
	is the modified Bessel function of the first kind of order $\alpha$ (cf. \cite[\S 10.25]{DLMF}) and $\Gamma(\cdot)$ is the Gamma function. If $d = 2(\alpha + 1)$ is an integer, it can be obtained as the square of the distance to the origin of a $d$-dimensional Brownian motion. By introducing a rescaling of the time
	$$
	t \mapsto \frac{T}{2n}t
	$$
	with $T$ being interpreted as the temperature, the model consists of $n$ independent copies of $X(t)$ such that they all start in $a \geq 0 $ at $t = 0$, remain nonnegative, end in $b \geq 0$ at $t = 1$, and do not intersect one another for $0 < t < 1$; see also \cite{Kat12,KIK08,KatTan04,KatTan11} for different types of initial and ending conditions.
	
  If $a=b=0$, i.e., all paths start and end at the origin, this system provides a process version of the Laguerre unitary ensemble, cf.~\cite{DF,KIK08,KonigOconnell}. The case where $a>0$ and $b=0$ was considered in \cite{KMW09,KMW11}, and in \cite{DKRZ12} for $a,b>0$. For the latter case, by assuming $T=1$ and letting $n\to\infty$, the paths fill out one of the regions of as shown in Figure \ref{fig:nsbp}.  It comes out that classification of the three distinguished cases depends on the product $ab$, as discussed in \cite{DKRZ12}. If $ab>1/4$, all paths remain positive for all the time and we are in a situation illustrated by the left picture in Figure \ref{fig:nsbp}. If $ab<1/4$, there are two critical times $t_1$ and $t_2$, as depicted in the middle picture of Figure \ref{fig:nsbp}. All paths hit the hard edge and stick to it during the time interval $(t_1,t_2)$. The times $t_1$ and $t_2$ will come together if $ab=1/4$, and we are then led to an intermediate situation shown in the right picture in Figure \ref{fig:nsbp}. The paths fill out a region bounded by the solid line, which is tangent to the hard edge at a multicritical time $t_c=\sqrt{a}/(\sqrt{a}+\sqrt{b})$.
	
	\begin{figure}
		\begin{center}
			\subfloat{\begin{overpic}[scale=0.3,viewport=79 210 551 595]{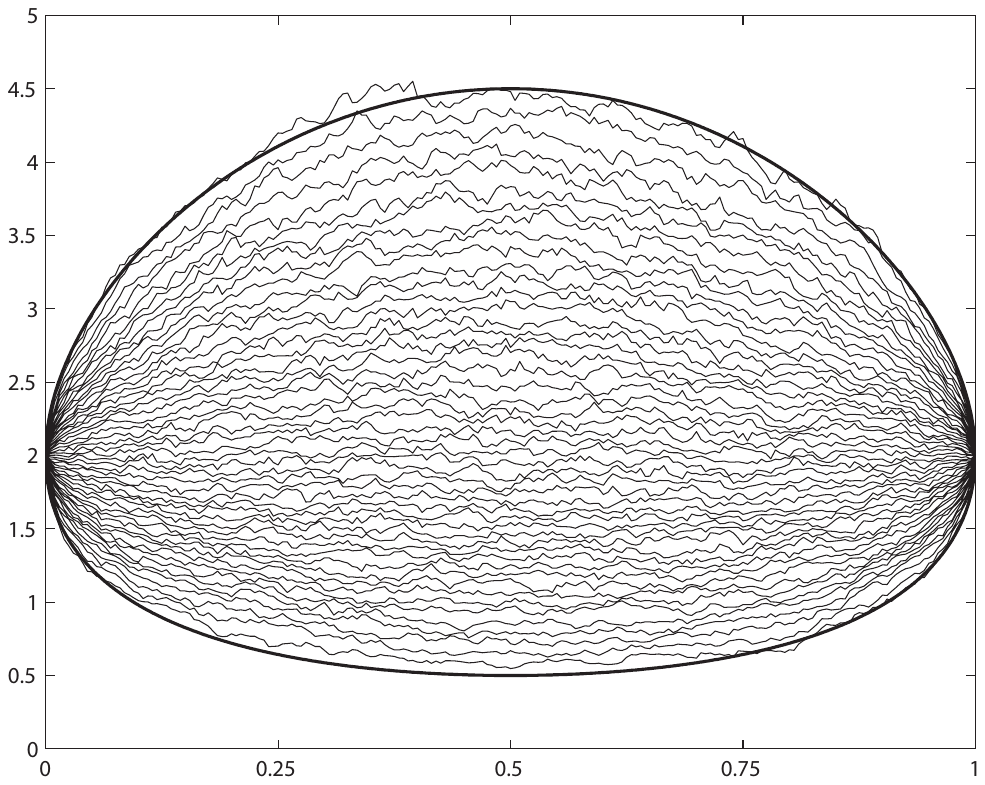}
			\end{overpic}}\hspace{1.5mm}
			\subfloat{\begin{overpic}[scale=0.3,viewport=79 210 551 595]{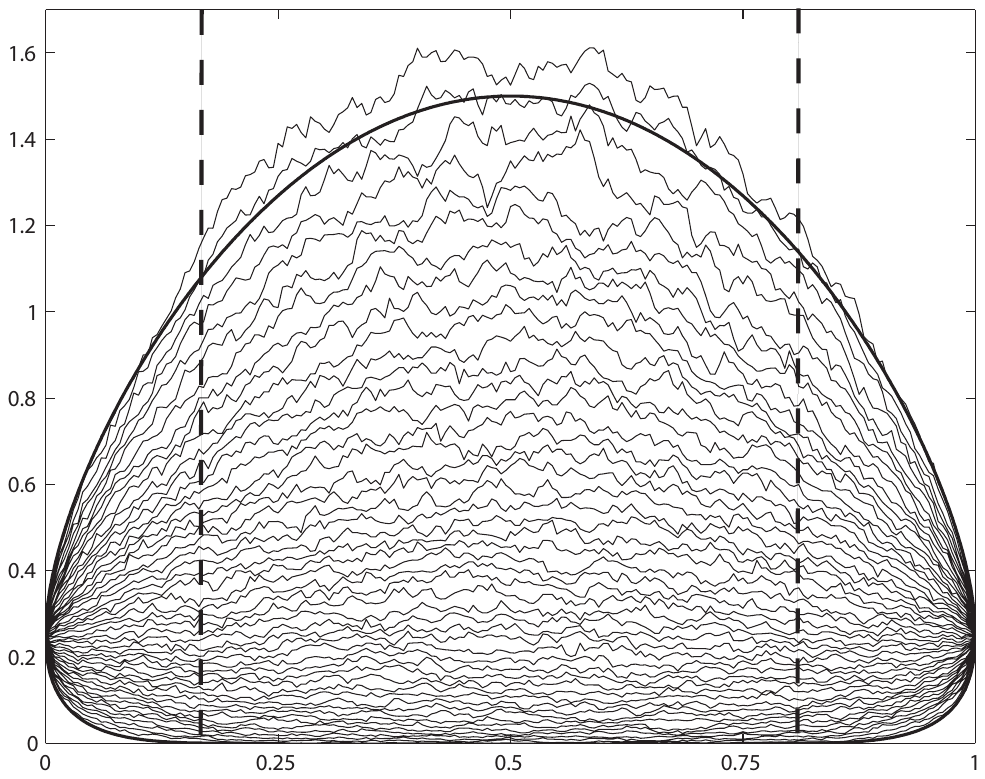}
					\put(15,-3){{\small $t_1$}} \put(77,-3){{\small $t_2$}}
			\end{overpic}}\hspace{1.5mm}
			\subfloat{\begin{overpic}[scale=0.3,viewport=79 210 551 595]{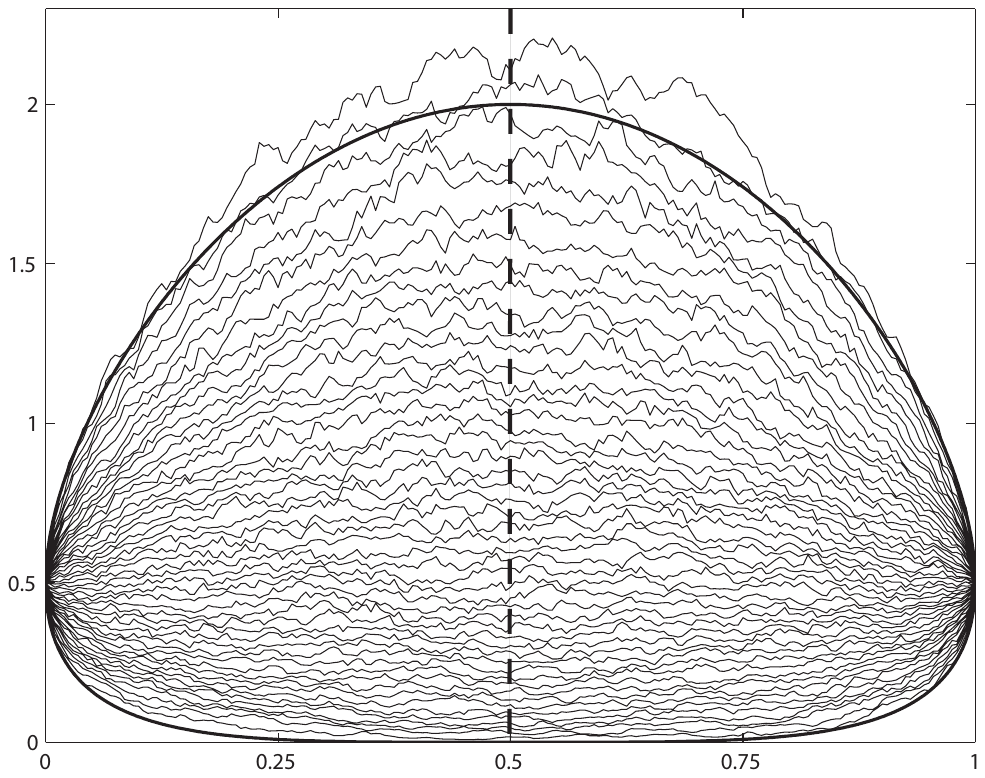} \put(49,-3){{\small $t_c$}}
			\end{overpic}}
		\end{center}
		\caption{Simulation picture of 50 rescaled non-intersecting squared Bessel paths at $T=1$ that start at $a$ and end at $b$ with $a=b=2$ (left), $a=b=1/4$ (middle) and $a=b=1/2$ (right).}
		\label{fig:nsbp}
	\end{figure}

As aforementioned, it is of great interest to study the random fluctuations around the limit shapes. For the non-intersecting squared Bessel paths model with one positive starting and ending point, the positions of these paths form a determinantal point process. Generically, the local path statistics is governed by the Airy process and the Bessel process at the soft edge and the hard edge of the limit shape \cite{DKRZ12}, respectively, which are canonical universal point processes in random matrix theory, cf. \cite{Deift1999,For,Mehta}. After rescaling around the two critical times $t_1$ and $t_2$, it is expected that one encounters the hard edge Pearcey process \cite{BK10,DF,KMW11}, corresponding to a hard-to-soft transition; see also \cite{DV15} for a multi-time extension. The hard edge tacnode process then describes local correlations of the paths around the tacnode point $t_c$. Assume that the endpoints $a$ and $b$ are fixed with $ab=1/4$, and let the time $t$ and the temperature $T$ depend on $n$ such that
$$ 
t=t_c+Kn^{-\frac{1}{3}},\quad
T=1+Ln^{-\frac{2}{3}}, \quad K,L\in \mathbb{R}. 
$$
Delvaux in \cite{Del13} showed that the triple scaling
limit of the correlation kernel for the non-intersecting squared Bessel paths is given by the hard edge tacnode kernel $K_{\tac}(u,v;s^*,\tau^*,\alpha)$, where the paramerters $s^*$ and $\tau^*$ are functions explicitly given in terms of $a,b,K$ and $L$. The kernel $K_{\tac}$  was built therein with the aid of the solution of a $4 \times 4$ Riemann-Hilbert (RH) problem; see also \cite{DV15,LW17} for different forms and extensions. This RH problem has a remarkable connection with the Hastings-McLeod solution \cite{FIKNBook,HM} of the inhomogeneous Painlev\'{e} II equation
	\begin{align}\label{eq:PII}
		q''(x)=2q(x)^3+xq(x)-\nu, \qquad \nu>-1/2,
	\end{align}
and plays an important role in the description of a critical phenomenon in the chiral 2-matrix model \cite{DGZ}. Later appearances of the hard edge tacnode process in the context of non-intersecting Brownian bridges with reflecting or absorbing walls \cite{FV17,LW17} and of random domino tilings of the restricted Aztec diamond \cite{FV21} reveal its universal feature.
	
 Consideration in this paper is the gap probability of the hard tacnode process -- a basic object in the theory of point processes. More precisely, let $\mathcal{K}_{\tac}$ be the integral operator acting on $L^2\left(0, s\right)$, $s\geq 0$, with the hard edge tacnode kernel $K_{\tac}$.  We are interested in the associated Fredholm determinant $\det\left(I-\gamma  \mathcal K_{\tac}\right)$, where $0< \gamma \leq 1$ is a real parameter. On the one hand, it is well-known that, due to the determinantal structure, $\det\left(I-\gamma  \mathcal K_{\tac}\right)$ can be interpreted as the probability of finding no particles (a.k.a. the gap probability) on the interval $(0,s)$ for the hard edge tacnode process and the thinned version, which corresponds to $\gamma=1$ and $0<\gamma<1$, respectively. The thinned process is obtained from the original one by removing each particle independently with probability $1-\gamma$. On the other hand, intensive studies of various Fredholm determinants and their deformed versions arising from random matrix theory or beyond have unraveled their rich structures and elegant forms of the large gap asymptotics. The relevant results can be found in \cite{BBD08,BB18,DIK2008,TW94} for the Airy determinant, in \cite{BIP19,Charlier21,DKV11,Ehr10,TW94b} for the Bessel determinant, and in \cite{DXZ23,YZ24a} for the hard edge Pearcey determinant, among others.
	
It is the aim of this work to add to the collection of formulas for the very few basic gap probabilities by working on the hard edge tacnode determinant. Our results are stated in the next section.
	
	\paragraph{Notations} Throughout this paper, the following notations are frequently used.
	\begin{itemize}
		\item If $A$ is a matrix, then $(A)_{ij}$ stands for its $(i,j)$-th entry and $A^{\msf T}$ stands for its transpose. An unimportant entry of $A$ is denoted by $\ast$. We use  $I$ to denote an identity matrix, and the size might differ in different contexts. To emphasize a $k\times k$ identity matrix, we also use the notation $I_k$.
		\item  It is notationally convenient to denote by $E_{jk}$ the $4\times 4$ elementary matrix
		whose entries are all $0$, except for the $(j,k)$-entry, which is $1$, that is,
		\begin{equation}\label{def:Eij}
			E_{jk}=\left( \delta_{l,j}\delta_{k,m} \right)_{l,m=1}^4,
		\end{equation}
		where $\delta_{j,k}$ is the Kronecker delta.

		\item We denote by $D(z_0, r)$ the open disc centred at $z_0$ with radius $r > 0$, i.e.,
		\begin{equation}\label{def:dz0r}
			D(z_0, r) := \{ z\in \mathbb{C} \mid |z-z_0|<r \},
		\end{equation}
		and by $\partial D(z_0, r)$ its boundary. The orientation of $\partial D(z_0, r)$ is taken in a clockwise manner.
		\item As usual, the three Pauli matrices $\{\sigma_j\}_{j=1}^3$ are defined by
		\begin{equation}\label{def:Pauli}
			\sigma_1=\begin{pmatrix}
				0 & 1 \\
				1 & 0
			\end{pmatrix},
			\qquad
			\sigma_2=\begin{pmatrix}
				0 & -\ii \\
				\ii & 0
			\end{pmatrix},
			\qquad
			\sigma_3=
			\begin{pmatrix}
				1 & 0 \\
				0 & -1
			\end{pmatrix}.
		\end{equation}
	\end{itemize}
 
\section{Statement of results}

\subsection{A RH characterization of $K_{\tac}$}
In \cite{Del13}, the hard edge tacnode kernel $K_{\tac}$ is constructed by a matrix-valued function 
$$
\widehat M(z) := \diag(z^{\frac14},z^{-\frac14},z^{\frac14},z^{-\frac14})\diag\left(\begin{pmatrix} 1 & -1
\\
1 & 1
\end{pmatrix}
,
\begin{pmatrix}
1 & 1
\\
-1 & 1
\end{pmatrix}
\right)M(z^{\frac12}),
$$
where $M$ is the unique solution of a $4\times 4$ RH problem that generalizes the one in \cite{DKZ11,DG13}. Although not explicitly formulated therein, it follows from the above definition, \cite[RH Problem 1]{Del13} and direct calculations that  $\widehat M$ solves the following RH problem.

	\begin{rhp}\label{rhp:hard edge tac}
		\hfill
		\begin{itemize}
			\item[\rm (a)]  $\widehat M(z)=\widehat M(z; \tilde{s},\nu,\tau)$ is analytic for $ z \in \C \setminus \Gamma_{\what M} $, where the parameters $\tilde{s},\nu,\tau$ are real with $\nu >-\frac{1}{2}$, and
			\begin{equation}
				\Gamma_{\widehat M}:=\cup_{k=0}^5\Gamma_k \cup \{0\}
			\end{equation}
			with
			\begin{equation}\label{phi}
				\begin{aligned}
					&\Gamma_0= (0,+\infty),&& \Gamma_1=e^{\varphi \ii}(0,+\infty), &&\Gamma_2=e^{-\varphi \ii}(-\infty,0),\\
					&\Gamma_3= (-\infty,0),&& \Gamma_4=e^{\varphi \ii}(-\infty,0), &&\Gamma_5=e^{-\varphi \ii}(0,+\infty), \quad 0<\varphi<\frac{\pi}{3};
				\end{aligned}
			\end{equation}
			see Figure \ref{fig:hard edge tacnode} for an illustration of the contour $\Gamma_{\what M}$ and the orientation.
   \begin{figure}[t]
		\begin{center}
			\setlength{\unitlength}{1truemm}
			\begin{picture}(100,70)(-5,2)
				\put(40,40){\line(-1,-1){20}}
				\put(40,40){\line(-1,1){20}}
				\put(40,40){\line(-1,0){30}}
				\put(40,40){\line(1,0){30}}
				\put(40,40){\line(1,1){20}}
				\put(40,40){\line(1,-1){20}}
				
				\put(30,50){\thicklines\vector(1,-1){1}}
				\put(25,40){\thicklines\vector(1,0){1}}
				\put(30,30){\thicklines\vector(1,1){1}}
				\put(50,50){\thicklines\vector(1,1){1}}
				\put(55,40){\thicklines\vector(1,0){1}}
				\put(50,30){\thicklines\vector(1,-1){1}}

				\put(39,36.3){$0$}
				
				\put(27,22){$\Gamma_4$}
				\put(-21,18) {$\begin{pmatrix} 1&0&0&0\\ -e^{-\nu \pi \ii}&1&0&0\\0&0&1& e^{-\nu \pi \ii}\\0&0&0&1 \end{pmatrix}$}
				
				\put(27,55){$\Gamma_2$}
				\put(-21,66){$\begin{pmatrix} 1&0&0&0\\ e^{\nu \pi \ii}&1&0&0\\0&0&1& -e^{\nu \pi \ii}\\0&0&0&1 \end{pmatrix}$}
				
				\put(17,42){$\Gamma_3$}
				\put(-21,40){$\begin{pmatrix}0&-\ii&0&0\\-\ii&0&0&0\\0&0&0&\ii\\0&0&\ii&0 \end{pmatrix}$}
				
				\put(48,22){$\Gamma_5$}
				\put(60,18){$\begin{pmatrix} 1&0&0&0\\0&1&0&0\\1&0&1&0\\0&0&0&1 \end{pmatrix}$}
				
				\put(48,55){$\Gamma_1$}
				\put(60,66) {$\begin{pmatrix} 1&0&0&0 \\ 0 &1&0&0\\1&0&1&0\\0&0&0&1 \end{pmatrix}$}
				
				\put(58,42){$\Gamma_0$}
				\put(72,40){ $\begin{pmatrix}
						0&0&1&0\\0&1&0&0\\-1&0&0&0\\0&0&0&1 \end{pmatrix}$ }
				\put(37,55){$\Omega_1$}
				\put(53,48){$\Omega_0$}
				\put(22,48){$\Omega_2$}
				\put(22,32){$\Omega_3$}
				\put(37,22){$\Omega_4$}
				\put(53,32){$\Omega_5$}
				\put(40,40){\thicklines\circle*{1}}
				
			\end{picture}
			
			\caption{Regions $\Omega_k$, the jump contours $\Gamma_k$ and the corresponding jump matrices $J_{k}$, $k=0,\ldots,5$, in the RH problem for $\widehat{M}$.}
			\label{fig:hard edge tacnode}
		\end{center}
	\end{figure}
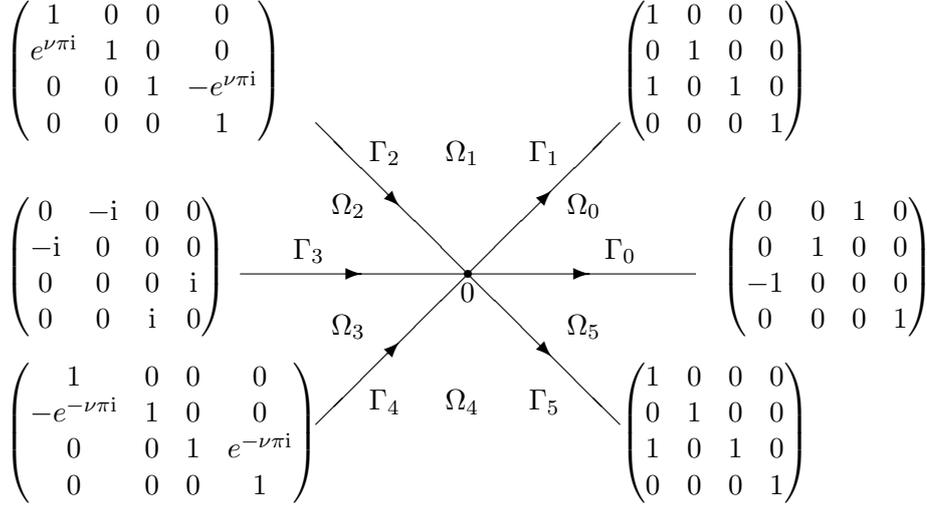
			\item[\rm (b)]  For $z\in\Gamma_k$, $k=0,1,\ldots,5$, we have
			\begin{equation}\label{jumps:Mhat}
				\widehat M_{+}(z) = \widehat M_{-}(z)J_k,\qquad k=0,\ldots,5,
			\end{equation}
			where the jump matrix $J_k$ for each ray $\Gamma_k$ is shown in Figure \ref{fig:hard edge tacnode}.
			\item[\rm (c)] As $z \to \infty$ with $z \in \mathbb{C} \setminus  \Gamma_{\widehat{M}}$, we have
			\begin{align}\label{eq:asy:Mhat}
					\widehat{M}(z)&= \diag \left(z^{\frac{1}{4}}, z^{-\frac{1}{4}},z^{\frac{1}{4}}, z^{-\frac{1}{4}}\right)    \diag \left(\begin{pmatrix} 1 & -1 \\ 1 & 1 \end{pmatrix},\begin{pmatrix} 1 & 1 \\ -1 & 1 \end{pmatrix}\right) \nonumber \\ 
	&\quad\times\left(I+\frac{M^{(1)}}{\sqrt{z}}+\frac{M^{(2)}}{z}+\Boh(z^{-\frac{3}{2}})\right)
					\diag \left((-\sqrt{z})^{-\frac{1}{4}}, z^{-\frac{1}{8}},(-\sqrt{z})^{\frac{1}{4}}, z^{\frac{1}{8}}\right) \nonumber\\
					&\quad\times A \diag \left(e^{-\theta_1(\sqrt{z})+\tau \sqrt{z}}, e^{-\theta_2(\sqrt{z})-\tau \sqrt{z}}, e^{\theta_1(\sqrt{z})+\tau \sqrt{z}}, e^{\theta_2(\sqrt{z})-\tau \sqrt{z}}\right),
 			\end{align}  
			where $M^{(1)}, M^{(2)}$ are independent of $z$ but depend on the parameter $\tilde{s},\nu,\tau$,
			\begin{align} \label{def:A}
				A:&=\frac{1}{\sqrt 2} \begin{pmatrix} 1 & 0 & -\ii & 0 \\ 0 & 1& 0& \ii \\
					-\ii & 0& 1& 0
					\\
					0 & \ii & 0 & 1 \end{pmatrix},
				\\
				\theta_1(z):&= \frac23 (-z)^{\frac 32} +2 \tilde{s} (-z)^{\frac 12}, \qquad z\in \mathbb{C}\setminus [0,\infty), \label{def:theta1}
				\\
				\theta_2(z):&=\frac23 z^{\frac 32} +2 \tilde{s} z^{\frac 12}, \qquad z\in \mathbb{C}\setminus (-\infty,0]. \label{def:theta2}
			\end{align}

        \item [\rm (d)] As $z \to 0$, there exists an analytic matrix-valued function $M_0(z)$ such that
		\begin{align}\label{def:M_0}
			\what M(z) = M_0(z) \begin{pmatrix}
				z^{\frac{2\nu-1}{4}} & \delta_{\nu}(z) & e^{-\nu\pi \mathrm{i}}\delta_{\nu}(z) & 0 \\
				0 & z^{-\frac{2\nu-1}{4}} & 0 & 0 \\
				0 & 0 & z^{-\frac{2\nu-1}{4}} & 0 \\
				0 & 0 & -\delta_{\nu}(z) & z^{\frac{2\nu-1}{4}}
			\end{pmatrix}\begin{cases}
				J_1^{-1} \quad & z\in \Omega_0, \\
				I \quad & z\in \Omega_1, \\
				J_2^{-1} \quad & z\in \Omega_2, \\
				J_1^{-1}J_0^{-1}J_5^{-1}J_4 \quad & z\in \Omega_3, \\
				J_1^{-1}J_0^{-1}J_5^{-1} \quad & z\in \Omega_4, \\
				J_1^{-1}J_0^{-1} \quad & z\in \Omega_5,
			\end{cases}
		\end{align}
        where the regions $\Omega_k$, $k=0,1,\ldots,5$, are shown in Figure \ref{fig:hard edge tacnode}, 
		\begin{equation}\label{def:delta}
			\delta_{\nu}(z) = \begin{cases}
				-\frac{\ln z\cdot z^{\frac{2\nu-1}{4}}}{2\pi e^{ (\nu-\frac{1}{2})\pi \ii}} ,\qquad & \nu-\frac{1}{2} \in \mathbb{Z}, \\
				-\frac{z^{\frac{2\nu-1}{4}}}{2\sin\left((\nu-\frac{1}{2})\pi \right)} ,\qquad & \nu-\frac{1}{2} \notin \mathbb{Z},
			\end{cases} 
		\end{equation}
and the principal branch cuts are taken for $\ln z$ and $z^{(2\nu-1)/4}$. 

		\end{itemize}
 \end{rhp}
It comes out that some entries of the matrix $M^{(1)}$ are related to 
the Hastings-McLeod solution $q_{\HM}$ of the inhomogeneous Painlev\'{e} II equation \eqref{eq:PII} characterized by the asymptoics
\begin{align}
	q_{\HM}(x)\sim \begin{cases}
				 \nu/x, \qquad & x \to +\infty, \\
				\sqrt{-x/2}, \qquad & x\to -\infty.
			\end{cases} 
\end{align}
More precisely, by \cite{Del13}, it follows that
\begin{align}\label{eq:M13M14}
	&M^{(1)}_{13}=-2^{-\frac{1}{3}}\ii u(2^{\frac{2}{3}}(\tilde{s}-\tau^2) + \tilde{s}^2,\\
	&M^{(1)}_{14}=2^{-\frac{1}{3}}\ii q_{\HM}(2^{\frac{2}{3}}(\tilde{s}-\tau^2),
\end{align}
where
\begin{align}
	u(x) := q_{\HM}'(x)^2 - xq_{\HM}(x)^2 - q_{\HM}(x)^4 + 2\nu q(x) 
\end{align}
is the associated Hamiltonian. Specifically, if $\tau=0$, we have 
\begin{align}\label{eq:zerotau}
    M^{(1)}_{11}=M^{(1)}_{12}=-M^{(1)}_{33}=M^{(1)}_{34}=\frac{(M^{(1)}_{13})^2-(M^{(1)}_{14})^2+\tilde{s}}{2},
\end{align}
which would be proved in Section \ref{sec:proof}.

 \begin{remark}
For later use, we rewrite large $z$ asymtotics of $\widehat{M}$ in \eqref{eq:asy:Mhat} as
			\begin{align}\label{eq:asy:Mhat2}
				\widehat{M}(z)&= \left(\widetilde{M}^{(0)} + \frac{\widetilde{M}^{(1)}}{z}+ \Boh(z^2)\right)
				\diag \left(z^{\frac{1}{4}}, z^{-\frac{1}{4}},z^{\frac{1}{4}}, z^{-\frac{1}{4}}\right)    \diag \left(\begin{pmatrix} 1 & -1 \\ 1 & 1 \end{pmatrix},\begin{pmatrix} 1 & 1 \\ -1 & 1 \end{pmatrix}\right) \nonumber\\ 
				&\quad\times\diag \left((-\sqrt{z})^{-\frac{1}{4}}, z^{-\frac{1}{8}},(-\sqrt{z})^{\frac{1}{4}}, z^{\frac{1}{8}}\right) \nonumber\\
				&\quad\times A \diag \left(e^{-\theta_1(\sqrt{z})+\tau \sqrt{z}}, e^{-\theta_2(\sqrt{z})-\tau \sqrt{z}}, e^{\theta_1(\sqrt{z})+\tau \sqrt{z}}, e^{\theta_2(\sqrt{z})-\tau \sqrt{z}}\right),
			\end{align}
			where
			\begin{equation}\label{def:M1}
				\widetilde{M}^{(0)} = \begin{pmatrix}
					1 & M^{(1)}_{11}+M^{(1)}_{12} & 0 & -M^{(1)}_{13}+M^{(1)}_{14} \\
					0 & 1 & 0 & 0 \\
					0 & M^{(1)}_{31}+M^{(1)}_{32} & 1 & -M^{(1)}_{33}+M^{(1)}_{34} \\
					0 & 0 & 0 & 1 
				\end{pmatrix}
			\end{equation}
			and
			\begin{equation}\label{def:M2}
				\widetilde{M}^{(1)} = \begin{pmatrix}
					*&*&*&*\\
					M^{(1)}_{11}-M^{(1)}_{12} &*&*&*\\
					*&*&*&*\\
					*&*&-M^{(1)}_{33}-M^{(1)}_{34}&*
				\end{pmatrix}.
			\end{equation}
 \end{remark}

\begin{remark}
 The analyticity of $M_0(z)$ in \eqref{def:M_0} can be verified directly through RH problem \ref{rhp:hard edge tac} for $\what M$. Furthermore, $M_0(0)$ takes the form
 \begin{align}\label{defi:M00}
			M_0(0)=\begin{pmatrix}
				0 & m_{12} & m_{13} &0 \\
				m_{21} & m_{22} & m_{23} & m_{24} \\
				0 & m_{32} & m_{33} & 0 \\
				m_{41} & m_{42} & m_{43} & m_{44} \\
			\end{pmatrix}
		\end{align}
with $\det M_0(0)= -\left((m_{13} m_{32} - m_{12} m_{33}) (m_{24} m_{41} - m_{21} m_{44})\right)=1$.   
\end{remark}

Let $\widetilde{M}$ be the analytic continuation of the restriction of $\widehat{M}$ in the sector $\Omega_1$  to the whole complex plane except for the negative real axis. By \cite[Equation (33)]{Del13}, we have 
	\begin{align}\label{def:hard tacnode kernel}
		& K_{\tac}(u,v) :=K_{\tac}(u,v;\tilde{s},\nu,\tau)
  \nonumber \\
  &= \frac{1}{2\pi \ii(u-v)}\begin{pmatrix}
				-1&0&1&0
			\end{pmatrix}
			\widehat{M}_{+}^{-1}(v) \widehat{M}_{+}(u)
			\begin{pmatrix}
				1&0&1&0
			\end{pmatrix}^\msf T \nonumber\\ 
  &=\frac{1}{2\pi \ii(u-v)}\begin{pmatrix}
			0&0&1&0
		\end{pmatrix}
		\widetilde{M}^{-1}(v) \widetilde{M}(u)
		\begin{pmatrix}
			1&0&0&0
		\end{pmatrix}^\msf T, \qquad u,v>0,
	\end{align}
where the second equality follows from item (b) of RH problem \ref{rhp:hard edge tac} for $\widehat M$. 



\subsection{Main results}
Let
	\begin{equation}\label{def:Fnotation}
		F(s;\gamma)=F(s;\gamma,\tilde{s},\nu,\tau):=\ln \det\left(I-\gamma  \mathcal K_{\tac}\right), 
	\end{equation}
our first result is an integral representation of $F$. The integral representation involves the Hamiltonian of a system of coupled differential equations, which reads as follows.
	\begin{align}\label{def:pq's}
		\left\{
		\begin{aligned}
			p_1'(s)&=-\frac{1}{s}\left(\sum_{j=1}^{4} p_j(s)(p_5(s) q_{4+j}(s)+p_9(s) q_{8+j}(s))\right)+\frac{\ii }{2}p_3(s),\\
			p_2'(s)&=-\frac{1}{s}\left(\sum_{j=1}^{4} p_j(s)(p_6(s) q_{4+j}(s)+p_{10}(s) q_{8+j}(s))\right)+\frac{\tau-\mathrm{i}(-M^{(1)}_{13}+M^{(1)}_{14})}{2}p_1(s)\\
			&\quad-\frac{\mathrm{i}\tilde{s}+\mathrm{i}(M^{(1)}_{11}+M^{(1)}_{12})+\mathrm{i}(-M^{(1)}_{33}+M^{(1)}_{34})}{2}p_3(s)-\frac{\mathrm{i}}{2}p_4(s),\\
			p_3'(s)&= -\frac{1}{s}\left(\sum_{j=1}^{4} p_j(s)(p_7(s) q_{4+j}(s)+p_{11}(s) q_{8+j}(s))\right),\\
			p_4'(s)&=-\frac{1}{s}\left(\sum_{j=1}^{4} p_j(s)(p_8(s) q_{4+j}(s)+p_{12}(s) q_{8+j}(s))\right)-\frac{\mathrm{i}}{2}p_1(s)-\frac{\tau+\mathrm{i}(-M^{(1)}_{13}+M^{(1)}_{14})}{2}p_3(s),\\
			p_{4+i}'(s)&=-\frac{1}{s}\left(\sum_{j=1}^{4} p_i(s) q_j(s) p_{4+j}(s)\right), \quad i=1,2,3,4, \\
			p_{8+i}'(s)&=-\frac{1}{s}\left(\sum_{j=1}^{4} p_i(s) q_j(s) p_{4+j}(s)\right), \quad i=1,2,3,4,\\
			q_1'(s)&=\frac{1}{s}\left(\sum_{j=1}^{4} q_j(s)(q_5(s) p_{4+j}(s)+q_9(s) p_{8+j}(s))\right)+\frac{\tau-\ii (-M^{(1)}_{13}+M^{(1)}_{14}) }{2}q_2(s)-\frac{\ii }{2}q_4(s),\\
			q_2'(s)&=\frac{1}{s}\left(\sum_{j=1}^{4} q_j(s)(q_6(s) p_{4+j}(s)+q_{10}(s) p_{8+j}(s))\right),\\
			q_3'(s)&=\frac{1}{s}\left(\sum_{j=1}^{4} q_j(s)(q_7(s) p_{4+j}(s)+q_{11}(s) p_{8+j}(s))\right)+\frac{\ii }{2}q_1(s)\\
			&\quad-\frac{\mathrm{i}\tilde{s}+\mathrm{i}(M^{(1)}_{11}+M^{(1)}_{12})+\mathrm{i}(-M^{(1)}_{33}+M^{(1)}_{34})}{2}q_2(s)-\frac{\tau+\mathrm{i}(-M^{(1)}_{13}+M^{(1)}_{14})}{2}q_4(s),\\
			q_4'(s)&=\frac{1}{s}\left(\sum_{j=1}^{4} q_j(s)(q_8(s) p_{4+j}(s)+q_{12}(s) p_{8+j}(s))\right)-\frac{\ii}{2}q_2(s),\\
			q_{4+i}'(s)&=\frac{1}{s}\sum_{j=1}^{4}q_i(s) p_j(s) q_{4+j}(s), \quad i=1,2,3,4,\\
			q_{8+i}'(s)&=\frac{1}{s} \sum_{j=1}^{4}q_i(s) p_j(s) q_{8+j}(s), \quad i=1,2,3,4,
		\end{aligned}\right.
	\end{align}
	where
	$$
	p_k(s)=p_k(s;\gamma, \tilde{s},\nu, \tau), \qquad q_k(s)=q_k(s;\gamma, \tilde{s},\nu, \tau), \qquad k=1,\ldots,12,
	$$
are 24 unknown functions and recall that $M^{(1)}_{ij}$ stands for the $(i,j)$-th entry of the matrix $M^{(1)}$ in \eqref{eq:asy:Mhat}. By introducing the matrix-valued functions
	\begin{equation}\label{def:A0}
		A_0=\begin{pmatrix}
			0 & \frac{\tau-\mathrm{i}(-M^{(1)}_{13}+M^{(1)}_{14})}{2} & 0 & -\frac{\mathrm{i}}{2} \\
			0 & 0 & 0 & 0 \\
			\frac{\mathrm{i}}{2} & -\frac{\mathrm{i}\tilde{s}+\mathrm{i}(M^{(1)}_{11}+M^{(1)}_{12})+\mathrm{i}(-M^{(1)}_{33}+M^{(1)}_{34})}{2} & 0 & -\frac{\tau+\mathrm{i}(-M^{(1)}_{13}+M^{(1)}_{14})}{2} \\
			0 & -\frac{\mathrm{i}}{2} & 0 & 0
		\end{pmatrix},
	\end{equation}
	\begin{equation}\label{def:A1}
		A_1(s) = \begin{pmatrix}
			q_1(s)\\q_2(s)\\q_3(s)\\q_4(s)
		\end{pmatrix}\begin{pmatrix}
			p_1(s) & p_2(s)&p_3(s)&p_4(s)
		\end{pmatrix}
	\end{equation}
	and
	\begin{align}\label{def:A2}
			A_2(s) &= \begin{pmatrix}
				q_5(s)\\q_6(s)\\q_7(s)\\q_8(s)
			\end{pmatrix}\begin{pmatrix}
				p_5(s) & p_6(s)&p_7(s)&p_8(s)
			\end{pmatrix}
			+\begin{pmatrix}
				q_9(s)\\q_{10}(s)\\q_{11}(s)\\q_{12}(s)
			\end{pmatrix}\begin{pmatrix}
				p_9(s) & p_{10}(s)&p_{11}(s)&p_{12}(s) 
			\end{pmatrix}\nonumber\\
			&\quad +\frac{2\nu-1}{4}I_4,
	\end{align}
	one can check
	\begin{equation}\label{def:H}
		H(s)=H(s;\gamma,\tilde{s}, \nu,\tau):= \begin{pmatrix}
			p_1(s) & p_2(s) & p_3(s) & p_4(s)
		\end{pmatrix}\left( A_0+\frac{A_2(s)}{s}\right)\begin{pmatrix}
			q_1(s)\\q_2(s)\\q_3(s)\\q_4(s)
		\end{pmatrix}
	\end{equation}
	is the Hamiltonian for the system of differential equations \eqref{def:pq's}, under
	the extra condition
	\begin{equation}\label{eq:sumpq}
		\sum_{k=1}^4 p_k(s)q_k(s)=0.
	\end{equation}
	That is, we have
	\begin{equation}\label{pq}
		q_k'(s)=\frac{\partial H}{\partial p_k}, \qquad p_k'(s)=-\frac{\partial H}{\partial q_k}, \qquad k=1,\ldots,12.
	\end{equation}
 
	\begin{theorem}\label{th:F-H}
		With the function $F(s;\gamma)$ defined in \eqref{def:Fnotation}, we have,
		\begin{align}\label{eq:F-H2}
			F(s;\gamma) = \int_0^s H(t; \gamma) \ud t, \qquad s \in (0, \infty),
		\end{align}
		where $H$ is the Hamiltonian \eqref{def:H} associated with a family of special solutions to the system of differential equations \eqref{def:pq's}. Moreover, $H(s)$ satisfies the following asymptotic behaviors: as $s \to 0^+$,
		\begin{equation}\label{eq:H-0}
			H(s)=\Boh(s^{\frac{2\nu-1}{2}}),
		\end{equation}
		and as $s \to +\infty$,
		\begin{align}\label{asy:H}
			H(s)=\begin{cases}
				-\frac{1}{8}s^{\frac{1}{2}} + \frac{\tilde{s}}{2} - \frac{\tilde{s}^{2}}{2\sqrt{s}} + \frac{(4\alpha^2-1)}{16s}-\frac{3}{128s}+\Boh(s^{-\frac{5}{4}}), &\qquad \gamma=1,\\
				\beta \ii  s^{-\frac{1}{4}}-\beta \ii \tilde{s} s^{-\frac{3}{4}}-\frac{7 \beta^2 }{8}s^{-1}-\frac{\beta \ii }{4} \cos  \left(2 \vartheta (s)\right)s^{-1}+\Boh({s^{-\frac{5}{4}}}),& \qquad 0 \le \gamma <1,
			\end{cases}
		\end{align}
		where
		\begin{equation}\label{def:beta}
			\beta := \frac{1}{2 \pi \ii} \ln (1-\gamma),\qquad  \alpha :=\nu-\frac{1}{2},
		\end{equation}
		and
		\begin{align}\label{def:theta}
			\vartheta (s):= \vartheta (s;\tilde{s})= \frac{2 }{3} s^{\frac 34}-2\tilde{s} s^{\frac 14} + \frac{3  \beta \ii}{4} \ln s +  \beta \ii \ln (8(1 - \frac{\tilde{s}}{\sqrt{s}})) + \arg \Gamma (1+\beta),
		\end{align}
		with $\Gamma$ being the Euler's gamma function.
	\end{theorem}
The local behavior of $H$ near the origin in \eqref{eq:H-0} ensures the integral \eqref{eq:F-H2} is well-defined.  The existence of a family of special solutions to the coupled differential equations follows from their explicit representations given in \eqref{def:q1p1}, \eqref{def:p5} and \eqref{def:q5} below. 

By inserting \eqref{asy:H} into \eqref{eq:F-H2}, we can obtain the first few terms in the asymptotic expansion of $F(s;\gamma)$ as $s \to +\infty$ except for the constant term. For $0<\gamma<1$, we are also able to determine the difficult constant term.
	\begin{theorem}\label{th:1}
		With the function $F(s;\gamma)$ defined in \eqref{def:Fnotation}, we have, as $s \to +\infty$,
		\begin{align}\label{ds-F}
			F(s;\gamma) = \begin{cases}
				-\frac{1}{12} s^{\frac{3}{2}} + \frac{\tilde{s}}{2} s - \tilde{s}^2 \sqrt{s}+\frac{(4\alpha^2-1)\ln{s}}{16} - \frac{3\ln{s}}{128} + C + \Boh(s^{-\frac{1}{4}}), & \gamma=1,\\
				\frac{4}{3}\beta \ii  s^{\frac{3}{4}}-4\beta \ii \tilde{s} s^{\frac{1}{4}}-\frac{3}{4}\beta^2 \ln s+(D_{\nu,1}+\frac{13}{24}-\ln 8)\beta^2+D_{\nu,2}\beta^4\\
   \qquad +\ln G(1+\beta)G(1-\beta)+\Boh(s^{-\frac{1}{4}})	, & 0 \le \gamma <1,
			\end{cases}
		\end{align}
 uniformly for  $\tilde{s} \in \mathbb{R}$, where $C$ is an undetermined constant independent of $s$,  $\beta$ is given in \eqref{def:beta}, the constants $D_{\nu,1}$ and $D_{\nu,2}$ are given in \eqref{def:D1} and \eqref{def:D2}, respectively, which are independent of $\beta$ and $\tau$, 
 and $G(z)$ is the Barnes G-function.
	\end{theorem}
If $\gamma \to 0^+$, we have that $\beta=0$ and $G(1)=0$. It is then easily seen from \eqref{ds-F} that $F(s;0)=\Boh(s^{-1/4})$. This is compatible with the fact that $F(s;0)=0$. By Forrester–Chen–Eriksen–Tracy proposition \cite{CET95,For1993} about the large gap behavior of the Fredholm determinant, it is expected that  $F(s;1)=\Boh(s^{3/2})$, since the limiting mean density of the non-intersecting squared Bessel paths near the hard edge tacnode behaves like $\Boh(x^{-1/4})$. Our Theorem \ref{th:1} then confirms this proposition for the hard edge tacnode determinant. Also, the constants $D_{\nu,i}$, $i=1,2$, in \eqref{ds-F-1} are expressed in terms of $(1,3),(1,4)$ and $(3,1)$ entries of the matrix $M^{(1)}$ in 
\eqref{eq:asy:Mhat}. By \eqref{eq:M13M14}, some of them are related to the Hastings-McLeod solution of the inhomogeneous Painlev\'{e} II equation \eqref{eq:PII}. 

    
As an application of the large $s$ asymptotics of $F(s;\gamma)$, $0<\gamma<1$, we consider the counting statistics of the hard edge tacnode process. More precisely, let $N(s)$ be the random variable that counts the number of points in the hard edge tacnode process falling into the interval $(0,s)$, $s\geq 0$. It is well-known that the following generating function
	\begin{equation}
		\mathbb{E} \left(e^{-2\pi x N(s)}\right) = \sum_{k=0}^{\infty} \mathbb{P} (N(s)=k)e^{-2\pi x k}, \qquad x \ge 0,
	\end{equation}
is equal to the deformed Fredholm determinant $\det \left(I - (1-e^{-2\pi x})\mathcal K_{\tac}\right)$. This, together with \eqref{ds-F}, allows us to establish various asymptotic statistical properties of $N$. We refer to \cite{Charlier21,CC21,CM,DXZ22,DXZ23,Sosh,YZ24b} for relevant results about the Sine, Airy, (hard edge) Pearcey, Pearcey and the soft tacnode point determinantal processes.

	\begin{corollary}\label{coro1}
		As $s \to +\infty$, we have
		\begin{align}
			\mathbb{E}(N(s)) &=\mu (s) + \Boh \left(\frac{\ln s}{s^{1/4}}\right),\label{def:EN}\\
			\Var (N(s)) & = \sigma (s)^2 + \frac{1 + \gamma_E-\frac{13}{24} + \ln 8-D_{\nu,1}}{2 \pi^2} + \Boh \left(\frac{(\ln s)^2}{s^{1/4}}\right),\label{def:VarN}
		\end{align}
		where $\gamma_E=-\Gamma'(1)\approx 0.57721$ is Euler’s constant, $D_{\nu,1}$ is defined in \eqref{def:D1}, and
		\begin{align}\label{def:mu-sigma}
			\mu (s) = \frac{2}{3 \pi} s^{\frac 34} - \frac{2\tilde{s}}{\pi} s^{\frac 14}, \qquad \sigma (s)^2 = \frac{3}{8 \pi^2} \ln s.
		\end{align}
		Furthermore, the random variable $\frac{N(s) - \mu (s)}{\sqrt{\sigma (s)^2}}$ converges in distribution to the normal law $\mathcal{N} (0,1)$ as $s \to +\infty$, and for any $\epsilon > 0$, we have
		\begin{align}\label{ub}
			\lim_{a \to \infty} \mathbb{P} \left(\sup_{s>a} \left|\frac{N(s) - \mu(s)}{\ln s}\right| \le \frac{3 }{4 \pi} + \epsilon\right)=1.
		\end{align}
	\end{corollary}
	The probabilistic bound \eqref{ub} particularly implies that, for large positive $s$, the
	counting function of the tacnode process lies in the interval $(\mu(s)-(3/(4\pi)+\epsilon) \ln s, \mu(s)+
	(3/(4\pi)+\epsilon) \ln s)$ with high probability.

	\paragraph{Organization of the rest of the paper}
	The rest of this paper is devoted to the proofs of our main results. The idea is to relate various derivatives of $F$ to a $4 \times 4$ RH problem under the general framework \cite{BD02, DIKZ}. The precise relationship in our case is discussed in Section \ref{sec:RHP}, where we connect $\ud F/\ud s$ to a $4 \times 4$ RH problem for $X$ with constant jumps. We then derive a lax pair for $X$ in Section \ref{sec:Lax}, and some useful differential identities for the Hamiltonian will also be included for later calculation. We perform a Deift-Zhou steepest descent analysis \cite{DIZ97} on the RH problem for $X$ as $s \to +\infty$ in Sections \ref{sec:AsyX1} and \ref{sec:AsyXgamma} for the cases $\gamma=1$ and $0 \leq \gamma < 1$, respectively, and deal with the small positive $s$ case in Section \ref{sec:AsyX0}. After computing the asymptotics of $\sum p_k(s) \frac{\partial q_k(s)}{\partial \beta}$, where $p_k$, $q_k$ satisfy differential equations \eqref{def:pq's} and \eqref{eq:sumpq} in Section \ref{sec:pq}, we finally present the proofs of our main results in Section \ref{sec:proof}.
	
	
	\section{A RH formulation } \label{sec:RHP} 
 We intend to establish a relation between $\ud F / \ud s$ and a RH problem with constant jumps. To proceed, we note that
	\begin{align}
		\frac{\ud}{\ud s}F(s;\gamma)= -\textrm{tr}\left((I-\gamma \mathcal{K}_{\tac})^{-1} \gamma
		\frac{\ud}{\ud s}\mathcal{K}_{\tac}\right) = -R(s,s), \label{eq:derivatives}
	\end{align}
	where $R(u,v)$ stands for the kernel of the resolvent operator and F is defined in \eqref{def:Fnotation}.

	By \eqref{def:hard tacnode kernel}, one readily sees that
	\begin{equation}\label{eq:tildeKdef}
		\gamma K_\tac (x,y) = \frac{\vec{f}(x)^{\msf T}\vec{h}(y)}{x-y},
	\end{equation}
	where
	\begin{equation}\label{def:fh}
		\vec{f}(x)=\begin{pmatrix}
			f_1
			\\
			f_2
			\\
			f_3
			\\
			f_4
		\end{pmatrix}:= \widetilde M (x)
		\begin{pmatrix}
			1
			\\
			0
			\\
			0
			\\
			0
		\end{pmatrix}, \qquad
		\vec{h}(y)=\begin{pmatrix}
			h_1
			\\
			h_2
			\\
			h_3
			\\
			h_4
		\end{pmatrix}
		:=
		\frac{\gamma}{2 \pi \ii}
		\widetilde M(y)^{-\msf T}
		\begin{pmatrix}
			0
			\\
			0
			\\
			1
			\\
			0
		\end{pmatrix}.
	\end{equation}
	This integrable structure of kernel $K_\tac$  in the sense of Its et al. \cite{IIKS90} implies that the resolvent kernel $R(u,v)$ is integrable as well; cf. \cite{DIZ97,IIKS90}. Indeed, by setting
	\begin{equation}\label{def:FH}
		\vec{F}(u)=
		\begin{pmatrix}
			F_1 \\
			F_2 \\
			F_3 \\
			F_4
		\end{pmatrix}:=\left(I-\gamma \mathcal{K}_{\tac}\right)^{-1}\vec{f}, \qquad \vec{H}(v)=\begin{pmatrix}
			H_1 \\
			H_2 \\
			H_3 \\
			H_4
		\end{pmatrix}
		:=\left(I-\gamma \mathcal{K}_{\tac}\right)^{-1}\vec{h},
	\end{equation}
	we have
	\begin{equation}\label{eq:resolventexpli}
		R(u,v)=\frac{\vec{F}(u)^{\msf T}\vec{H}(v)}{u-v}.
	\end{equation}
	Moreover, the functions $\vec{F}(u)$ and $\vec{H}(v)$ are closely related to the following RH problem.
	\begin{rhp}\label{rhp:Y}
		\hfill
		\begin{enumerate}
			\item[\rm (a)] $Y(z)$ is a $4 \times 4$ matrix-valued function defined and analytic in $\mathbb{C} \setminus [0, s]$.
			
			\item[\rm (b)] For $x \in(0, s)$, we have
			\begin{equation}
				Y_{+}(x)=Y_{-}(x)\left(I-2 \pi \mathrm{i} \mathbf{f}(x) \mathbf{h}(x)^{\msf T}\right),
			\end{equation}
			where the functions $\mathbf{f}$ and $\mathbf{h}$ are defined above .
			
			\item[\rm (c)] As $z \to \infty$, we have
			\begin{equation}\label{eq:Y-infty}
				Y(z)=I+\frac{Y^{(1)}}{z}+\mathcal{O}\left(z^{-2}\right),
			\end{equation}
			where the function $Y^{(1)}$ is independent of $z$.
			\item[\rm (d)] As $z \to s$, we have 
			\begin{equation}
				Y(z)=\mathcal{O}(\ln (z-s)).
			\end{equation}
			\item[\rm (e)] As $z \to 0$, we have 
			\begin{equation}
				Y(z) = \begin{cases}
					O(\ln z), \quad & \nu \geq \frac{1}{2},\\
					O(z^{\frac{2\nu-1}{2}}), \quad &-\frac{1}{2}< \nu \leq \frac{1}{2}.
				\end{cases}
			\end{equation}
		\end{enumerate}
	\end{rhp}
	By \cite{DIZ97}, it follows that
	\begin{equation}\label{eq:Yexpli}
		Y(z)=I-\int_{0}^s \frac{\mathbf{F}(w) \mathbf{h}(w)^{\msf T}}{w-z} \ud w
	\end{equation}
	and
	\begin{equation}\label{def:FH2}
		\mathbf{F}(z)=Y(z) \mathbf{f}(z), \quad \mathbf{H}(z)=Y(z)^{-\msf{T}} \mathbf{h}(z).
	\end{equation}

	We now make an undressing transformation to arrive at an RH problem that is related to $\ud F/ \ud s$ with the aid of the RH problems for $M$ and $Y$. We start with definitions
	\begin{equation}\label{def:Gammajs}
		\begin{aligned}
			&\Gamma_0^{(s)}:=(s,+\infty), &&  \Gamma_1^{(s)}:=s+e^{\varphi \ii }(0,+\infty),  && \Gamma_2^{(s)}:=e^{-\varphi \ii}(-\infty,0),\\
			&\Gamma_3^{(s)}:= (-\infty,0),&& \Gamma_4^{(s)}:=e^{\varphi \ii}(-\infty,0), &&\Gamma_5^{(s)}:=s+e^{-\varphi \ii}(0,+\infty), \quad 0<\varphi<\frac{\pi}{3}.
		\end{aligned}
	\end{equation}
	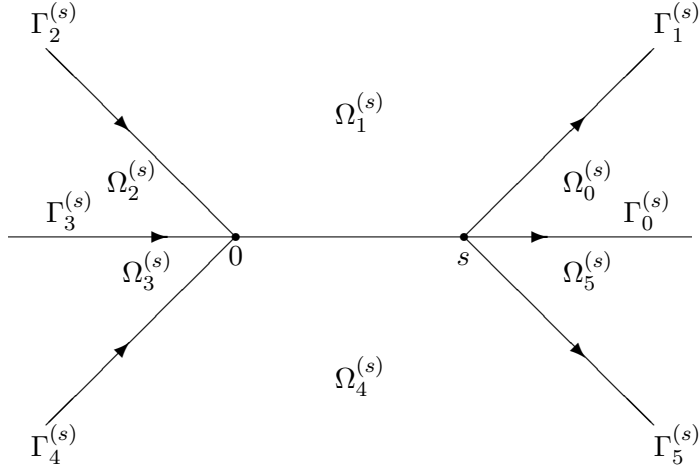
\begin{figure}[t]
		\begin{center}
			\setlength{\unitlength}{1truemm}
			\begin{picture}(100,70)(-5,2)
				\put(25,40){\line(-1,0){30}}
				\put(55,40){\line(1,0){30}}
				
				\put(25,40){\line(1,0){30}}
				
				\put(25,40){\line(-1,-1){25}}
				\put(25,40){\line(-1,1){25}}
				
				\put(55,40){\line(1,1){25}}
				\put(55,40){\line(1,-1){25}}
				
				\put(15,40){\thicklines\vector(1,0){1}}
				\put(65,40){\thicklines\vector(1,0){1}}

				\put(10,55){\thicklines\vector(1,-1){1}}
				\put(10,25){\thicklines\vector(1,1){1}}
				\put(70,25){\thicklines\vector(1,-1){1}}
				\put(70,55){\thicklines\vector(1,1){1}}
				
				\put(-2,11){$\Gamma_4^{(s)}$}

				\put(-2,67){$\Gamma_2^{(s)}$}
				\put(0,42){$\Gamma_3^{(s)}$}
				\put(80,11){$\Gamma_5^{(s)}$}
				\put(80,67){$\Gamma_1^{(s)}$}
				\put(76,42){$\Gamma_0^{(s)}$}

				\put(8,46){$\Omega_2^{(s)}$}
				\put(10,34){$\Omega_3^{(s)}$}
				\put(68,46){$\Omega_0^{(s)}$}
				\put(68,34){$\Omega_5^{(s)}$}
				\put(38,55){$\Omega_1^{(s)}$}
				\put(38,20){$\Omega_4^{(s)}$}
				
				\put(25,40){\thicklines\circle*{1}}
				\put(55,40){\thicklines\circle*{1}}
				
				\put(24,36.3){$0$}
				\put(54,36.3){$s$}
				
			\end{picture}
			\caption{Regions $\Omega_j^{(s)}$,  and the contours $\Gamma_{j}^{(s)}$, $j=0,1,\ldots,5$ in the RH problem for $X$.}
			\label{fig:X}
		\end{center}
	\end{figure}
Clearly, the rays $\Gamma_k^{(s)}$, $k=1,2,4,5$, and $\mathbb{R}$ divide the whole complex plane into 6 regions $\Omega_j^{(s)}$, $j=0,\ldots,5$; see Figure \ref{fig:X} for an illustration. The transformation is defined by
	\begin{equation}\label{eq:YtoX}
		X(z)=X\left(z ;  s,\gamma,\tilde{s},\nu ,\tau\right) = \begin{cases}
			Y(z) \widetilde{M}(z), \quad & z \in \Omega_1^{(s)},\\
			Y(z) \widehat{M}(z), \quad &  \textrm{elsewhere}, \\
			Y(z) \widetilde{M}(z)\begin{pmatrix}
				1 & 0 & -1 & 0 \\
				0 & 1 & 0 & 0 \\
				0 & 0 & 1 & 0 \\
				0 & 0 & 0 & 1
			\end{pmatrix}, \quad & z \in \Omega_4^{(s)}.
		\end{cases}
	\end{equation}
	Then $X$ satisfies the following RH problem.
	\begin{rhp}\label{rhp:X}
		\hfill
		\begin{enumerate}
			\item[\rm (a)] $X(z)$ is defined and analytic in $\mathbb{C} \setminus  \Gamma_X$, where
			\begin{equation}\label{def:gammaX}
				\Gamma_X:=\cup_{j=0}^5 \Gamma_j^{(s)} \cup[0, s]
			\end{equation}
			with the rays $\Gamma_j^{(s)}, j=0,1, \ldots, 5$, defined in \eqref{def:Gammajs}; see Figure \ref{fig:X} for the orientations of $\Gamma_X$.
			
			\item[\rm (b)] For $z \in \Gamma_X, X$ satisfies the jump condition
			\begin{align}\label{eq:X-jump}
				X_{+}(z)=X_{-}(z) J_X(z)
			\end{align}
			where
			\begin{equation}\label{def:JX}
				J_X(z):= \begin{cases}\begin{pmatrix}
						0 & 0 & 1 & 0 \\
						0 & 1 & 0 & 0 \\
						-1 & 0 & 0 & 0 \\
						0 & 0 & 0 & 1
					\end{pmatrix}, & \quad z \in \Gamma_0^{(s)}, \\
					
					I+E_{3,1}, & \quad z \in \Gamma_1^{(s)},\\
					
					I+e^{\nu\pi \mathrm{i}} E_{2,1}-e^{\nu\pi \mathrm{i}} E_{3,4}, & \quad z \in \Gamma_2^{(s)}, \\
					\begin{pmatrix}
						0 & -\mathrm{i}& 0 & 0 \\
						-\mathrm{i} & 0 & 0 & 0 \\
						0 & 0 & 0 & \mathrm{i}\\
						0 & 0 & \mathrm{i} & 0
					\end{pmatrix}, & \quad z \in \Gamma_3^{(s)},\\
					I-e^{-\nu\pi \mathrm{i}} E_{2,1}+e^{-\nu\pi \mathrm{i}} E_{3,4}, & \quad z \in \Gamma_4^{(s)}, \\
					
					I+E_{3,1}, & \quad z \in \Gamma_5^{(s)},\\
					
					I+(1-\gamma) E_{1,3}, & \quad z\in (0,s).
				\end{cases}
			\end{equation}

			\item[\rm (c)] As $z \to \infty$ with $z \in \mathbb{C} \setminus  \Gamma_X$, we have
				\begin{align}\label{eq:asyX}
					X(z)
					&= \left(X_{0}+\frac{X^{(1)}}{z}+\Boh(z^{-2})\right)\diag \left(z^{\frac{1}{4}}, z^{-\frac{1}{4}},z^{\frac{1}{4}}, z^{-\frac{1}{4}}\right)    \diag \left(\begin{pmatrix} 1 & -1 \\ 1 & 1 \end{pmatrix},\begin{pmatrix} 1 & 1 \\ -1 & 1 \end{pmatrix}\right)\nonumber\\ 
					&\quad \times \diag \left((-\sqrt{z})^{-\frac{1}{4}}, z^{-\frac{1}{8}},(-\sqrt{z})^{\frac{1}{4}}, z^{\frac{1}{8}}\right) 
					A \nonumber\\
					&\quad \times \diag \left(e^{-\theta_1(\sqrt{z})+\tau \sqrt{z}}, e^{-\theta_2(\sqrt{z})-\tau \sqrt{z}}, e^{\theta_1(\sqrt{z})+\tau \sqrt{z}}, e^{\theta_2(\sqrt{z})-\tau \sqrt{z}}\right),
				\end{align}
			where $A$, $\theta_1$ and $\theta_2$ are defined in \eqref{def:A}--\eqref{def:theta2} respectively, $X_0$ and $X^{(1)}$ are independent of $z$ with
			\begin{equation}\label{def:X0}
				X_0=\widetilde{M}^{(0)},
			\end{equation}
			and
			\begin{equation}\label{def:X1}
				X^{(1)}=Y^{(1)} \widetilde{M}^{(0)} +  \widetilde{M}^{(1)},
			\end{equation}
			with $\widetilde{M}^{(0)}$ and $ \widetilde{M}^{(1)}$  given in \eqref{def:M1} and \eqref{def:M2}, 
             respectively.	
			\item[\rm (d)] As $z \to s$, we have
			\begin{align}\label{eq:X-near-s} 
					X(z)= X_R(z)\begin{pmatrix}
						1 & 0 & -\frac{\gamma}{2 \pi \mathrm{i}} \ln (z-s) & 0 \\
						0 & 1 & 0 & 0 \\
						0 & 0 & 1 & 0 \\
						0 & 0 & 0 & 1
					\end{pmatrix}  \left\{\begin{aligned}
						&I, &\quad z \in \Omega_1^{(s)}, \\
						&\begin{pmatrix}
							1 & 0 & -1 & 0 \\
							0 & 1 & 0 & 0 \\
							0 & 0 & 1 & 0 \\
							0 & 0 & 0 & 1
						\end{pmatrix}, &\quad z \in \Omega_4^{(s)},
					\end{aligned}\right.
				\end{align}

			where the principal branch is taken for $\ln (z-s)$, and $X_R(z)$ is analytic at $z=s$ satisfying 
			\begin{equation}\label{eq: X-expand-s}
				X_R(z)=X_{R, 0}(s)\left(I+X_{R, 1}(s)(z-s)+\mathcal{O}\left((z-s)^2\right)\right), \quad z \to s,
			\end{equation}
			for some functions $X_{R, 0}(s)$ and $X_{R, 1}(s)$ depending on the parameters $\gamma, \tilde{s}, \nu$ and $\tau$.
			\item[\rm (e)] As $z \to 0$, we have
			\begin{align}\label{eq:X-near-0}
					X(z)&=  X_L(z)\begin{pmatrix}
						z^{\frac{2\nu-1}{4}} & \delta_{\nu}(z) & (1-\gamma ) e^{-\nu\pi \mathrm{i}}\delta_{\nu}(z) & 0 \\
						0 & z^{-\frac{2\nu-1}{4}} & 0 & 0 \\
						0 & 0 & z^{-\frac{2\nu-1}{4}} & 0 \\
						0 & 0 & -\delta_{\nu}(z) & z^{\frac{2\nu-1}{4}}
					\end{pmatrix} \nonumber\\
					& \quad \times\left\{\begin{aligned}
						&I, &\quad z \in \Omega_1^{(s)}, \\
						&\begin{pmatrix}
							1 & 0 & \gamma-1 & 0 \\
							0 & 1 & 0 & 0 \\
							0 & 0 & 1 & 0 \\
							0 & 0 & 0 & 1
						\end{pmatrix},& \quad z \in \Omega_4^{(s)} ,
					\end{aligned}\right.
			\end{align}
			where $\delta_{\nu}(z)$ is given in \eqref{def:delta}. $X_L(z)$ is analytic at $z=0$ satisfying 
			\begin{equation}\label{eq:X-expand-0}
				X_L(z)=X_{L, 0}(s)\left(I+X_{L, 1}(s)z+\Boh(z^2)\right), \quad z \to 0,
			\end{equation}	
			for some functions $X_{L, 0}(s)$ and $X_{L, 1}(s)$ depending on the parameters $\gamma,\tilde{s},\nu$ and $\tau$.
		\end{enumerate}
	\end{rhp}
\begin{proof}
The jump condition is easy to check,
we verify the local behavior of $X(z)$ near $z=0$ given in \eqref{eq:X-near-0}, and it suffices to show that it satisfies the jump condition of $X$. For $x<0$, it follows from \eqref{def:delta}, \eqref{eq:X-jump}, \eqref{def:JX} and \eqref{eq:X-near-0} that 
\begin{align}
    &X_{-}(z)^{-1}X_{+}(z) \nonumber\\
    &= \begin{pmatrix}
							1 & -e^{\nu\pi \mathrm{i}} & 0 & 0 \\
							0 & 1 & 0 & 0 \\
							0 & 0 & 1 & e^{\nu\pi \mathrm{i}} \\
							0 & 0 & 0 & 1
						\end{pmatrix} \begin{pmatrix}
						(e^{-\pi \ii}\lvert z \rvert)^{-\frac{2\nu-1}{4}} & \delta_{\nu,-}(z) & (1-\gamma ) e^{-\nu\pi \mathrm{i}}\delta_{\nu,-}(z) & 0 \\
						0 & (e^{-\pi \ii}\lvert z \rvert)^{\frac{2\nu-1}{4}} & 0 & 0 \\
						0 & 0 & (e^{-\pi \ii}\lvert z \rvert)^{\frac{2\nu-1}{4}} & 0 \\
						0 & 0 & -\delta_{\nu,-}(z) &(e^{-\pi \ii}\lvert z \rvert)^{-\frac{2\nu-1}{4}}
					\end{pmatrix} \nonumber\\
    &\quad \times X_{L,-}(z)^{-1} X_{L,+}(z) \nonumber\\
    &\quad \times\begin{pmatrix}
						(e^{\pi \ii}\lvert z \rvert)^{\frac{2\nu-1}{4}} & \delta_{\nu,+}(z) & (1-\gamma ) e^{-\nu\pi \mathrm{i}}\delta_{\nu,+}(z) & 0 \\
						0 & (e^{\pi \ii}\lvert z \rvert)^{-\frac{2\nu-1}{4}} & 0 & 0 \\
						0 & 0 & (e^{\pi \ii}\lvert z \rvert)^{-\frac{2\nu-1}{4}} & 0 \\
						0 & 0 & -\delta_{\nu,+}(z) & (e^{\pi \ii}\lvert z \rvert)^{\frac{2\nu-1}{4}}
					\end{pmatrix} \begin{pmatrix}
							1 & -e^{\nu\pi \mathrm{i}} & 0 & 0 \\
							0 & 1 & 0 & 0 \\
							0 & 0 & 1 & e^{\nu\pi \mathrm{i}} \\
							0 & 0 & 0 & 1
						\end{pmatrix} \nonumber\\
    &=\begin{pmatrix}
						0 & -\mathrm{i}& 0 & 0 \\
						-\mathrm{i} & 0 & 0 & 0 \\
						0 & 0 & 0 & \mathrm{i}\\
						0 & 0 & \mathrm{i} & 0
					\end{pmatrix},
\end{align}
as required.
For $x\in (0,s)$, it could be checked directly by \eqref{eq:X-near-0}.

Finally, the local behavior of $X(z)$ at $z=s$ can be verified in a manner similar to that at origin, we omit the details here.
\end{proof}
	\begin{proposition}\label{prop:derivativeandX}
		With $F$ defined in \eqref{def:Fnotation}, we have
		\begin{align}
			\frac{\ud}{\ud s} F(s;\gamma,\tilde{s},\nu, \tau) =-\frac{\gamma}{2 \pi \ii} \lim_{z \to s} \left(X(z)^{-1}X'(z)\right)_{31} ,
			\label{eq:derivativeinsX-2}
		\end{align}
		where the limit is taken from $\Omega_1^{(s)}$ and
		\begin{equation}\label{eq:derivativeins-tau}
			\frac{\ud}{\ud \tau} F(s;\gamma,\tilde{s},\nu, \tau) = X^{(1)}_{43}-X^{(1)}_{21}-\widetilde{M}^{(1)}_{43}+\widetilde{M}^{(1)}_{21},		\end{equation}
		where $X^{(1)}$ and $\widetilde{M}^{(1)}$ are given in \eqref{def:X1} and \eqref{def:M2}.
	\end{proposition}
	To prove Proposition \ref{prop:derivativeandX}, we need the following lemma.
	\begin{lemma}\label{prop:M}
		Let $\widehat{M}$ be the unique solution to the RH problem \ref{rhp:hard edge tac}. We have
		\begin{align}
			\frac{\partial \widehat{M}}{\partial \tau} &= \left( \begin{pmatrix}
				0 & z & 0 & 0\\
				0 & 0 & 0 & 0\\
				0 & 0 & 0 & -z\\
				0 & 0 & 0 & 0
			\end{pmatrix}+ \widetilde{M}^{(0)}\begin{pmatrix}
				0 & 0 & 0 & 0\\
				1 & 0 & 0 & 0\\
				0 & 0 & 0 & 0\\
				0 & 0 & -1 & 0
			\end{pmatrix}\left(\widetilde{M}^{(0)}\right)^{-1}+\frac{\partial \widetilde{M}^{(0)}}{\partial \tau}\left(\widetilde{M}^{(0)}\right)^{-1}\right.\nonumber\\
			&\left.\quad+\left[\widetilde{M}^{(1)}\left(\widetilde{M}^{(0)}\right)^{-1}, \begin{pmatrix}
				0 & 1 & 0 & 0\\
				0 & 0 & 0 & 0\\
				0 & 0 & 0 & -1\\
				0 & 0 & 0 & 0
			\end{pmatrix}\right]\right)\widehat{M}, \label{eq:lax:M2}
		\end{align}
		where $[A,B]$ denotes the commutator of two matrices, i.e., $[A,B]=AB-BA$.
	\end{lemma}
	\begin{proof}
		Since the RH problem for $\widehat{M}$ has constant jumps, we obtain $(\partial \widehat{M} / \partial \tau) M^{-1}$ is analytic for $z \in \mathbb{C}\setminus \{0\}$. It's also analytic at $z=0$, due to the fact that the matrix $\begin{pmatrix}
			z^{\frac{2\nu-1}{4}} & \delta_{\nu}(z) & e^{-\nu \pi \mathrm{i}}\delta_{\nu}(z) & 0 \\
			0 & z^{-\frac{2\nu-1}{4}} & 0 & 0 \\
			0 & 0 & z^{-\frac{2\nu-1}{4}} & 0 \\
			0 & 0 & -\delta_{\nu}(z) & z^{\frac{2\nu-1}{4}}
		\end{pmatrix}$ in \eqref{def:M_0} is independent of $\tau$.
		As $z \to \infty$, we find from \eqref{eq:asy:Mhat2} that
		\begin{align}
			\frac{\partial \widehat{M}}{\partial \tau} \widehat{M}^{-1} &= \left(\widetilde{M}^{(0)} + \frac{\widetilde{M}^{(1)}}{z}+ \Boh(z^{-2})\right)
			\diag \left(z^{\frac{1}{4}}, z^{-\frac{1}{4}},z^{\frac{1}{4}}, z^{-\frac{1}{4}}\right)    \diag \left(\begin{pmatrix} 1 & -1 \\ 1 & 1 \end{pmatrix},\begin{pmatrix} 1 & 1 \\ -1 & 1 \end{pmatrix}\right) \nonumber\\ 
			&\quad\times \diag \left((-\sqrt{z})^{-\frac{1}{4}}, z^{-\frac{1}{8}},(-\sqrt{z})^{\frac{1}{4}}, z^{\frac{1}{8}}\right)  A \diag \left(\sqrt{z}, -\sqrt{z},\sqrt{z},-\sqrt{z} \right),
			\nonumber  \\
			& \quad \times A^{-1} \diag \left((-\sqrt{z})^{\frac{1}{4}}, z^{\frac{1}{8}},(-\sqrt{z})^{-\frac{1}{4}}, z^{-\frac{1}{8}}\right) \diag \left(\begin{pmatrix} 1 & 1 \\ -1 & 1 \end{pmatrix},\begin{pmatrix} 1 & -1 \\ 1 & 1 \end{pmatrix}\right) \nonumber \\
			&\quad\times\diag \left(z^{-\frac{1}{4}}, z^{\frac{1}{4}},z^{-\frac{1}{4}}, z^{\frac{1}{4}}\right)\left(\left(\widetilde{M}^{(0)}\right)^{-1} - \frac{\left(\widetilde{M}^{(0)}\right)^{-1}\widetilde{M}^{(1)}\left(\widetilde{M}^{(0)}\right)^{-1}}{z}+ \Boh(z^{-2})\right) \nonumber\\
			&\quad +\frac{\partial \widetilde{M}^{(0)}}{\partial \tau}\left(\widetilde{M}^{(0)}\right)^{-1} +\Boh(z^{-1}) \nonumber\\
			&=  \begin{pmatrix}
				0 & z & 0 & 0\\
				0 & 0 & 0 & 0\\
				0 & 0 & 0 & -z\\
				0 & 0 & 0 & 0
			\end{pmatrix}+ \widetilde{M}^{(0)}\begin{pmatrix}
				0 & 0 & 0 & 0\\
				1 & 0 & 0 & 0\\
				0 & 0 & 0 & 0\\
				0 & 0 & -1 & 0
			\end{pmatrix}\left(\widetilde{M}^{(0)}\right)^{-1}+\frac{\partial \widetilde{M}^{(0)}}{\partial \tau}\left(\widetilde{M}^{(0)}\right)^{-1}\nonumber\\
			&\quad+\left[\widetilde{M}^{(1)}\left(\widetilde{M}^{(0)}\right)^{-1}, \begin{pmatrix}
				0 & 1 & 0 & 0\\
				0 & 0 & 0 & 0\\
				0 & 0 & 0 & -1\\
				0 & 0 & 0 & 0
			\end{pmatrix}\right] + \Boh(z^{-1}).
		\end{align}
		Keeping only the polynomial terms in $z$, we obtain \eqref{eq:lax:M2}.
		
		This finishes the proof of Lemma \ref{prop:M}.
	\end{proof}
 
\paragraph{Proof of Proposition \ref{prop:derivativeandX}}
		For $z\in \Omega_1^{(s)}$, we see from \eqref{def:fh}, \eqref{def:FH2} and \eqref{eq:YtoX} that
		\begin{equation}\label{eq:FHinX}
			\vec{F}(z)=Y(z)\vec{f}(z)=Y(z)\widetilde M(z)
			\begin{pmatrix}
				1
				\\
				0
				\\
				0
				\\
				0
			\end{pmatrix}= X(z)
			\begin{pmatrix}
				1
				\\
				0
				\\
				0
				\\
				0
			\end{pmatrix}
		\end{equation}
		and
		\begin{align}
			\vec{H}(z)&=Y(z)^{-\msf T}\vec{h}(z)=  X(z)^{-\msf T}\widetilde M(z)^{\msf T} \cdot \frac{\gamma }{2 \pi \ii} \widetilde M(z)^{-\msf T}
			\begin{pmatrix}
				0
				\\
				0
				\\
				1
				\\
				0
			\end{pmatrix}
			=\frac{\gamma }{2 \pi \ii}  X(z)^{-\msf T}
			\begin{pmatrix}
				0
				\\
				0
				\\
				1
				\\
				0
			\end{pmatrix}.
		\end{align}
		Combining the above formulas and \eqref{eq:resolventexpli}, we obtain
		\begin{equation} \label{eq: R-Xij}
			R(z,z) = \frac{\gamma}{2 \pi \ii}  \left(X(z)^{-1}X'(z)\right)_{31}, \qquad  z\in \Omega_1^{(s)}.
		\end{equation}
		This, together with \eqref{eq:derivatives}, gives us \eqref{eq:derivativeinsX-2}.
		To show \eqref{eq:derivativeins-tau}, we note from \eqref{def:fh} and \eqref{eq:lax:M2} that
		\begin{align}
			\frac{\partial \vec{f}}{\partial \tau}(x) &=  \left( \begin{pmatrix}
				0 & x & 0 & 0\\
				0 & 0 & 0 & 0\\
				0 & 0 & 0 & -x\\
				0 & 0 & 0 & 0
			\end{pmatrix}+ \widetilde{M}^{(0)}\begin{pmatrix}
				0 & 0 & 0 & 0\\
				1 & 0 & 0 & 0\\
				0 & 0 & 0 & 0\\
				0 & 0 & -1 & 0
			\end{pmatrix}\left(\widetilde{M}^{(0)}\right)^{-1}+\frac{\partial \widetilde{M}^{(0)}}{\partial \tau}\left(\widetilde{M}^{(0)}\right)^{-1}\right.\nonumber\\
			&\left.\quad+\left[\widetilde{M}^{(1)}\left(\widetilde{M}^{(0)}\right)^{-1}, \begin{pmatrix}
				0 & 1 & 0 & 0\\
				0 & 0 & 0 & 0\\
				0 & 0 & 0 & -1\\
				0 & 0 & 0 & 0
			\end{pmatrix}\right]\right)\vec{f}(x),\\
			\frac{\partial \vec{h}}{\partial \tau}(y) &= - \left( \begin{pmatrix}
				0 & y & 0 & 0\\
				0 & 0 & 0 & 0\\
				0 & 0 & 0 & -y\\
				0 & 0 & 0 & 0
			\end{pmatrix}+ \widetilde{M}^{(0)}\begin{pmatrix}
				0 & 0 & 0 & 0\\
				1 & 0 & 0 & 0\\
				0 & 0 & 0 & 0\\
				0 & 0 & -1 & 0
			\end{pmatrix}\left(\widetilde{M}^{(0)}\right)^{-1}+\frac{\partial \widetilde{M}^{(0)}}{\partial \tau}\left(\widetilde{M}^{(0)}\right)^{-1}\right.\nonumber\\
			&\left.\quad+\left[\widetilde{M}^{(1)}\left(\widetilde{M}^{(0)}\right)^{-1}, \begin{pmatrix}
				0 & 1 & 0 & 0\\
				0 & 0 & 0 & 0\\
				0 & 0 & 0 & -1\\
				0 & 0 & 0 & 0
			\end{pmatrix}\right]\right)^{\msf T}\vec{h}(y).
		\end{align}
		The above equations, together with \eqref{eq:tildeKdef}, imply that
		\begin{align}
			\frac{\partial}{\partial \tau} \left(\gamma K_{\tac} (x,y)\right)& = \frac{\frac{\partial \vec{f}^{\msf T}}{\partial \tau}(x)\vec{h}(y)+ \vec{f}(x)^{\msf T}\frac{\partial \vec{h}}{\partial \tau}(y)}{x-y} \nonumber\\
			&=\vec{f}(x)^{\msf T}\left( E_{2,1}-E_{4,3}\right)\vec{h}(y) =  f_2(x)h_1(y)-f_4(x)h_3(y).
		\end{align}
		Thus, it is readily seen from \eqref{def:Fnotation} and \eqref{def:FH} that
		\begin{align}
			\frac{\partial}{\partial \tau}F(s;\gamma,\tilde{s},\nu, \tau) &=\frac{\partial}{\partial \tau} \ln \det\left(I-\gamma  \mathcal K_{\tac}\right) = -\tr{\left((I - \gamma  \mathcal K_{\tac})^{-1} \frac{\partial}{\partial \tau} \left(\gamma \mathcal K_{\tac} \right)\right)}\nonumber\\
			&=\int_{0}^s F_4(v) h_3(v)-F_2(v) h_1(v) \ud v.
		\end{align}
		On the other hand, from \eqref{eq:Y-infty} and \eqref{eq:Yexpli} we have
		\begin{align}
			Y^{(1)} = \int_{0}^s \vec{F}(v) \vec{h}(v)^{\msf T} \ud v.
		\end{align}
		A combination of the above two equations gives us
		\begin{equation}
			\frac{\ud}{\ud \tau} F(s;\gamma,\tilde{s},\nu, \tau)	= Y^{(1)}_{43}-Y^{(1)}_{21}.
		\end{equation}
We thus obtain \eqref{eq:derivativeins-tau} by applying \eqref{def:X1} to the above formula.

This completes the proof of Proposition \ref{prop:derivativeandX}.
\qed 
	

\section{Lax pair equations and differential identities}\label{sec:Lax}

\subsection{The Lax system for $X$}

\begin{proposition}\label{pro:lax}
For the matrix-valued function $X(z) = X(z;s)$ defined in \eqref{eq:YtoX}, we have
\begin{equation}\label{eq:lax}
	\frac{\partial}{\partial z} X(z;s) = L(z;s) X(z;s), \qquad \frac{\partial}{\partial s} X(z;s) = U(z;s) X(z;s),
\end{equation}
where
\begin{equation}\label{eq:L}
	L(z;s) =A_0+\frac{A_1(s)}{z-s}+\frac{A_2(s)}{z},
\end{equation}
and
\begin{equation}\label{eq:U}
	U(z;s) = -\frac{A_1(s)}{z-s}
\end{equation}
with the functions $A_k$, $k=0,1,2$, defined in \eqref{def:A0}--\eqref{def:A2}. In addition, the functions $p_i$ and $q_i$, $i =1,\ldots,12$ in the definitions of $A_k$ satisfy the coupled differential equations \eqref{def:pq's}. 
\end{proposition}
\begin{proof}
The proof is based on the RH problem \ref{rhp:X} 
 for $X$. Since the jump matrices for $X$ are all independent of $z$ and $s$, it's easily noticed that
\begin{equation}
	L(z;s)=\frac{\partial}{\partial z}X(z;s) X(z;s)^{-1}, \qquad U(z;s)=\frac{\partial}{\partial s}X(z;s) X(z;s)^{-1}
\end{equation}
are analytic in the complex plane except for possible isolated singularities at $z=0$, $z=s$ and $z=\infty$. We now solve the functions $L(z;s)$ and $U(z;s)$ one by one.

From the large $z$ behavior of $X$ in \eqref{eq:asyX}, we have, as $z \to \infty$,
\begin{align}
		L(z;s)&= \left(X_{0}+\frac{X_1}{z}+\Boh(z^{-2})\right) \diag \left(z^{\frac{1}{4}}, z^{-\frac{1}{4}},z^{\frac{1}{4}}, z^{-\frac{1}{4}}\right) \nonumber\\   
		&\quad\times \diag \left(\begin{pmatrix} 1 & -1 \\ 1 & 1 \end{pmatrix},\begin{pmatrix} 1 & 1 \\ -1 & 1 \end{pmatrix}\right)
		\diag \left((-\sqrt{z})^{-\frac{1}{4}}, z^{-\frac{1}{8}},(-\sqrt{z})^{\frac{1}{4}}, z^{\frac{1}{8}}\right)  A \nonumber\\ 
		&\quad\times\diag \left(e^{-\theta_1(\sqrt{z})+\tau \sqrt{z}}\left(-\theta_1^{'}(\sqrt{z})+\frac{\tau}{2\sqrt{z}}\right) , e^{-\theta_2(\sqrt{z})-\tau \sqrt{z}}\left(-\theta_2^{'}(\sqrt{z})-\frac{\tau}{2\sqrt{z}}\right), \right.\nonumber\\  
		&\left. \quad  e^{\theta_1(\sqrt{z})+\tau \sqrt{z}}\left(\theta_1^{'}(\sqrt{z})+\frac{\tau}{2\sqrt{z}}\right), e^{\theta_2(\sqrt{z})-\tau \sqrt{z}}\left(\theta_2^{'}(\sqrt{z})-\frac{\tau}{2\sqrt{z}}\right) \right) X(z;s)^{-1} \nonumber \\
         & \quad +\Boh(z^{-\frac{1}{2}}) 
         \nonumber \\
		&= X_0 \begin{pmatrix}
			0 & \frac{\tau}{2} & 0 & -\frac{\mathrm{i}}{2} \\
			0 & 0 & 0 & 0 \\
			\frac{\mathrm{i}}{2} & -\frac{\mathrm{i}\tilde{s}}{2} & 0 & -\frac{\tau}{2} \\
			0 & -\frac{\mathrm{i}}{2} & 0 & 0
		\end{pmatrix}
		X_0^{-1}+\Boh(z^{-\frac{1}{2}}).
\end{align}
From the definition of $X_{0}$ in \eqref{def:X0}, we derive
\begin{align}
	A_0=\begin{pmatrix}
		0 & \frac{\tau-\mathrm{i}(-M^{(1)}_{13}+M^{(1)}_{14})}{2} & 0 & -\frac{\mathrm{i}}{2} \\
		0 & 0 & 0 & 0 \\
		\frac{\mathrm{i}}{2} & -\frac{\mathrm{i}\tilde{s}+\mathrm{i}(M^{(1)}_{11}+M^{(1)}_{12})+\mathrm{i}(-M^{(1)}_{33}+M^{(1)}_{34})}{2} & 0 & -\frac{\tau+\mathrm{i}(-M^{(1)}_{13}+M^{(1)}_{14})}{2} \\
		0 & -\frac{\mathrm{i}}{2} & 0 & 0
	\end{pmatrix},
\end{align}
where $M^{(1)}$ is given in \eqref{eq:asy:Mhat}. 
It's easily seen that $A_0$ is independent of $s$.

If $z \to s$, applying \eqref{eq:X-near-s}, we have 
\begin{equation}
	L(z;s) \sim \frac{A_1(s)}{z-s},
\end{equation}
where
\begin{equation}\label{def:A1-1}
	A_1(s)= -\frac{\gamma}{2\pi \mathrm{i}} X_{R,0}(s) 
 \begin{pmatrix}
		0 & 0 & 1 & 0 \\
		0 & 0 & 0 & 0 \\
		0 & 0 & 0 & 0\\
		0 & 0 & 0 & 0 
	 \end{pmatrix}
	X_{R,0}(s)^{-1}.
\end{equation}
By setting
\begin{equation}\label{def:q1p1}
	\begin{pmatrix}
		q_1(s)\\q_2(s)\\q_3(s)\\q_4(s)
	\end{pmatrix}=X_{R,0}(s) \begin{pmatrix}
		1\\0\\0\\0
	\end{pmatrix} \quad \textrm{and} \quad \begin{pmatrix}
		p_1(s)\\p_2(s)\\p_3(s)\\p_4(s)
	\end{pmatrix}=-\frac{\gamma}{2 \pi \ii}X_{R,0}(s)^{-\msf T} \begin{pmatrix}
		0\\0\\1\\0
	\end{pmatrix},
\end{equation}
we obtain the expression of $A_1(s)$ in \eqref{def:A1}. From \eqref{def:A1-1}, we also note that
\begin{equation}\label{eq:trA1}
	\tr A_1(s) = \sum_{k=1}^4 q_k(s)p_k(s)=0.
\end{equation}
As $z \to 0$, applying \eqref{eq:X-near-0}, we have 
\begin{equation}
	L(z;s) \sim \frac{A_2(s)}{z},
\end{equation}
where
\begin{equation}
	A_2(s)=X_{L,0}(s)T_2(s)X_{L,0}(s)^{-1}
\end{equation}
with
\begin{equation}
	T_2(s)=\left\{\begin{aligned}
		&\begin{pmatrix}
			\frac{2\nu-1}{4} & 0 & 0 & 0 \\
			0 & \frac{1-2\nu}{4} & 0 & 0 \\
			0 & 0 & \frac{1-2\nu}{4} & 0\\
			0 & 0 & 0 & \frac{2\nu-1}{4} 
		\end{pmatrix}, &\quad \nu \neq \frac{1}{2} , \\
		&\begin{pmatrix}
			0 & -\frac{1}{2\pi } & -\frac{\mathrm{i}(\gamma-1)}{2\pi } & 0 \\
			0 & 0 & 0 & 0 \\
			0 & 0 & 0 & 0 \\
			0 & 0 & \frac{1}{2\pi }& 0
		\end{pmatrix},& \quad \nu=\frac{1}{2}.
	\end{aligned}\right.
\end{equation}	
By setting
\begin{equation}\label{def:p5}
	\begin{pmatrix}
		p_5(s)\\p_6(s)\\p_7(s)\\p_8(s)
	\end{pmatrix}=X_{L,0}^{-\msf T}(s)
	\left\{\begin{aligned}
		\begin{pmatrix}
			0\\ \frac{1-2\nu}{2}\\0\\0
		\end{pmatrix} ,&\quad \nu\neq \frac{1}{2}, \\
		\begin{pmatrix}
			0\\	-\frac{1}{2\pi}\\0\\0
		\end{pmatrix}  , &\quad \nu = \frac{1}{2}, 
	\end{aligned}   \right.
 \quad 
 \begin{pmatrix}
		p_9(s)\\p_{10}(s)\\p_{11}(s)\\p_{12}(s)
	\end{pmatrix}=X_{L,0}^{-\msf T}(s)
	\left\{\begin{aligned}
		\begin{pmatrix}
			0\\ 0\\\frac{1-2\nu}{2}\\0
		\end{pmatrix} ,&\quad \nu\neq \frac{1}{2}, \\
		\begin{pmatrix}
			0\\	0\\-\frac{1}{2\pi}\\0
		\end{pmatrix}  , &\quad \nu = \frac{1}{2}, 
	\end{aligned}   \right.
\end{equation}
and
\begin{equation} \label{def:q5}
	\begin{pmatrix}
		q_5(s)\\q_6(s)\\q_7(s)\\q_8(s)
	\end{pmatrix}=X_{L,0}(s)
	\left\{\begin{aligned}
		\begin{pmatrix}
			0\\ 1\\0\\0
		\end{pmatrix} ,&\quad \nu\neq \frac{1}{2}, \\
		\begin{pmatrix}
			1\\0\\0\\0
		\end{pmatrix}  , &\quad \nu = \frac{1}{2},
	\end{aligned}   \right. 
      \qquad \begin{pmatrix}
		q_9(s)\\q_{10}(s)\\q_{11}(s)\\q_{12}(s)
	\end{pmatrix}=X_{L,0}(s)
	\left\{\begin{aligned}
		\begin{pmatrix}
			0\\ 0\\1\\0
		\end{pmatrix} ,&\quad \nu\neq \frac{1}{2}, \\
		\begin{pmatrix}
			\ii(\gamma-1)\\0\\0\\-1
		\end{pmatrix}  , &\quad \nu = \frac{1}{2},
	\end{aligned}   \right.
\end{equation}
we obtain the expression of $A_2(s)$ in \eqref{def:A2}.

The computation of $U(z;s)$ is similar. It is easy to check that
\begin{equation}
	U(z;s)=\Boh(z^{-1}), \qquad z \to \infty,
\end{equation}
and
\begin{equation}
	U(z;s) \sim -\frac{A_1(s)}{z-s}, \quad z \to s; \quad  U(z;s)=\Boh(1), \quad z \to 0,
\end{equation}
where $A_1(s)$ is given in \eqref{def:A1-1}. The above equations imply the expression of $U(z;s)$ in \eqref{eq:U}.

It remains to establish the various relations satisfied by the functions $p_i (s)$ and $q_i(s)$, $i=1,\ldots,12$, in the definitions of $A_k(s)$, $k=0,1,2$. By \eqref{eq:trA1}, we have proved \eqref{eq:sumpq}. We recall that the compatibility condition
\begin{equation}
	\frac{\partial^2}{\partial z \partial s} X(z;s) = \frac{\partial^2}{\partial s \partial z} X(z;s)
\end{equation}
for the Lax pair \eqref{eq:lax} is the zero curvature relation
\begin{equation}\label{eq:compatibility}
	\frac{\partial}{\partial s}L(s;z) - \frac{\partial}{\partial z}U(s;z) = \frac{A_1'(s)}{z-s}+\frac{A_2'(s)}{z}=[U,L].
\end{equation}

If we calculate the residue at $z=s$ and $z=0$ on the both sides of \eqref{eq:compatibility}, it is easily seen that
\begin{equation}\label{def:A1's}
	A_1'(s)=-\left[A_1(s),A_0+\frac{A_2(s)}{s}\right],
\end{equation}
\begin{equation}\label{def:A2's}
	A_2'(s)=\left[\frac{A_1(s)}{s},A_2(s)\right].
\end{equation}
On account of \eqref{eq:lax}--\eqref{eq:U}, we obtain
\begin{equation}
    \left(\frac{\partial}{\partial z}X(z;s)+\frac{\partial}{\partial s}X(z;s)\right)X(z;s)^{-1}=L(z;s)+U(z;s)
   \sim A_0+\frac{A_2(s)}{s}, \quad z \to s.
\end{equation}
On the other hand, substituting \eqref{eq:X-near-s} into the left hand side of the above equation gives us
\begin{equation}
	\left(\frac{\partial}{\partial z}X(z;s)+\frac{\partial}{\partial s}X(z;s)\right)X(z;s)^{-1} \sim \frac{\ud}{\ud s} X_{R,0}(s) \cdot X_{R,0}(s)^{-1}, \quad z \to s.
\end{equation}
Therefore, we have from the above two formulas that
\begin{equation}
	\frac{\ud}{\ud s} X_{R,0}(s)=\left(A_0+\frac{A_2(s)}{s}\right)X_{R,0}(s).
\end{equation}
Recall the definitions of $q_k(s)$, $k=1,\ldots,4$, given in \eqref{def:q1p1}, it is readily seen that
\begin{equation}\label{eq:q1's}
	\begin{pmatrix}
		q_1'(s)\\q_2'(s)\\q_3'(s)\\q_4'(s)
	\end{pmatrix}=\left(A_0+\frac{A_2(s)}{s}\right)\begin{pmatrix}
		q_1(s)\\q_2(s)\\q_3(s)\\q_4(s)
	\end{pmatrix}.
\end{equation}
It then follows from \eqref{def:A0}, \eqref{def:A2} and straightforward calculations that $q_k(s)$, $k=1,\ldots,4$ satisfy the equations in \eqref{def:pq's}.

To show the equations for $p_i(s)$, $i=1,\ldots,4$, we see from \eqref{def:A1} and \eqref{def:A1's} that
\begin{align}
	A_1'(s)&= \begin{pmatrix}
		q_1'(s)\\q_2'(s)\\q_3'(s)\\q_4'(s)
	\end{pmatrix}\begin{pmatrix}
		p_1(s) & p_2(s)&p_3(s)&p_4(s)
	\end{pmatrix}+ \begin{pmatrix}
		q_1(s)\\q_2(s)\\q_3(s)\\q_4(s)
	\end{pmatrix}\begin{pmatrix}
		p_1'(s) & p_2'(s)&p_3'(s)&p_4'(s)
	\end{pmatrix}\nonumber\\
	&=-\left[A_1(s),A_0+\frac{A_2(s)}{s}\right]\nonumber\\
	&=-\begin{pmatrix}
		q_1(s)\\q_2(s)\\q_3(s)\\q_4(s)
	\end{pmatrix}\begin{pmatrix}
		p_1(s) & p_2(s)&p_3(s)&p_4(s)
	\end{pmatrix}\left(A_0+\frac{A_2(s)}{s}\right)\nonumber\\
	&\quad +\left(A_0+\frac{A_2(s)}{s}\right)\begin{pmatrix}
		q_1(s)\\q_2(s)\\q_3(s)\\q_4(s)
	\end{pmatrix}\begin{pmatrix}
		p_1(s) & p_2(s)&p_3(s)&p_4(s)
	\end{pmatrix}.
\end{align}
A combination of this formula and \eqref{eq:q1's} gives us
\begin{align}
	\begin{pmatrix}
		p_1'(s) & p_2'(s)&p_3'(s)&p_4'(s)
	\end{pmatrix}
	=-\begin{pmatrix}
		p_1(s) & p_2(s)&p_3(s)&p_4(s)
	\end{pmatrix}\left(A_0+\frac{A_2(s)}{s}\right),
\end{align}
which is equivalent to the first four equations in \eqref{def:pq's} by using \eqref{def:A0} and \eqref{def:A2}.

Similarly, taking $z\to 0$ in 
\begin{align}
	\frac{\partial}{\partial s}X(z;s) X(z;s)^{-1}=U(z;s),
\end{align}
we have
\begin{equation}
	\frac{\ud}{\ud s} X_{L,0}(s)=\frac{A_1(s)}{s}X_{L,0}(s).
\end{equation}
Recall the definitions of $q_k(s)$, $k=5,\ldots,12$, given in \eqref{def:q5}. It is readily seen that
\begin{equation}\label{eq:q5's}
	\begin{pmatrix}
		q_5'(s)\\q_6'(s)\\q_7'(s)\\q_8'(s)
	\end{pmatrix}=\frac{A_1(s)}{s}\begin{pmatrix}
		q_5(s)\\q_6(s)\\q_7(s)\\q_8(s)
	\end{pmatrix}
~
\textrm{and}
~
	\begin{pmatrix}
		q_9'(s)\\q_{10}'(s)\\q_{11}'(s)\\q_{12}'(s)
	\end{pmatrix}=\frac{A_1(s)}{s}\begin{pmatrix}
		q_9(s)\\q_{10}(s)\\q_{11}(s)\\q_{12}(s)
	\end{pmatrix}.
\end{equation}
We then obtain the last two equations in \eqref{def:pq's}  by using \eqref{def:A1}.

To show the equations of $p_i(s)$, $i=5,\ldots,12$, we see from \eqref{def:A2} and \eqref{def:A2's} that
	\begin{align}
		A_2'(s)& = \begin{pmatrix}
			q_5'(s)\\q_6'(s)\\q_7'(s)\\q_8'(s)
		\end{pmatrix}\begin{pmatrix}
			p_5(s) & p_6(s)&p_7(s)&p_8(s)
		\end{pmatrix}+ \begin{pmatrix}
			q_5(s)\\q_6(s)\\q_7(s)\\q_8(s)
		\end{pmatrix}\begin{pmatrix}
			p_5'(s) & p_6'(s)&p_7'(s)&p_8'(s)
		\end{pmatrix}\nonumber\\
		&\quad + \begin{pmatrix}
			q_9'(s)\\q_{10}'(s)\\q_{11}'(s)\\q_{12}'(s)
		\end{pmatrix}\begin{pmatrix}
			p_9(s) & p_{10}(s)&p_{11}(s)&p_{12}(s)
		\end{pmatrix}+ \begin{pmatrix}
			q_9(s)\\q_{10}(s)\\q_{11}(s)\\q_{12}(s)
		\end{pmatrix}\begin{pmatrix}
			p_9'(s) & p_{10}'(s)&p_{11}'(s)&p_{12}'(s)
		\end{pmatrix}\nonumber\\
		&=\left[A_1(s),\frac{A_2(s)}{s}\right]\nonumber\\
		&= -\begin{pmatrix}
			q_5(s)\\q_6(s)\\q_7(s)\\q_8(s)
		\end{pmatrix}\begin{pmatrix}
			p_5(s) & p_6(s)&p_7(s)&p_8(s)
		\end{pmatrix}\frac{A_1(s)}{s}\nonumber\\
		&\quad -\begin{pmatrix}
			q_9(s)\\q_{10}(s)\\q_{11}(s)\\q_{12}(s)
		\end{pmatrix}\begin{pmatrix}
			p_9(s) & p_{10}(s)&p_{11}(s)&p_{12}(s)
		\end{pmatrix}\frac{A_1(s)}{s}\nonumber\\
		&\quad +\frac{A_1(s)}{s}\begin{pmatrix}
			q_5(s)\\q_6(s)\\q_7(s)\\q_8(s)
		\end{pmatrix}\begin{pmatrix}
			p_5(s) & p_6(s)&p_7(s)&p_8(s)
		\end{pmatrix}\nonumber\\
		&\quad+\frac{A_1(s)}{s}\begin{pmatrix}
			q_9(s)\\q_{10}(s)\\q_{11}(s)\\q_{12}(s)
		\end{pmatrix}\begin{pmatrix}
			p_9(s) & p_{10}(s)&p_{11}(s)&p_{12}(s)
		\end{pmatrix}.
	\end{align}
A combination of this formula and \eqref{eq:q5's} gives us
\begin{align}
	&\begin{pmatrix}
		q_5(s)\\q_6(s)\\q_7(s)\\q_8(s)
	\end{pmatrix}\begin{pmatrix}
		p_5'(s) & p_6'(s)&p_7'(s)&p_8'(s)
	\end{pmatrix}+\begin{pmatrix}
		q_9(s)\\q_{10}(s)\\q_{11}(s)\\q_{12}(s)
	\end{pmatrix}\begin{pmatrix}
		p_9'(s) & p_{10}'(s)&p_{11}'(s)&p_{12}'(s)
	\end{pmatrix}\nonumber\\
	&=-\begin{pmatrix}
		q_5(s)\\q_6(s)\\q_7(s)\\q_8(s)
	\end{pmatrix}\begin{pmatrix}
		p_5(s) & p_6(s)&p_7(s)&p_8(s)
	\end{pmatrix}\frac{A_1(s)}{s}\nonumber\\
	&\quad -\begin{pmatrix}
		q_9(s)\\q_{10}(s)\\q_{11}(s)\\q_{12}(s)
	\end{pmatrix}\begin{pmatrix}
		p_9(s) & p_{10}(s)&p_{11}(s)&p_{12}(s)
	\end{pmatrix}\frac{A_1(s)}{s}.
\end{align}
Note that $\begin{pmatrix}
	q_5(s)\\q_6(s)\\q_7(s)\\q_8(s)
\end{pmatrix}$ and $\begin{pmatrix}
	q_9(s)\\q_{10}(s)\\q_{11}(s)\\q_{12}(s)
\end{pmatrix}$ are linearly independent, so we have 
\begin{align}
	&\begin{pmatrix}
		p_5'(s) & p_6'(s)&p_7'(s)&p_8'(s)
	\end{pmatrix} =\begin{pmatrix}
		p_5(s) & p_6(s)&p_7(s)&p_8(s)
	\end{pmatrix} \frac{A_1(s)}{s}, \\
	&\begin{pmatrix}
		p_9'(s) & p_{10}'(s)&p_{11}'(s)&p_{12}'(s)
	\end{pmatrix} =\begin{pmatrix}
		p_9(s) & p_{10}(s)&p_{11}(s)&p_{12}(s)
	\end{pmatrix} \frac{A_1(s)}{s},
\end{align}
 we obtain the fifth and sixth equations in \eqref{def:pq's}  by using \eqref{def:A1}.

This completes the proof of Proposition \ref{pro:lax}.
\end{proof}
From the general theory of Jimbo-Miwa-Ueno \cite{JMU81}, we have that the Hamiltonian associated with the Lax system \eqref{eq:lax} is given by
\begin{align}\label{eq:JMU}
H(s) &= \frac{\gamma}{2 \pi \ii} \tr{\left( -X_{R,1}(s)\begin{pmatrix}
		0 & 0 & 1 & 0\\
		0 & 0 & 0 & 0\\
		0 & 0 & 0 & 0\\
		0 & 0 & 0 & 0
	\end{pmatrix}\right)} = -\frac{\gamma}{2 \pi \ii} X_{R,1}(s)_{31},
\end{align}
where $X_{R,1}(s)$ is given in \eqref{eq: X-expand-s}.
Taking $z \to s$ in the first equation of the Lax pair \eqref{eq:lax}, we obtain from \eqref{eq:L} and \eqref{eq:X-near-s} that the $\Boh(1)$ term gives
\begin{align}
X_{R,1}(s) &= \frac{\gamma}{2 \pi \ii}\left[X_{R,1}(s), \begin{pmatrix}
	0 & 0 & 1 & 0\\
	0 & 0 & 0 & 0\\
	0 & 0 & 0 & 0\\
	0 & 0 & 0 & 0
\end{pmatrix}\right]  + X_{R,0}(s)^{-1} \left[A_0+\frac{A_2(s)}{s}\right] X_{R,0}(s).
\end{align}
Inserting the above equation into \eqref{eq:JMU} yields
\begin{align}\label{eq:H}
H(s) &= -\frac{\gamma}{2 \pi \ii} \begin{pmatrix}
	0 & 0 & 1 & 0
\end{pmatrix}X_{R,0}(s)^{-1} \left[A_0+\frac{A_2(s)}{s}\right] X_{R,0}(s) \begin{pmatrix}
	1\\0\\0\\0
\end{pmatrix} ,
\end{align}
which is consistent with the definition of $H$ in \eqref{def:H}.
\subsection{Differential identities for the Hamiltonian}
\begin{proposition}\label{prop:H}
With the Hamiltonian $H$ defined in \eqref{def:H}, we have
\begin{align}\label{eq:differentialH}
	\frac{\ud}{\ud s}H(s)= -\begin{pmatrix}
	    p_1(s)& p_2(s) & p_3(s) &p_4(s)
	\end{pmatrix} \frac{A_2(s)}{s^2} \begin{pmatrix}
	    q_1(s)& q_2(s) & q_3(s) &q_4(s)
	\end{pmatrix}^{\msf T}
\end{align}
and 
\begin{align}\label{eq:differentialH1}
	\sum_{k=1}^{12} \left(p_k(s)q'_k(s)  \right) - H(s) = H(s)-\frac{\ud}{\ud s}(sH(s)).
\end{align}
We also have the following differential identity with respect to the parameter $\gamma$:
\begin{align}\label{eq:differential:gamma}
	\frac{\partial}{\partial \gamma}\left(\sum_{k=1}^{12} p_k(s)q'_k(s) - H(s)\right) = \frac{\ud}{\ud s}\sum_{k=1}^{12}\left(p_k(s)\frac{\partial}{\partial \gamma}q_k(s) \right).
\end{align}
\end{proposition}

\begin{proof}
The differential identity \eqref{eq:differentialH} follows directly from \eqref{eq:H}. 
The equation \eqref{eq:differentialH1} could be obtained from \eqref{eq:differentialH} and the following equation  
	\begin{align}
		&\sum_{k=1}^{12} p_k(s)q'_k(s)   - H(s) \nonumber
   \\
		& = \frac{1}{s}\left[\left(\sum_{i=1}^{4} p_i(s) q_{4+i}(s)\right) \left(\sum_{i=1}^{4} p_{4+i}(s) q_i(s)\right)+\left(\sum_{i=1}^{4} p_i(s) q_{8+i}(s)\right) \left(\sum_{i=1}^{4} p_{8+i}(s) q_i(s)\right)\right].
	\end{align}
To see the differential identity with respect to the parameter $\gamma$, we have from \eqref{pq} that
\begin{align}
	\frac{\partial}{\partial \gamma}H(s)& = \sum_{k=1}^{12}\left(\frac{\partial H}{\partial p_k} \frac{\partial}{\partial \gamma} p_k(s) + \frac{\partial H}{\partial q_k} \frac{\partial}{\partial \gamma}q_k(s)  \right)=\sum_{k=1}^{12}\left(q'_k(s) \frac{\partial}{\partial \gamma} p_k(s)  -p'_k(s) \frac{\partial}{\partial \gamma}q_k(s) \right),
\end{align}
which leads to \eqref{eq:differential:gamma}.

This completes the proof of Proposition \ref{prop:H}.
\end{proof}

\section{Large $s$ asymptotic analysis of the RH problem with $\gamma =1$}
\label{sec:AsyX1}


\subsection{First transformation: $X \to T$}
This transformation is a rescaling of the $\mathrm{RH}$ problem for $X$, which is defined by
\begin{equation}\label{def:XtoT}
T(z)=\operatorname{diag}\left(s^{-\frac{1}{8}}, s^{\frac{3}{8}}, s^{-\frac{3}{8}}, s^{\frac{1}{8}}\right) {X_0}^{-1} X(sz) ,
\end{equation}
where $X_0$ is given in \eqref{eq:asyX}. In view of $\mathrm{RH}$ problem for $X$, it is readily seen that $T$ satisfies the following $\mathrm{RH}$ problem. 
\begin{rhp}\label{rhp:T}
\hfill
\begin{enumerate}
	\item[\rm (a)] $T(z)$ is defined and analytic in $\mathbb{C} \setminus \left\{\Gamma_T \cup\{0\} \cup\{1\}\right\}$, where
	\begin{equation}\label{jumpT}
		\Gamma_T:=\cup_{j=0}^5 \Gamma_j^{(1)}
	\end{equation}
	with the contours $\Gamma_j^{(1)}, j=0,1, \ldots, 5$, defined in \eqref{def:Gammajs} with $s=1$.
	
	\item[\rm (b)] T satisfies 
				$T_{+}(z)=T_{-}(z) J_T(z)$  for $z \in \Gamma_T$, 
	where
	\begin{equation}
		J_T(z):= \begin{cases}\begin{pmatrix}
		0 & 0 & 1 & 0 \\
				0 & 1 & 0 & 0 \\
				-1 & 0 & 0 & 0 \\
				0 & 0 & 0 & 1
			\end{pmatrix}, & \quad z \in \Gamma_0^{(1)}, \\
			
			I+E_{3,1}, & \quad z \in \Gamma_1^{(1)},\\
			
			I+e^{\nu\pi \mathrm{i}} E_{2,1}-e^{\nu\pi \mathrm{i}} E_{3,4}, & \quad z \in \Gamma_2^{(1)}, \\
			\begin{pmatrix}
				0 & -\mathrm{i}& 0 & 0 \\
				-\mathrm{i} & 0 & 0 & 0 \\
				0 & 0 & 0 & \mathrm{i}\\
				0 & 0 & \mathrm{i} & 0
			\end{pmatrix}, & \quad z \in \Gamma_3^{(1)},\\
			I-e^{-\nu\pi \mathrm{i}} E_{2,1}+e^{-\nu\pi \mathrm{i}} E_{3,4}, & \quad z \in \Gamma_4^{(1)}, \\
			I+E_{3,1}, & \quad z \in \Gamma_5^{(1)}.
		\end{cases}
	\end{equation}
	
	\item[\rm (c)] As $z \to \infty$ with $z \in \mathbb{C} \setminus  \Gamma_T$, we have
	\begin{align}\label{eq:asyT}
			T(z)&=  \left(I+\frac{T^{(1)}}{z}+\Boh(z^{-2})\right)\diag \left(z^{\frac{1}{4}}, z^{-\frac{1}{4}},z^{\frac{1}{4}}, z^{-\frac{1}{4}}\right)\nonumber\\
			&\quad \times    \diag \left(\begin{pmatrix}
				1 & -1 \\ 1 & 1 	\end{pmatrix},\begin{pmatrix}
				1 & 1 \\ -1 & 1 	\end{pmatrix}\right) 
			\diag \left((-\sqrt{z})^{-\frac{1}{4}}, z^{-\frac{1}{8}},(-\sqrt{z})^{\frac{1}{4}}, z^{\frac{1}{8}}\right)\nonumber\\
			&\quad \times A \diag \left(e^{-s^{\frac34} \what \theta_1(z)+\tau \sqrt{sz}}, e^{-s^{\frac34} \what \theta_2(z)-\tau \sqrt{sz}}, e^{s^{\frac34} \what \theta_1(z)+\tau \sqrt{sz}},  e^{s^{\frac34}\what \theta_2(z)-\tau \sqrt{sz}}\right),
		\end{align}
	where $T^{(1)}$ is independent of $z$, $A$ is defined in \eqref{def:A} and
	\begin{align}\label{def:wtilthetai}
			& \what \theta_1(z)=\frac{2}{3} (-\sqrt{z})^{\frac{3}{2}}+\frac{2 \tilde{s}}{\sqrt{s}}(-\sqrt{z})^{\frac{1}{2}}, &&\quad z \in \mathbb{C} \backslash \mathbb{R}, \\
			& \what \theta_2(z)=\frac{2}{3}   z^{\frac{3}{4}}+\frac{2 \tilde{s}}{\sqrt{s}} z^{\frac{1}{4}}, &&\quad z \in \mathbb{C} \backslash(-\infty, 0] .
		\end{align}
	\item[\rm (d)]
	As $z \to  1$, we have $ T(z)=\Boh(\ln(z - 1))$.
 \item [\rm(e)]
 As $z \to  0$, we have
 \begin{equation}
			T(z) =\Boh(1)\cdot \begin{pmatrix}
						z^{\frac{2\nu-1}{4}} & \delta_{\nu}(z) & 0 & 0 \\
						0 & z^{-\frac{2\nu-1}{4}} & 0 & 0 \\
						0 & 0 & z^{-\frac{2\nu-1}{4}} & 0 \\
						0 & 0 & -\delta_{\nu}(z) & z^{\frac{2\nu-1}{4}}
					\end{pmatrix}, \quad z\in \Omega_1^{(1)},
		\end{equation}
  with $\delta_{\nu}(z)$ defined in \eqref{def:delta}.
\end{enumerate}
\end{rhp}

\subsection{Second transformation: $T \to S$}
In this transformation we partially normalize RH problem \ref{rhp:T} for $T$ at infinity. For this purpose, we introduce the following two $g$-functions:
\begin{align}
g_1(z)=\frac{2}{3} (1-\sqrt{z})^{\frac{3}{2}}+\left(-1+\frac{2 \tilde{s}}{\sqrt{s}}\right)(1-\sqrt{z})^{\frac{1}{2}}, & \qquad z \in \mathbb{C} \backslash\{(-\infty,0]\cup [1,+\infty)\}, \label{def:g1}\\
g_2(z)=\frac{2}{3} (\sqrt{z}+1)^{\frac{3}{2}}+\left(-1+\frac{2 \tilde{s}}{\sqrt{s}}\right)(\sqrt{z}+1)^{\frac{1}{2}}, & \qquad z \in \mathbb{C} \backslash(-\infty,0]. \label{def:g2}
\end{align}

As $z\to \infty$, it is readily seen that
\begin{align}
g_1(z)&=\what\theta_{1}(z)+\left(-\frac{1}{4}+\frac{\tilde{s}}{\sqrt{s}}\right)(-\sqrt{z})^{-\frac12}+\Boh(z^{-\frac34}), \label{eq:asyg1}
\\
g_2(z)&=\what\theta_{2}(z)+\left(-\frac{1}{4}+\frac{\tilde{s}}{\sqrt{s}}\right)z^{-\frac14}+\Boh(z^{-\frac34}), \label{eq:asyg2}
\end{align}
where $\what \theta_{i}(z)$, $i=1,2$, are defined in \eqref{def:wtilthetai}. The second transformation is defined as 
\begin{equation}\label{def:TtoS}
\begin{aligned}
	S(z)= & \left(I+\mathrm{i}s^\frac{3}{4}\left(-\frac{1}{4}+\frac{\tilde{s}}{\sqrt{s}}\right)(E_{1,4}+E_{4,2}-E_{3,1})\right)T(z) \\
	& \times \diag \left(e^{s^{\frac{3}{4}} g_1(z)-\tau \sqrt{s z}}, e^{s^{\frac{3}{4}} g_2(z)+\tau \sqrt{s z}}, e^{-s^{\frac{3}{4}} g_1(z)-\tau \sqrt{s z}}, e^{-s^{\frac{3}{4}} g_2(z)+\tau \sqrt{s z}}\right) .
\end{aligned}
\end{equation}
Then, $S$ satisfies the following RH problem.
\begin{rhp}\label{rhp:S}
	\hfill
\begin{enumerate}

	\item[\rm (a)] $S(z)$ is defined and analytic in $\mathbb{C} \setminus \left\{\Gamma_T \cup\{0\} \cup\{1\}\right\}$, where $\Gamma_T $ is defined in \eqref{jumpT} .
	\item[\rm (b)] S satisfies 
				$S_{+}(z)=S_{-}(z) J_S(z)$  for $z \in \Gamma_T$,
	where
\begin{equation}\label{def:JS}
		J_S(z):= \begin{cases}\begin{pmatrix}
				0 & 0 & 1 & 0 \\
				0 & 1 & 0 & 0 \\
				-1 & 0 & 0 & 0 \\
				0 & 0 & 0 & 1
			\end{pmatrix}, & z \in \Gamma_0^{(1)}, \\
			
			I+e^{2s^\frac{3}{4}g_1(z)}E_{3,1}, & z \in \Gamma_1^{(1)},\\
			
			I+e^{\nu\pi \mathrm{i}+s^\frac{3}{4}(g_1(z)-g_2(z))-2\tau \sqrt{sz}} E_{2,1}-e^{\nu\pi \mathrm{i}+s^\frac{3}{4}(g_1(z)-g_2(z))+2\tau \sqrt{sz}} E_{3,4}, & z \in \Gamma_2^{(1)}, \\
			\begin{pmatrix}
				0 & -\mathrm{i}& 0 & 0 \\
				-\mathrm{i} & 0 & 0 & 0 \\
				0 & 0 & 0 & \mathrm{i}\\
				0 & 0 & \mathrm{i} & 0
			\end{pmatrix}, & z \in \Gamma_3^{(1)},\\
			I-e^{-\nu \pi \mathrm{i}+s^\frac{3}{4}(g_1(z)-g_2(z))-2\tau \sqrt{sz}} E_{2,1}+e^{-\nu\pi \mathrm{i}+s^\frac{3}{4}(g_1(z)-g_2(z))+2\tau \sqrt{sz}} E_{3,4}, & z \in \Gamma_4^{(1)}, \\
			
			I+e^{2s^\frac{3}{4}g_1(z)}E_{3,1}, & z \in \Gamma_5^{(1)}.
		\end{cases}
	\end{equation}
	\item[\rm (c)]As $z \to \infty$ with $z \in \mathbb{C} \setminus  \Gamma_T$, we have
	\begin{align}\label{eq:asyS}
		S(z)&=  \left(I+\frac{S^{(1)}}{z}+O(z^{-2})\right)\diag \left(z^{\frac{1}{4}}, z^{-\frac{1}{4}},z^{\frac{1}{4}}, z^{-\frac{1}{4}}\right)
		\nonumber \\
		&\quad \times    \diag \left(\begin{pmatrix}
			1 & -1 \\ 1 & 1 	\end{pmatrix},\begin{pmatrix}
			1 & 1 \\ -1 & 1 	\end{pmatrix}\right) 
		\diag \left((-\sqrt{z})^{-\frac{1}{4}}, z^{-\frac{1}{8}},(-\sqrt{z})^{\frac{1}{4}}, z^{\frac{1}{8}}\right) A,
	\end{align}
	where $S^{(1)}$ is independent of $z$ and $A$ is defined in \eqref{def:A}. 
	\item[\rm (d)]
	As $z \to  1$ or $z \to  0$, $S$ has the same local behaviors as $T$.
\end{enumerate}
\end{rhp}
\begin{proof}
All the items follow directly from \eqref{def:TtoS} and RH problem \ref{rhp:T} for $T$. In particular, to check the jump condition of $S$ on $\mathbb{R}$, we need the facts that
\begin{align*}
	g_{1,+}(x)+g_{1,-}(x)&=0, \qquad x \in [1,+\infty),
	\\
	-g_{1,\pm}(x)+g_{2,\mp}(x)&=0, \qquad x \in (-\infty,0].
\end{align*}

To establish the large $z$ behavior of $S$ shown in item (c), we observe from \eqref{eq:asyT}, \eqref{eq:asyg1} and \eqref{eq:asyg2} that, as $z\to \infty$,
\begin{align}
	& T(z) \diag \left(e^{s^{\frac{3}{4}} g_1(z)-\tau \sqrt{s z}}, e^{s^{\frac{3}{4}} g_2(z)+\tau \sqrt{s z}}, e^{-s^{\frac{3}{4}} g_1(z)-\tau \sqrt{s z}}, e^{-s^{\frac{3}{4}} g_2(z)+\tau \sqrt{s z}}\right)
	\nonumber
	\\
	& = \left( I+ \Boh(z^{-1}) \right) \diag \left(z^{\frac{1}{4}}, z^{-\frac{1}{4}},z^{\frac{1}{4}}, z^{-\frac{1}{4}}\right)    \diag \left(\begin{pmatrix}
		1 & -1 \\ 1 & 1 	\end{pmatrix},\begin{pmatrix}
		1 & 1 \\ -1 & 1 	\end{pmatrix}\right) 
\nonumber
	\\	
 & \quad \times \diag \left((-\sqrt{z})^{-\frac{1}{4}}, z^{-\frac{1}{8}},(-\sqrt{z})^{\frac{1}{4}}, z^{\frac{1}{8}}\right) A
	\nonumber
	\\
	& \quad \times \Big (I+ s^{\frac34}\left(-\frac{1}{4}+\frac{\tilde{s}}{\sqrt{s}}\right)(-\sqrt{z})^{-\frac12}E_{1,1} + s^{\frac34}\left(-\frac{1}{4}+\frac{\tilde{s}}{\sqrt{s}}\right)z^{\frac14}E_{2,2}
	\nonumber
	\\
	& \qquad -s^{\frac34}\left(-\frac{1}{4}+\frac{\tilde{s}}{\sqrt{s}}\right)(-\sqrt{z})^{-\frac12}E_{3,3} -s^{\frac34}\left(-\frac{1}{4}+\frac{\tilde{s}}{\sqrt{s}}\right)z^{\frac14}E_{4,4}+\Boh (z^{-\frac34}) \Big).
	\label{eq:Texp}
\end{align}
By a direct calculation, it follows that
\begin{align}
	& \diag \left(z^{\frac{1}{4}}, z^{-\frac{1}{4}},z^{\frac{1}{4}}, z^{-\frac{1}{4}}\right)    \diag \left(\begin{pmatrix}
		1 & -1 \\ 1 & 1 	\end{pmatrix},\begin{pmatrix}
		1 & 1 \\ -1 & 1 	\end{pmatrix}\right) 
	\diag \left((-\sqrt{z})^{-\frac{1}{4}}, z^{-\frac{1}{8}},(-\sqrt{z})^{\frac{1}{4}}, z^{\frac{1}{8}}\right) A
	\nonumber
	\\
	&  \times \Big (I+ s^{\frac34}\left(-\frac{1}{4}+\frac{\tilde{s}}{\sqrt{s}}\right)(-\sqrt{z})^{-\frac12}E_{1,1} + s^{\frac34}\left(-\frac{1}{4}+\frac{\tilde{s}}{\sqrt{s}}\right)z^{-\frac14}E_{2,2}
	\nonumber
	\\
	&\quad  -s^{\frac34}\left(-\frac{1}{4}+\frac{\tilde{s}}{\sqrt{s}}\right)(-\sqrt{z})^{-\frac12}E_{3,3} -s^{\frac34}\left(-\frac{1}{4}+\frac{\tilde{s}}{\sqrt{s}}\right)z^{-\frac14}E_{4,4} \Big)
	\nonumber
	\\
	&=\left(I-\mathrm{i}s^\frac{3}{4}(-\frac{1}{4}+\frac{\tilde{s}}{\sqrt{s}})(E_{1,4}+E_{4,2}-E_{3,1})+\Boh(z^{-1})\right)
	\nonumber
	\\
	&\quad \times  \diag \left(z^{\frac{1}{4}}, z^{-\frac{1}{4}},z^{\frac{1}{4}}, z^{-\frac{1}{4}}\right)    \diag \left(\begin{pmatrix}
		1 & -1 \\ 1 & 1 	\end{pmatrix},\begin{pmatrix}
		1 & 1 \\ -1 & 1 	\end{pmatrix}\right) 
\nonumber
	\\	
 &\quad \times \diag \left((-\sqrt{z})^{-\frac{1}{4}}, z^{-\frac{1}{8}},(-\sqrt{z})^{\frac{1}{4}}, z^{\frac{1}{8}}\right) A.
\end{align}
This, together with \eqref{def:TtoS} and \eqref{eq:Texp}, gives us \eqref{eq:asyS}.
\end{proof}

\subsection{Global parametrix}
As $s \to \infty$, all the jump matrix $J_S(z)$ of $S$ given in \eqref{def:JS} tend to the identity matrix exponentially fast for $z$ bounded away from the intervals $(-\infty,0)\cup(1,+\infty)$.
Indeed, by the definitions of $g$-functions in \eqref{def:g1} and \eqref{def:g2}, it is readily seen that, as $z$ large,
\begin{equation}\label{eq:largezg1}
\Re g_1(z) \sim \Re \left( \frac23 (1-\sqrt{z})^{\frac 32} \right)<0, \quad \arg (z-1)\in (-\pi, 0) \cup (0, \pi),
\end{equation}
and
\begin{equation}\label{eq:largezg1-g2}
\Re \left(g_1(z)-g_2(z)\right) \sim \Re \left( \frac23 (1-\sqrt{z})^{\frac 32} -\frac23 (\sqrt{z} + 1)^{\frac 32} \right)<0, \quad \arg z\in (-\pi, \pi).
\end{equation}
Moreover, when $s$ is large while $z$ remains bounded, the signature picture of $\Re g_1(z)$ and $\Re \left(g_1(z)-g_2(z)\right)$ is depicted in Figure \ref{fig:real}.
A combination of \eqref{eq:largezg1}, \eqref{eq:largezg1-g2} and Figure \ref{fig:real} indicates that $J_S(z) \to I$ as $s \to +\infty$ for $z \in \Gamma_T \setminus(\Gamma_0^{(1)} \cup \Gamma_3^{(1)})$ by deforming the contours if necessary.

\begin{figure}[!ht]
\centering
\subfloat[\label{subfig-1}]{%
	\begin{overpic}[width=0.45\textwidth]{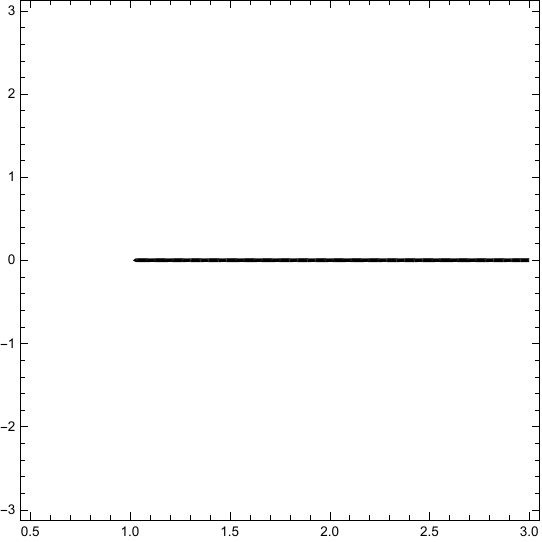}
		\put(20,45){$(1,0)$}
		\put(40,70){$-$}
        \put(40,25){$-$}
	\end{overpic}
}\hspace{1.5mm}
\subfloat[\label{subfig-2}]{%
	\begin{overpic}[width=0.45\textwidth]{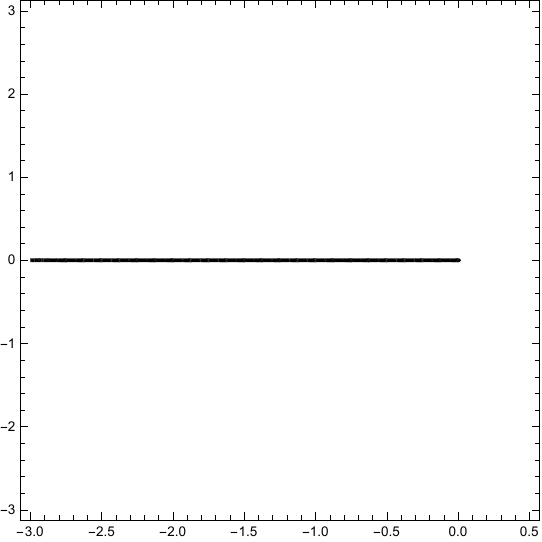}
		\put(80,45){$(0,0)$}
		\put(40,70){$-$}
        \put(40,25){$-$}
	\end{overpic}
}
\caption{Signature table of $\Re g_1(z)$ $(a)$ and $\Re \left(g_1(z)-g_2(z)\right)$ $(b)$ for large $s$ with $z$ bounded. 
The $``-"$ sign represents where $\Re g_1(z) < 0$ $(a)$ and $\Re (g_1(z)- g_2(z))< 0$ $(b)$.}\label{fig:real}
\end{figure}

\begin{rhp}
\hfill
\begin{enumerate}
	\item[\rm (a)] $N(z)$ is defined and analytic in $\mathbb{C} \setminus 
 \{(-\infty,0] \cup [1,+\infty)\} $.
	
	\item[\rm (b)] $N$ satisfies
		$N_{+}(x)=N_{-}(x) J_N(x)$ for $z\in (-\infty,0) \cup(1,+\infty)$, 
	where
	\begin{equation}
		J_N(x):= \begin{cases}\begin{pmatrix}
				0 & 0 & 1 & 0 \\
				0 & 1 & 0 & 0 \\
				-1 & 0 & 0 & 0 \\
				0 & 0 & 0 & 1
			\end{pmatrix}, & x>1, \\ 
			\begin{pmatrix}
				0 & -\mathrm{i}& 0 & 0 \\
				-\mathrm{i} & 0 & 0 & 0 \\
				0 & 0 & 0 & \mathrm{i}\\
				0 & 0 & \mathrm{i} & 0
			\end{pmatrix}, & x<0.
		\end{cases}
	\end{equation}
	\item[\rm (c)]As $z \to \infty$ with $z \in \mathbb{C} \setminus \mathbb{R}$,  we have
	\begin{align}\label{eq:asyN}
		N(z)&=\left( I+\frac{N^{(1)}}{z}+ \Boh(z^{-2}) \right) \diag \left(z^{\frac{1}{4}}, z^{-\frac{1}{4}},z^{\frac{1}{4}}, z^{-\frac{1}{4}}\right)
		\nonumber \\
		&\quad \times    \diag \left(\begin{pmatrix}
			1 & -1 \\ 1 & 1 	\end{pmatrix},\begin{pmatrix}
			1 & 1 \\ -1 & 1 	\end{pmatrix}\right) 
		\diag \left((-\sqrt{z})^{-\frac{1}{4}}, z^{-\frac{1}{8}},(-\sqrt{z})^{\frac{1}{4}}, z^{\frac{1}{8}}\right) A,
	\end{align}
	where $N^{(1)}$ is independent of $z$ and $A$ is defined in \eqref{def:A}.
\end{enumerate}
\end{rhp}
The above RH problem can be solved explicitly, and its solution is given by

\begin{align}\label{def:N}
N(z)&=\diag \left(z^{\frac{1}{4}}, z^{-\frac{1}{4}},z^{\frac{1}{4}}, z^{-\frac{1}{4}}\right)    \diag \left(\begin{pmatrix} 1 & -1 \\ 1 & 1 \end{pmatrix},\begin{pmatrix} 1 & 1 \\ -1 & 1 \end{pmatrix}\right) \nonumber
\\ &\quad \times\diag \left((1-\sqrt{z})^{-\frac{1}{4}}, (1+\sqrt{z})^{-\frac{1}{4}},(1-\sqrt{z})^{\frac{1}{4}}, (1+\sqrt{z})^{\frac{1}{4}}\right) A,
\end{align}
where we take the branch cuts of $(1-\sqrt{z})^{\frac14}$ and $(\sqrt{z}+1)^{\frac14}$ along $(-\infty,0]\cup [1,\infty)$ and $(-\infty,0]$, respectively.

\subsection{Local parametrix near $z=0$}
In a small neighborhood of the origin, we intend to solve the following RH problem, which serves as an approximation of $S$. 
\begin{rhp}\label{rhp:P0}
\hfill
\begin{itemize}

	\item[\rm (a)] $P^{(0)}(z)$ is defined and analytic in $D(0, \varepsilon) \setminus  \Gamma_T$, where $D\left(z_0, \varepsilon\right)$ and $\Gamma_T$  are defined in \eqref{def:dz0r} and \eqref{jumpT},	respectively.
	\item[\rm (b)] For $z \in D(0, \varepsilon) \cap \Gamma_T$, we have
	\begin{equation}\label{eq:P0-jump}
		P_{+}^{(0)}(z)=P_{-}^{(0)}(z) J_S(z),
	\end{equation}
	where $J_S(z)$ is defined in \eqref{def:JS}.
	\item[\rm (c)] As $s \rightarrow \infty, P^{(0)}(z)$ satisfies the matching condition
	\begin{equation}\label{eq:P0:matching}
		P^{(0)}(z)=\left(I+\Boh\left(s^{-\frac{3}{4}}\right)\right) N(z), \quad z \in \partial D(0, \varepsilon),
	\end{equation}
where $N(z)$ is given in \eqref{def:N}.
\end{itemize}
\end{rhp}
The above RH problem can be solved by using the  Bessel parametrix $\Phi_{\alpha}^{(\operatorname{Bes})}(z)$ defined in Appendix \ref{A}. To proceed,  we introduce the functions
\begin{equation}\label{def:f}
f_{0,1}(z)=\left ( \frac{1}{2} +\frac{ \tilde{s}}{\sqrt{s}}+\tau s^{-\frac{1}{4}}\right)^2 z \quad \textrm{and} \quad f_{0,2}(z)=\left ( \frac{1}{2} +\frac{ \tilde{s}}{\sqrt{s}}-\tau s^{-\frac{1}{4}}\right)^2 z.
\end{equation}
Clearly, $f_{0,1}(z)$ and $f_{0,2}(z)$ are analytic in $D(0, \varepsilon)$ and are conformal mappings for large positive $s$. We now define
\begin{small}
\begin{align}\label{def:P0}
	& P^{(0)}(z)=E_0(z)\nonumber\\
	&\times\begin{pmatrix}
		\Phi^{(\mathrm{Bes})}_{\alpha,11}\left(s^\frac{3}{2} f_{0,1}(z)\right) & -\mathrm{i}\Phi^{(\mathrm{Bes})}_{\alpha,12}\left(s^\frac{3}{2} f_{0,1}(z)\right) & 0 & 0 \\
		\Phi^{(\mathrm{Bes})}_{\alpha,21}\left(s^\frac{3}{2} f_{0,1}(z)\right) & -\mathrm{i}\Phi^{(\mathrm{Bes})}_{\alpha,22}\left(s^\frac{3}{2} f_{0,1}(z)\right) & 0 & 0 \\
		0 & 0 & \Phi^{(\mathrm{Bes})}_{\alpha,12}\left(s^\frac{3}{2} f_{0,2}(z)\right) & -\mathrm{i}\Phi^{(\mathrm{Bes})}_{\alpha,11}\left(s^\frac{3}{2} f_{0,2}(z)\right) \\
		0 & 0 & \Phi^{(\mathrm{Bes})}_{\alpha,22}\left(s^\frac{3}{2} f_{0,2}(z)\right) & -\mathrm{i}\Phi^{(\mathrm{Bes})}_{\alpha,21}\left(s^\frac{3}{2} f_{0,2}(z)\right)
	\end{pmatrix}\nonumber \\
	& \times\begin{pmatrix}
		e^{s^{\frac{3}{4}} \frac{g_1(z)-g_2(z) }{2}-\tau \sqrt{sz}} & 0 & 0 & 0 \\
		0 & e^{-s^{\frac{3}{4}} \frac{g_1(z)-g_2(z) }{2}+\tau \sqrt{sz}}  & 0 & 0 \\
		0 & 0 & e^{-s^{\frac{3}{4}} \frac{g_1(z)-g_2(z) }{2}-\tau \sqrt{sz}} & 0 \\
		0 & 0 & 0 & e^{s^{\frac{3}{4}} \frac{g_1(z)-g_2(z) }{2}+\tau \sqrt{sz}} 
	\end{pmatrix},
\end{align}
\end{small}
where
\begin{small}
\begin{align}\label{def:E0}
E_0(z)=\frac{1}{\sqrt{2}} N(z)\begin{pmatrix}
	\pi^{\frac{1}{2}} s^{\frac{3}{8}} f_{0,1}(z)^{\frac{1}{4}} &-\mathrm{i}\pi^{-\frac{1}{2}} s^{-\frac{3}{8}} f_{0,1}(z)^{-\frac{1}{4}} & 0 & 0 \\
	\pi^{\frac{1}{2}} s^{\frac{3}{8}} f_{0,1}(z)^{\frac{1}{4}}& \mathrm{i}\pi^{-\frac{1}{2}} s^{-\frac{3}{8}} f_{0,1}(z)^{-\frac{1}{4}} &0 & 0 \\
	0 & 0 & -\mathrm{i}\pi^{\frac{1}{2}} s^{\frac{3}{8}} f_{0,2}(z)^{\frac{1}{4}} & \pi^{-\frac{1}{2}} s^{-\frac{3}{8}} f_{0,2}(z)^{-\frac{1}{4}}\\
	0 & 0 &\mathrm{i}\pi^{\frac{1}{2}} s^{\frac{3}{8}}f_{0,2}(z)^{\frac{1}{4}}  & \pi^{-\frac{1}{2}} s^{-\frac{3}{8}} f_{0,2}(z)^{-\frac{1}{4}}
\end{pmatrix},
\end{align}
\end{small}
with $\alpha=\nu-1/2$.

\begin{proposition}\label{pro:P0}
The function $P^{(0)}(z)$ defined in \eqref{def:P0} solves RH problem \ref{rhp:P0}.
\end{proposition}
\begin{proof}
We start with showing that $E_0(z)$ is analytic in $D(0, \varepsilon)$. From \eqref{def:E0}, the only possible jump is on $(- \varepsilon, 0)$, and for $z \in (- \varepsilon, 0)$,
\begin{small}
	
	\begin{align}
		&E_{0,-}(z)^{-1}E_{0,+}(z)\nonumber\\
		&=\frac 12  \begin{pmatrix}
			\pi^{-\frac 12} s^{-\frac 38} f_{0,1,-}(z)^{-\frac 14} & \pi^{-\frac 12} s^{-\frac 38} f_{0,1,-}(z)^{-\frac 14} & 0 & 0\\
			\ii\pi^{\frac 12} s^{\frac 38} f_{0,1,-}(z)^{\frac 14} & -\ii \pi^{\frac 12} s^{\frac 38} f_{0,1,-}(z)^{\frac 14} & 0 & 0\\
			0 & 0 & \ii\pi^{-\frac 12} s^{-\frac 38} f_{0,2,-}(z)^{-\frac 14} & -\ii \pi^{-\frac 12} s^{-\frac 38} f_{0,2,-}(z)^{-\frac 14} \\
			0 & 0 & \pi^{\frac 12} s^{\frac 38} f_{0,2,-}(z)^{\frac 14}  & \pi^{\frac 12} s^{\frac 38} f_{0,2,-}(z)^{\frac 14}
		\end{pmatrix} \nonumber\\
		&\quad \times\begin{pmatrix}0&-\ii&0&0\\-\ii&0&0&0\\0&0&0&\ii\\0&0&\ii&0 \end{pmatrix}
		\nonumber\\ 	&\quad\times \begin{pmatrix}
			\pi^{\frac{1}{2}} s^{\frac{3}{8}} f_{0,1,+}(z)^{\frac{1}{4}} &-\mathrm{i}\pi^{-\frac{1}{2}} s^{-\frac{3}{8}} f_{0,1,+}(z)^{-\frac{1}{4}} & 0 & 0 \\
			\pi^{\frac{1}{2}} s^{\frac{3}{8}} f_{0,1,+}(z)^{\frac{1}{4}}& \mathrm{i}\pi^{-\frac{1}{2}} s^{-\frac{3}{8}} f_{0,1,+}(z)^{-\frac{1}{4}} &0 & 0 \\
			0 & 0 & -\mathrm{i}\pi^{\frac{1}{2}} s^{\frac{3}{8}} f_{0,2,+}(z)^{\frac{1}{4}} & \pi^{-\frac{1}{2}} s^{-\frac{3}{8}} f_{0,2,+}(z)^{-\frac{1}{4}}\\
			0 & 0 &\mathrm{i}\pi^{\frac{1}{2}} s^{\frac{3}{8}} f_{0,2,+}(z)^{\frac{1}{4}}  & \pi^{-\frac{1}{2}} s^{-\frac{3}{8}} f_{0,2,+}(z)^{-\frac{1}{4}}
		\end{pmatrix}=I. 
	\end{align}
\end{small}
Moreover, it is readily seen that 
\begin{align}
	E_0(0)=\begin{pmatrix}
		0 & \ii \pi^{-\frac 12} s^{-\frac 38} \textsf{a}^{-1} & 0 &\ii \pi^{-\frac 12} s^{-\frac 38} \textsf{b}^{-1} \\
		\pi^{\frac 12} s^{\frac 38} \textsf{a} & -\frac{1}{4}\ii \pi^{-\frac 12} s^{-\frac 38}\textsf{a}^{-1} & -\pi^{\frac 12} s^{\frac 38} \textsf{b} & -\frac{1}{4}\ii \pi^{-\frac 12} s^{-\frac 38}\textsf{b}^{-1} \\
		0 & -\pi^{-\frac 12} s^{-\frac 38} \textsf{a}^{-1} & 0 & -\ii \pi^{-\frac 12} s^{-\frac 38} \textsf{b}^{-1} \\
		\ii \pi^{\frac 12} s^{\frac 38} \textsf{a} &  -\frac{1}{4} \pi^{-\frac 12} s^{-\frac 38}\textsf{a}^{-1}  & -\ii \pi^{\frac 12} s^{\frac 38} \textsf{b}  & -\frac{1}{4} \pi^{-\frac 12} s^{-\frac 38}\textsf{b}^{-1}
	\end{pmatrix},
\end{align}
where
\begin{equation}
	\textsf{a}=\left ( \frac{1}{2} +\frac{ \tilde{s}}{\sqrt{s}}+\tau s^{-\frac{1}{4}}\right)^{\frac 12},\qquad\textsf{b}=\left ( \frac{1}{2} +\frac{ \tilde{s}}{\sqrt{s}}-\tau s^{-\frac{1}{4}}\right)^{\frac 12}.
\end{equation}
Therefore $E_0(z)$ is analytic in $D(0, \varepsilon)$.
The jump condition of $P^{(0)}(z)$ in \eqref{eq:P0-jump} can be verified from the analyticity of $E_0(z)$ and the jump condition in \eqref{eq:jump:Bessel}.

Finally, we check the matching condition in \eqref{eq:P0:matching}, it follows from the asymptotic behavior of the Bessel parametrix at infinity in \eqref{eq:infty:Bessel} that, as $s \to +\infty$,
\begin{equation}
	P^{(0)}(z) N(z)^{-1} = I + \frac{J^{(0)}_1(z)}{s^{3/4}} + \Boh(s^{-3/2}),
\end{equation}
where
\begin{equation}\label{def:J0}
	J^{(0)}_1(z) = \frac{1}{8} N(z) \begin{pmatrix}
		-\frac{1+4\alpha^2}{f_{0,1}(z)^{1/2}} & -\frac{2}{f_{0,1}(z)^{1/2}} & 0 & 0\\
		\frac{2}{f_{0,1}(z)^{1/2}} & \frac{1+4\alpha^2}{f_{0,1}(z)^{1/2}} & 0 & 0\\
		0 & 0 & \frac{1+4\alpha^2}{f_{0,2}(z)^{1/2}} & -\frac{2}{f_{0,2}(z)^{1/2}}\\
		0 & 0 & \frac{2}{f_{0,2}(z)^{1/2}} & -\frac{1+4\alpha^2}{f_{0,2}(z)^{1/2}}
	\end{pmatrix}N(z)^{-1}.
\end{equation}

This completes the proof of Proposition \ref{pro:P0}.
\end{proof}
	
For later use, we record the residue of $J^{(0)}_1(z)$ at $z = 0$:
	\begin{align}
		\Res_{\zeta = 0} J^{(0)}_1(\zeta) =-\frac{(4\alpha^2-1)\ii }{8 (\frac{1}{2} +\frac{ \tilde{s}}{\sqrt{s}}+\tau s^{-1/4})} E_{41}-\frac{(4\alpha^2-1)\ii }{8 (\frac{1}{2} +\frac{ \tilde{s}}{\sqrt{s}}-\tau s^{-1/4})} E_{23}.
	\end{align}
	
	\subsection{Local parametrix near $z=1$}
	We then move to the construction of the local parametrix near $z=1$.
	\begin{rhp}\label{rhp:P1}
		\hfill
		\begin{itemize}
			\item[\rm (a)] $P^{(1)}(z)$ is defined and analytic in $D(1, \varepsilon) \setminus  \Gamma_T$,  where $D\left(z_0, \varepsilon\right)$ and $\Gamma_T$  are defined in \eqref{def:dz0r} and \eqref{jumpT},	respectively.
			
			\item[\rm (b)] For $z \in D(1, \varepsilon) \cap \Gamma_T$, we have
			\begin{equation}\label{eq:P1-jump}
				P_{+}^{(1)}(z)=P_{-}^{(1)}(z) J_S(z),
			\end{equation}
			where $J_S(z)$ is defined in \eqref{def:JS}.
			\item[\rm (c)] As $s \rightarrow \infty$, $P^{(1)}(z)$ satisfies the following matching condition
			\begin{equation}\label{eq:P1:matching}
				P^{(1)}(z)=\left(I+\Boh \left(s^{-\frac{3}{4}}\right)\right) N(z), \quad z \in \partial D(1, \varepsilon) ,
			\end{equation}
   where $N(z)$ is given in \eqref{def:N}.
		\end{itemize}
		
	\end{rhp}
In a similar way, the above RH problem can be solved by using the  Bessel parametrix $\Phi_{\alpha}^{(\operatorname{Bes})}(z)$ with $\alpha=0$ defined in Appendix \ref{A}. We introduce the  function
	\begin{align}\label{def:tildef}
   & f_1(z)=g_1(z)^2
  \nonumber \\
  &=\frac{1}{2} \left(1-\frac{2 \tilde{s}}{\sqrt{s}}\right)^2(1-z)-\frac{1}{8} \left(\frac{5}{3}+\frac{2 \tilde{s}}{\sqrt{s}}\right) \left(1-\frac{2 \tilde{s}}{\sqrt{s}}\right)(1-z)^2+\Boh((1-z)^3), \quad z \to 1,
	\end{align}
where $g_1(z)$ is defined in \eqref{def:g1}. Clearly, $f_1(z)$ is analytic in $D(1, \varepsilon)$ and is a conformal mapping for large positive $s$. We now define
	\begin{align}\label{def:P1}
		P^{(1)}(z)&=E_1(z)\begin{pmatrix}
			\Phi^{(\mathrm{Bes})}_{0,11}\left(s^\frac{3}{2} f_1(z)\right) & 0 & -\Phi^{(\mathrm{Bes})}_{0,12}\left(s^\frac{3}{2} f_1(z)\right) & 0 \\
			0 & 1 & 0 & 0 \\
			-\Phi^{(\mathrm{Bes})}_{0,21}\left(s^\frac{3}{2} f_1(z)\right) & 0 & \Phi^{(\mathrm{Bes})}_{0,22}\left(s^\frac{3}{2} f_1(z)\right) & 0 \\
			0 & 0 & 0 & 1
		\end{pmatrix} \nonumber \\
		&\quad \times\begin{pmatrix}
			e^{s^{\frac{3}4} g_1(z)} & 0 & 0 & 0 \\
			0 & 1 & 0 & 0 \\
			0 & 0 & e^{-s^{\frac{3}4} g_1(z)} & 0 \\
			0 & 0 & 0 & 1
		\end{pmatrix},
	\end{align}
	where
	\begin{equation}\label{def:E1}
		E_1(z)=\frac{1}{\sqrt{2}} N(z)\begin{pmatrix}
			\pi^{\frac{1}{2}} s^{\frac{3}{8}} f_1(z)^{\frac{1}{4}} & 0 & \mathrm{i} \pi^{-\frac{1}{2}} s^{-\frac{3}{8}} f_1(z)^{-\frac{1}{4}} & 0 \\
			0 & \sqrt{2} & 0 & 0 \\
			\mathrm{i} \pi^{\frac{1}{2}} s^{\frac{3}{8}} f_1(z)^{\frac{1}{4}} & 0 & \pi^{-\frac{1}{2}} s^{-\frac{3}{8}} f_1(z)^{-\frac{1}{4}} & 0 \\
			0 & 0 & 0 & \sqrt{2}
		\end{pmatrix}.
	\end{equation}
	\begin{proposition}\label{pro:P1}
		$P^{(1)}(z)$ defined in \eqref{def:P1} solves RH problem \ref{rhp:P1}.
	\end{proposition}
	\begin{proof}
		First, we show $E_1(z)$ is analytic in $D(1, \varepsilon)$. From \eqref{def:E1}, the only possible jump is on $(1, 1 + \varepsilon)$. For $z \in (1, 1 + \varepsilon)$,
		\begin{align}
			E_{1,-}(z)^{-1}E_{1,+}(z)&=\frac 12  \begin{pmatrix}
				\pi^{-\frac 12} s^{-\frac 38} f_{1,-}(z)^{-\frac 14} & 0 &  -\ii \pi^{-\frac 12} s^{-\frac 38} f_{1,-}(z)^{-\frac 14} & 0\\
				0 & \sqrt{2} & 0 & 0\\
				-\ii \pi^{\frac 12} s^{\frac 38} f_{1,-}(z)^{\frac 14} & 0 & \pi^{\frac 12} s^{\frac 38} f_{1,-}(z)^{\frac 14} & 0\\
				0 & 0 & 0 & \sqrt{2}
			\end{pmatrix}
			\begin{pmatrix}0&0&1&0\\0&1&0&0\\-1&0&0&0\\0&0&0&1 \end{pmatrix}\nonumber\\
			&\quad\times \begin{pmatrix}
				\pi^{\frac 12} s^{\frac 38} f_{1,+}(z)^{\frac 14} & 0 &  \ii \pi^{-\frac 12} s^{-\frac 38} f_{1,+}(z)^{-\frac 14} & 0\\
				0 & \sqrt{2} & 0 & 0\\
				\ii \pi^{\frac 12} s^{\frac 38} f_{1,+}(z)^{\frac 14} & 0 & \pi^{-\frac 12} s^{-\frac 38} f_{1,+}(z)^{-\frac 14} & 0\\
				0 & 0 & 0 & \sqrt{2}
			\end{pmatrix}=I.
		\end{align}
		Moreover, as $z \to 1$, we have
		\begin{equation}\label{def:localE1}
			E_1(z) = E_1(1) + E_1'(1)(z-1) + \Boh\left((z-1)^2\right),
		\end{equation}
		where
		\begin{equation}
			E_1(1) = \begin{pmatrix}
				\pi^{\frac 12} s^{\frac 38} \left(1 - \frac{2\tilde{s}}{\sqrt{s}}\right)^{\frac 12} & -2^{-\frac 34} & 0 &-\ii 2^{-\frac 34}\\
				\pi^{\frac 12} s^{\frac 38} \left(1 - \frac{2\tilde{s}}{\sqrt{s}}\right)^{\frac 12} & 2^{-\frac 34} & 0 & \ii 2^{-\frac 34}\\
				0 & \ii 2^{-\frac 14} & \pi^{-\frac 12} s^{-\frac 38} \left(1 - \frac{2\tilde{s}}{\sqrt{s}}\right)^{-\frac 12}& 2^{-\frac 14}\\
				0 & \ii 2^{-\frac 14} &- \pi^{-\frac 12} s^{-\frac 38} \left(1 - \frac{2\tilde{s}}{\sqrt{s}}\right)^{-\frac 12} & 2^{-\frac 14}
			\end{pmatrix},
		\end{equation}
		and
		\begin{small}
			\begin{equation}
				E_1'(1) = \begin{pmatrix}
					\frac{\pi^{\frac 12} s^{\frac 38}}{4}\left(1 - \frac{2\tilde{s}}{\sqrt{s}}\right)^{\frac 12}+\frac{\pi^{\frac 12} s^{\frac 38}}{6 \left(1 - \frac{2\tilde{s}}{\sqrt{s}}\right)^{1/2}} & -\frac{3}{16 \cdot 2^{3/4}} & 0 & -\frac{3\ii}{16 \cdot 2^{3/4}}\\
					-\frac{\pi^{\frac 12} s^{\frac 38}}{4}\left(1 - \frac{2\tilde{s}}{\sqrt{s}}\right)^{\frac 12}+\frac{\pi^{\frac 12} s^{\frac 38}}{6\left(1 - \frac{2\tilde{s}}{\sqrt{s}}\right)^{1/2}}  & -\frac{5}{16 \cdot 2^{3/4}} & 0 & -\frac{5\ii}{16 \cdot 2^{3/4}}\\
					0 & \frac{5\ii}{16 \cdot 2^{1/4}} & \frac{\pi^{-\frac 12} s^{-\frac 38}}{4\left(1 - \frac{2\tilde{s}}{\sqrt{s}}\right)^{1/2}}+\frac{\pi^{-1/2} s^{-3/8} }{6\left(1 - \frac{2\tilde{s}}{\sqrt{s}}\right)^{3/2}}  & \frac{5}{16 \cdot 2^{1/4}}\\
					0 & -\frac{3\ii}{16 \cdot 2^{1/4}} & \frac{\pi^{-\frac 12} s^{-\frac 38}}{4\left(1 - \frac{2\tilde{s}}{\sqrt{s}}\right)^{1/2}}-\frac{\pi^{-1/2} s^{-3/8} }{6\left(1 - \frac{2\tilde{s}}{\sqrt{s}}\right)^{3/2}}  & -\frac{3}{16 \cdot 2^{1/4}}
				\end{pmatrix}.
			\end{equation}
		\end{small}
Therefore, $E_1(z)$ is indeed analytic in $D(1, \varepsilon)$. The jump condition of $P^{(1)}(z)$ in \eqref{eq:P1-jump} can be verified from the analyticity of $E_1(z)$ and the jump condition of $\Phi^{(\mathrm{Bes})}$ in \eqref{eq:jump:Bessel}.
		
Finally, we check the matching condition \eqref{eq:P1:matching}. From the asymptotic behavior of the Bessel parametrix at infinity in \eqref{eq:infty:Bessel}, it follows that, as $s \to +\infty$,
		\begin{equation}
			P^{(1)}(z) N(z)^{-1} = I + \frac{J^{(1)}_1(z)}{s^{3/4}} + \Boh(s^{-3/2}),
		\end{equation}
		where
		\begin{equation}\label{def:J1}
			J^{(1)}_1(z) = \frac{1}{8f_1(z)^{1/2}} N(z) \begin{pmatrix}
				-1 & 0 & 2\ii & 0\\
				0 & 0 & 0 & 0\\
				2\ii & 0 & 1 & 0\\
				0 & 0 & 0 & 0
			\end{pmatrix}N(z)^{-1},
		\end{equation}
  as required. This completes the proof of Proposition \ref{pro:P1}.
	\end{proof}
We note from \eqref{def:N} and \eqref{def:J1} that
	\begin{align*}
		J^{(1)}_1(z) =  \frac{\ii}{16f_1(z)^{1/2}} \begin{pmatrix}
			0 & 0 & (1-\sqrt{z})^{-\frac{1}{2}} & -z^{\frac{1}{2}}(1-\sqrt{z})^{-\frac{1}{2}}\\
			0 & 0 & z^{-\frac{1}{2}}(1-\sqrt{z})^{-\frac{1}{2}} & -(1-\sqrt{z})^{-\frac{1}{2}}\\
			3(1-\sqrt{z})^{\frac{1}{2}} & -3z^{\frac{1}{2}}(1-\sqrt{z})^{\frac{1}{2}} & 0 & 0\\
			-3z^{-\frac{1}{2}}(1-\sqrt{z})^{\frac{1}{2}} & -3(1-\sqrt{z})^{\frac{1}{2}} & 0 & 0
		\end{pmatrix}.
	\end{align*}
Thus, it is readily seen that 
	\begin{equation}\label{J1-1}
		J^{(1)}_1(z) = \frac{\Res_{\zeta = 1} J^{(1)}_1(\zeta)}{z-1} +\mathscr{J}_0+\mathscr{J}_1(z-1)+\Boh \left((z-1)^2\right) ,\quad z \to 1,
	\end{equation}
	where
	\begin{align}\label{def:J_1}
		\mathscr{J}_1=\begin{pmatrix}
			\begin{smallmatrix}
				0 & 0 & \frac{\ii }{1296\left(1 - 2\tilde{s}/\sqrt{s}\right)}-\frac{\ii }{96\left(1 - 2\tilde{s}/\sqrt{s}\right)^2} -\frac{\ii }{144\left(1 - 2\tilde{s}/\sqrt{s}\right)^3}&\frac{\ii }{128\left(1 - 2\tilde{s}/\sqrt{s}\right)}-\frac{\ii }{32\left(1 - 2\tilde{s}/\sqrt{s}\right)^2} -\frac{\ii }{144\left(1 - 2\tilde{s}/\sqrt{s}\right)^3}\\
				0 & 0 & \frac{3\ii }{128\left(1 - 2\tilde{s}/\sqrt{s}\right)}+\frac{\ii }{96\left(1 - 2\tilde{s}/\sqrt{s}\right)^2} +\frac{\ii }{144\left(1 - 2\tilde{s}/\sqrt{s}\right)^3} &-\frac{\ii }{128\left(1 - 2\tilde{s}/\sqrt{s}\right)}+\frac{\ii }{96\left(1- 2\tilde{s}/\sqrt{s}\right)^2} +\frac{\ii }{144\left(1 - 2\tilde{s}/\sqrt{s}\right)^3} \\
				\frac{\ii }{16\left(1 - 2\tilde{s}/\sqrt{s}\right)^2} & \frac{\ii (5-\frac{6\tilde{s}}{\sqrt{s}})}{32\left(1 - 2\tilde{s}/\sqrt{s}\right)^2} & 0 & 0\\
				\frac{\ii (1-\frac{6\tilde{s}}{\sqrt{s}})}{32\left(1 - 2\tilde{s}/\sqrt{s}\right)^2} & -\frac{\ii }{16\left(1 - 2\tilde{s}/\sqrt{s}\right)^2} & 0 & 0
			\end{smallmatrix}
		\end{pmatrix}.
	\end{align}
	
	\subsection{Final transformation}
	The final transformation is defined by
	\begin{equation}\label{def:R}
		R(z) = \begin{cases}
			S(z) P^{(0)}(z)^{-1}, & \quad z \in D(0, \varepsilon),\\
			S(z) P^{(1)}(z)^{-1}, & \quad z \in D(1, \varepsilon),\\
			S(z) N(z)^{-1}, & \quad \textrm{elsewhere}.
		\end{cases}
	\end{equation}
	Based on the RH problems for $S$, $N$, $P^{(0)}$ and  $P^{(1)}$, it can be verified  that both $0$ and $1$ are actually removable singularities of $R$. Consequently, $R$ satisfies the following RH problem.
	
	\begin{figure}[t]
		\begin{center}
			\setlength{\unitlength}{1truemm}
			\begin{picture}(100,70)(-13,2)
				
				\put(20,35){\line(-1,-1){25}}
				\put(20,45){\line(-1,1){25}}
				
				\put(60,45){\line(1,1){25}}
				\put(60,35){\line(1,-1){25}}

				\put(10,55){\thicklines\vector(1,-1){1}}
				\put(10,25){\thicklines\vector(1,1){1}}
				\put(70,25){\thicklines\vector(1,-1){1}}
				\put(70,55){\thicklines\vector(1,1){1}}
				\put(25,47){\thicklines\vector(1,0){1}}
				\put(55,47){\thicklines\vector(1,0){1}}

				\put(-12,11){$\Gamma_4^{(1)}$}
				\put(-12,67){$\Gamma_2^{(1)}$}
				
				\put(85,11){$\Gamma_5^{(1)}$}
				\put(85,67){$\Gamma_1^{(1)}$}

				\put(25,40){\thicklines\circle*{1}}
				\put(55,40){\thicklines\circle*{1}}
				
				\put(25,40){\circle{20}}
				\put(55,40){\circle{20}}
				
				\put(24,36.3){$0$}
				\put(54,36.3){$1$}
				\put(20,29){$\partial D(0, \varepsilon)$}
				\put(49,29){$\partial D(1, \varepsilon)$}
			\end{picture}
			\caption{The contour $\Gamma_{R}$ of the RH problem for $R$.}
			\label{fig:R}
		\end{center}
	\end{figure}
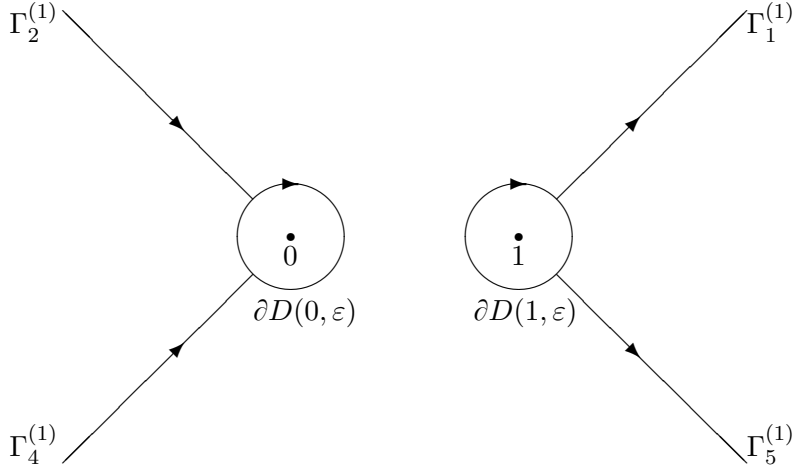
	\begin{rhp}\label{rhp:R}
		\hfill
		\begin{itemize}
			\item [\rm{(a)}] $R(z)$ is defined and analytic in $\mathbb{C} \setminus \Gamma_{R}$, where
			\begin{equation}
				\Gamma_{R}:=\Gamma_T \cup \partial D(0,\varepsilon) \cup \partial D(1,\varepsilon) \setminus \{\mathbb{R}
				\cup D(0,\varepsilon) \cup D(1,\varepsilon) \};
			\end{equation}
			see Figure \ref{fig:R} for an illustration.
			\item [\rm{(b)}] For $z \in \Gamma_{R}$, we have
			\begin{equation}\label{eq:Rjump}
				R_+(z) = R_-(z) J_R (z),
			\end{equation}
			where
			\begin{equation}
				J_R(z) = \begin{cases}
					P^{(0)}(z) N(z)^{-1}, & \quad z \in \partial D(0, \varepsilon),\\
					P^{(1)}(z) N(z)^{-1}, & \quad z \in \partial D(1, \varepsilon),\\
					N(z) J_S(z) N(z)^{-1}, & \quad z \in \Gamma_{R} \setminus \{\partial D(0, \varepsilon) \cup \partial D(1, \varepsilon)\},
				\end{cases}
			\end{equation}
			with $J_S(z)$ defined in \eqref{def:JS}.
			\item [\rm{(c)}] As $z \to \infty$, we have
			\begin{equation}\label{eq:asyR}
				R(z) = I + \frac{R^{(1)}}{z} + \Boh (z^{-2}),
			\end{equation}
			where $R^{(1)}$ is independent of $z$.
		\end{itemize}
	\end{rhp}
	Since the jump matrix $J_S(z)$ of $S$ given in \eqref{def:JS} tend to the identity matrix exponentially fast except for $z\in \Gamma_0^{(1)} \cup \Gamma_3^{(1)}$ as $s \to +\infty$, from the matching conditions \eqref{eq:P0:matching} and \eqref{eq:P1:matching}, we have
	\begin{equation}
		J_R(z) = I + \Boh (s^{-\frac 34}), \qquad s \to +\infty.
	\end{equation}
	By a standard argument \cite{Deift1999, Deift1993}, we conclude that
	\begin{equation}\label{def:RR}
		R(z) = I + \frac{R_1(z)}{s^{3/4}} + \Boh (s^{-3/2}), \quad s \to +\infty,
	\end{equation}
	uniformly for $z\in \mathbb{C}\setminus \Gamma_{R}$. Moreover, inserting the above expansion into \eqref{eq:Rjump}, it follows that function $R_1$ is analytic in $\mathbb{C} \setminus (\partial D(0, \varepsilon) \cup \partial D(1, \varepsilon))$ with asymptotic behavior $\Boh(1/z)$ as $z\to \infty$, and satisfies
	$$
	R_{1,+}(z)-R_{1,-}(z)= \left\{
	\begin{array}{ll}
		J^{(0)}_1(z), & \hbox{$z\in \partial D(0,\varepsilon)$,} \\
		J^{(1)}_1(z), & \hbox{$z\in \partial D(1,\varepsilon)$,}
	\end{array}
	\right.
	$$
	where the functions $J^{(0)}_1(z)$ and $J^{(1)}_1(z)$ are given in \eqref{def:J0} and \eqref{def:J1}, respectively.
	By Cauchy's residue theorem, we have
	\begin{align}\label{J1--1}
		R_1(z) &= \frac{1}{2 \pi \ii} \oint_{\partial D(0, \varepsilon)} \frac{J_1^{(0)} (\zeta)}{z-\zeta} \ud \zeta + \frac{1}{2 \pi \ii} \oint_{\partial D(1, \varepsilon)} \frac{J_1^{(1)} (\zeta)}{z-\zeta} \ud \zeta \nonumber\\
		&=\begin{cases}
			\frac{\Res_{\zeta = 0} J^{(0)}_1(\zeta)}{z} + \frac{\Res_{\zeta = 1} J^{(1)}_1(\zeta)}{z-1}, &\quad z \in \mathbb{C} \setminus \{D(0, \varepsilon) \cup D(1, \varepsilon)\},\\
			\frac{\Res_{\zeta = 0} J^{(0)}_1(\zeta)}{z} + \frac{\Res_{\zeta = 1} J^{(1)}_1(\zeta)}{z-1} - J^{(0)}_1(z), &\quad z \in D(0, \varepsilon),\\
			\frac{\Res_{\zeta = 0} J^{(0)}_1(\zeta)}{z} + \frac{\Res_{\zeta = 1} J^{(1)}_1(\zeta)}{z-1} - J^{(1)}_1(z), &\quad z \in D(1, \varepsilon).
		\end{cases}
	\end{align}
	In view of \eqref{J1--1} and \eqref{J1-1}, it follows that
	\begin{equation}\label{eq:R1'1}
		R_1'(1) = -\mathscr{J}_1-\Res_{\zeta = 0} J^{(0)}_1(\zeta),
	\end{equation}
where $\mathscr{J}_1$ is given in \eqref{def:J_1}.

	\section{Large $s$ asymptotic analysis of the RH problem for $X$ with $0<\gamma<1$}
	\label{sec:AsyXgamma}
	
	If $0< \gamma <1$, the matrix-valued function $X$ has a non-trivial jump
	$\begin{pmatrix}
		1&0&1-\gamma&0\\0&1&0&0\\0&0&1&0\\0&0&0&1 \end{pmatrix}$ over the interval $(0,s)$. This extra jump will lead to a completely different asymptotic analysis of the RH problem for $X$ as $s \to +\infty$, which will be conducted in this section.
	
	\subsection{First transformation: $X \mapsto \what T$}
	This transformation is a rescaling and normalization of the RH problem for $X$, and it is defined by
	\begin{align}\label{def:XToTgamma}
		\what T(z) &= \operatorname{diag}\left(s^{-\frac{1}{8}}, s^{\frac{3}{8}}, s^{-\frac{3}{8}}, s^{\frac{1}{8}}\right) X_0^{-1}  X(sz)\nonumber\\
		&\quad \times\diag \left(
		e^{\theta_1(\sqrt{sz})-\tau \sqrt{sz}}, e^{\theta_2(\sqrt{sz})+ \tau \sqrt{sz}}, e^{-\theta_1(\sqrt{sz})-\tau
			\sqrt{sz}},e^{-\theta_2(\sqrt{sz}) + \tau \sqrt{sz}} \right),
	\end{align}
	where $X_0, \theta_1$ and $\theta_2$ are given in \eqref{def:X0}, \eqref{def:theta1} and \eqref{def:theta2}, respectively. In view of the facts that
	\begin{equation}\label{eq:thetairelations}
		\begin{aligned}
			\theta_{1,+}(\sqrt{sx})+\theta_{1,-}(\sqrt{sx})= 0, &\qquad x>0,\\
			-\theta_{1,\pm}(\sqrt{sx})+\theta_{2,\mp}(\sqrt{sx})=0, &\qquad x<0,
		\end{aligned}
	\end{equation}
	and RH problem \ref{rhp:X} for $X$, it is readily seen that $\what T$ defined in \eqref{def:XToTgamma} satisfies the following RH problem.
	\begin{rhp}\label{rhp:hatT}
		\hfill
		\begin{enumerate}
			\item[\rm (a)] $\what T(z)$ is defined and analytic in $\mathbb{C} \setminus \Gamma_{\what T}$, where
			\begin{equation}\label{def:gammahatT}
				\Gamma_{\what T}:=\cup^5_{j=0}\Gamma_j^{(1)}\cup [0,1],
			\end{equation}
			and where the contours $\Gamma_j^{(1)}$, $j=0,1,\ldots,5$, are defined in \eqref{def:Gammajs} with $s=1$.
			
			\item[\rm (b)]  $\what T(z)$ satisfies
			$
				\what T_+(z)=\what T_-(z)J_{\what T}(z)
			$ for $z\in \Gamma_{\what T}$,
			where
			\begin{equation}\label{def:JhatT}
				J_{\what T}(z):=\left\{
				\begin{array}{ll}
					\begin{pmatrix}0&0&1&0\\0&1&0&0\\-1&0&0&0\\0&0&0&1 \end{pmatrix}, &  \hbox{$z\in \Gamma_0^{(1)}$,} \\
					I+e^{2\theta_1(\sqrt{sz})}E_{3,1},  &   \hbox{$z\in \Gamma_1^{(1)}$,} \\
					I+e^{\nu \pi \ii+\theta_1(\sqrt{sz})-\theta_2(\sqrt{sz})-2\tau \sqrt{sz}} E_{2,1}-e^{\nu \pi \ii+\theta_1(\sqrt{sz})-\theta_2(\sqrt{sz})+2\tau \sqrt{sz}}E_{3,4},  &  \hbox{$z\in \Gamma_2^{(1)}$,} \\
					\begin{pmatrix}0&-\ii&0&0\\ -\ii &0&0&0\\0&0&0&\ii \\0&0&\ii&0 \end{pmatrix}, &   \hbox{$z\in \Gamma_3^{(1)}$,} \\
					I-e^{-\nu \pi \ii+\theta_1(\sqrt{sz})-\theta_2(\sqrt{sz})-2\tau \sqrt{sz}} E_{2,1}+e^{-\nu \pi \ii+\theta_1(\sqrt{sz})-\theta_2(\sqrt{sz})+2\tau \sqrt{sz}}E_{3,4}, &   \hbox{$z\in \Gamma_4^{(1)}$,} \\
					I+e^{2\theta_1(\sqrt{sz})} E_{3,1}, &   \hbox{$z\in \Gamma_5^{(1)}$,} \\
					\msf J_{\msf R}(z), &   \hbox{$z \in (0,1)$,}
				\end{array}
				\right.
			\end{equation}
			with
			\begin{align}\label{def:msfJR}
				\msf J_{\msf R}(z):
				=\begin{pmatrix}e^{\theta_{1,+}(\sqrt{sz})-\theta_{1,-}(\sqrt{sz})}&0&1-\gamma&0
					\\0&1&0&0\\0&0&e^{\theta_{1,-}(\sqrt{sz})-\theta_{1,+}(\sqrt{sz})}&0\\0&0&0&1
				\end{pmatrix}.
			\end{align}

			\item[\rm (c)]As $z \to \infty$ with $z\in \mathbb{C} \setminus \Gamma_{\what T}$, we have
			\begin{align}\label{eq:asyhatT}
				\what T(z)& =\left( I+ \frac{\what T^{(1)}}{z} + \Boh(z^{-2}) \right) \operatorname{diag}\left(z^{\frac{1}{4}}, z^{-\frac{1}{4}},z^{\frac{1}{4}}, z^{-\frac{1}{4}}\right)    \operatorname{diag}\left(\begin{pmatrix} 1 & -1 \\ 1 & 1 \end{pmatrix},\begin{pmatrix} 1 & 1 \\ -1 & 1 \end{pmatrix}\right) \nonumber\\ 
				&~~ \times \operatorname{diag}\left((-\sqrt{z})^{-\frac{1}{4}}, z^{-\frac{1}{8}},(-\sqrt{z})^{\frac{1}{4}}, z^{\frac{1}{8}}\right) A,
			\end{align}
			where $\what T^{(1)}$ is independent of $z$ and $A$ is defined in \eqref{def:A}.
			\item[\rm (d)]
			As $z \to 1$, we have $\what T(z)=\Boh(\ln(z-1))$.
   \item [\rm(e)]
 As $z \to  0$, we have
 \begin{equation}
			\what T(z) = \Boh(1)\cdot \begin{pmatrix}
						z^{\frac{2\nu-1}{4}} & \delta_{\nu}(z) & (1-\gamma ) e^{-\nu\pi \mathrm{i}}\delta_{\nu}(z) & 0 \\
						0 & z^{-\frac{2\nu-1}{4}} & 0 & 0 \\
						0 & 0 & z^{-\frac{2\nu-1}{4}} & 0 \\
						0 & 0 & -\delta_{\nu}(z) & z^{\frac{2\nu-1}{4}}
					\end{pmatrix}\cdot \left(I+\Boh(z)\right), \quad z\in \Omega_1^{(1)},
		\end{equation}
  with $\delta_{\nu}(z)$ defined in \eqref{def:delta}.
		\end{enumerate}
	\end{rhp}
	
	\subsection{Second transformation: $\what T \mapsto \what S$}
On account of the definitions of $\theta_1(z)$ and $\theta_2(z)$ given in \eqref{def:theta1} and \eqref{def:theta2}, it is readily seen that $J_{\what T}(z)$ in \eqref{def:JhatT} tends to $I$ exponentially fast as $s\to +\infty$ for $z\in \Gamma_{\what T} \setminus \mathbb{R}$. Moreover, the $(1,1)$, $(3,3)$ entries of $J_{\msf R}$ in \eqref{def:msfJR} are highly oscillatory for large positive $s$. Thus, in contrast to the transformation \eqref{def:TtoS}, we need to open lens around the interval $(0,1)$ in what follows. The idea is to remove the highly oscillatory terms of $J_{\what T}$ with the cost of creating extra jumps that tend to the identity matrices on some new contours. To proceed, we observe from \eqref{eq:thetairelations} and \eqref{def:msfJR} that
	\begin{equation}
		\msf J_{\msf R}(z)=\msf J_{1,-}(z)
		\begin{pmatrix}
			0&0&1-\gamma&0
			\\
			0&1&0&0
			\\
			\frac{1}{\gamma-1}&0&0&0
			\\0&0&0&1
		\end{pmatrix}\msf J_{1,+}(z), \qquad z\in(0,1),
	\end{equation}
	where
	\begin{align}
		\msf J_1(z)&=
		I+\frac{e^{2\theta_{1}(\sqrt{sz})}}{1-\gamma}E_{3,1}.
		\label{def:msfJ1}
	\end{align}
Let $\Omega_{\msf R, \pm}$ be the lenses on the $\pm$-side of $(0,1)$, as shown in Figure \ref{fig:lenses}. The second transformation in this case is defined by
	\begin{equation}\label{def:hatS}
		\what S = \what T\left\{
		\begin{array}{ll}
			\msf J_1(z)^{-1}, & \hbox{$z\in \Omega_{\msf R,+}$,}
			\\
			\msf J_1(z), & \hbox{$z\in \Omega_{\msf R,-}$,}
			\\
			I, & \hbox{elsewhere.}
		\end{array}
		\right.
	\end{equation}
It is then readily seen from RH problem \ref{rhp:hatT} for $\what T$ that $\what S$ satisfies the following RH problem.
	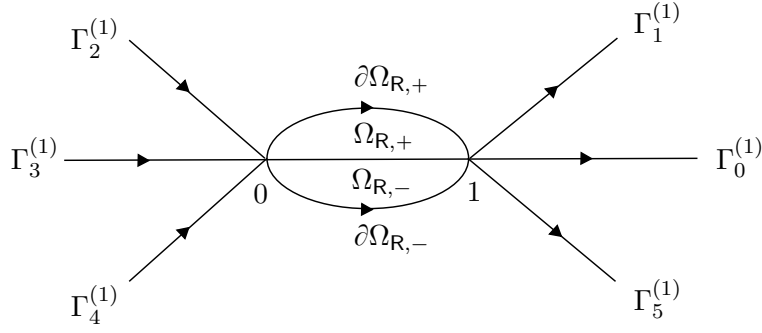
\begin{figure}[t]
		\center
		
		\tikzset{every picture/.style={line width=0.55pt}} 
		
		\begin{tikzpicture}[x=0.75pt,y=0.75pt,yscale=-1,xscale=1]
			
			\draw    (158,118.85) -- (474.11,117.86) ;
			\draw  [fill={rgb, 255:red, 0; green, 0; blue, 0 }  ,fill opacity=1 ] (200.13,118.98) -- (194.9,121.7) -- (195.03,116.05) -- cycle ;
			\draw   (259.24,117.86) .. controls (259.24,103.9) and (281.76,92.58) .. (309.55,92.58) .. controls (337.34,92.58) and (359.87,103.9) .. (359.87,117.86) .. controls (359.87,131.82) and (337.34,143.14) .. (309.55,143.14) .. controls (281.76,143.14) and (259.24,131.82) .. (259.24,117.86) -- cycle ;
			\draw    (190.45,58.43) -- (258.2,117.86) ;
			\draw    (359.44,117.86) -- (433.21,179.74) ;
			\draw    (434.16,57.37) -- (360.47,117.86) ;
			\draw    (259.24,117.86) -- (190.45,179.74) ;
			\draw  [fill={rgb, 255:red, 0; green, 0; blue, 0 }  ,fill opacity=1 ] (421.2,118) -- (415.97,120.71) -- (416.1,115.06) -- cycle ;
			\draw  [fill={rgb, 255:red, 0; green, 0; blue, 0 }  ,fill opacity=1 ] (311.7,143.64) -- (306.47,146.35) -- (306.6,140.7) -- cycle ;
			\draw  [fill={rgb, 255:red, 0; green, 0; blue, 0 }  ,fill opacity=1 ] (311.7,92.35) -- (306.47,95.07) -- (306.6,89.42) -- cycle ;
			\draw  [fill={rgb, 255:red, 0; green, 0; blue, 0 }  ,fill opacity=1 ] (220.1,84.12) -- (214.3,82.82) -- (218.34,78.68) -- cycle ;
			\draw  [fill={rgb, 255:red, 0; green, 0; blue, 0 }  ,fill opacity=1 ] (220.13,152.78) -- (218.28,158.18) -- (214.31,153.98) -- cycle ;
			\draw  [fill={rgb, 255:red, 0; green, 0; blue, 0 }  ,fill opacity=1 ] (406.11,157.03) -- (400.27,155.95) -- (404.13,151.67) -- cycle ;
			\draw  [fill={rgb, 255:red, 0; green, 0; blue, 0 }  ,fill opacity=1 ] (404.01,81.76) -- (402.16,87.17) -- (398.19,82.98) -- cycle ;

			\draw (251.15,129.58) node [anchor=north west][inner sep=0.75pt]   [align=left] {0};
			\draw (357.56,128.59) node [anchor=north west][inner sep=0.75pt]   [align=left] {1};
			\draw (482,106) node [anchor=north west][inner sep=0.75pt]   [align=left] {$\Gamma_0^{(1)}$};
			\draw (441,37.86) node [anchor=north west][inner sep=0.75pt]   [align=left] {$\Gamma_1^{(1)}$};
			\draw (160,46) node [anchor=north west][inner sep=0.75pt]   [align=left] {$\Gamma_2^{(1)}$};
			\draw (130,106) node [anchor=north west][inner sep=0.75pt]   [align=left] {$\Gamma_3^{(1)}$};
			\draw (160,180) node [anchor=north west][inner sep=0.75pt]   [align=left] {$\Gamma_4^{(1)}$};
			\draw (441,177) node [anchor=north west][inner sep=0.75pt]   [align=left] {$\Gamma_5^{(1)}$};
			\draw (301,99) node [anchor=north west][inner sep=0.75pt]   [align=left] {$\Omega_{\msf R,+}$};
			\draw (300,122) node [anchor=north west][inner sep=0.75pt]   [align=left] {$\Omega_{\msf R,-}$};
			\draw (301,70) node [anchor=north west][inner sep=0.75pt]   [align=left] {$\partial \Omega_{\msf R,+}$};
			\draw (301,150) node [anchor=north west][inner sep=0.75pt]   [align=left] {$\partial \Omega_{\msf R,-}$};
		\end{tikzpicture}
		\caption{Regions $\Omega_{\msf R, \pm}$ and the jump contours of the RH problem for $\what S$.}
		\label{fig:lenses}
	\end{figure}
	
	\begin{rhp}\label{rhp:hatS}
		\hfill
		\begin{enumerate}
			\item[\rm (a)] $\what S(z)$ is defined and analytic in $\mathbb{C} \setminus \Gamma_{\what S}$, where
			\begin{equation}\label{def:gammahatS}
				\Gamma_{\what S}:=\cup^5_{j=0}\Gamma_j^{(1)}\cup [0,1] \cup \partial \Omega_{\msf R, \pm};
			\end{equation}
			see Figure \ref{fig:lenses} for an illustration.
			
			\item[\rm (b)]  $\what S$ satisfies 
			$
				\what S_+(z)=\what S_-(z)J_{\what S}(z)
			$ for $z\in \Gamma_{\what S}$,
			where
			\begin{equation}\label{def:JhatS}
				J_{\what S}(z):=\left\{
				\begin{array}{ll}
					J_{\what T}(z), & \qquad \hbox{$z\in \cup^5_{j=0}\Gamma_j^{(1)}$,} \\
					\msf J_1(z), & \qquad  \hbox{$z\in \partial \Omega_{\msf R,+}$,} \\
					\msf J_1(z), & \qquad  \hbox{$z\in \partial \Omega_{\msf R,-}$,} \\
					\begin{pmatrix}
						0&0&1-\gamma&0
						\\
						0&1&0&0
						\\
						\frac{1}{\gamma-1}&0&0&0
						\\0&0&0&1
					\end{pmatrix}, & \qquad  \hbox{$z \in (0,1)$,}
				\end{array}
				\right.
			\end{equation}
			and where the functions $J_{\what T}(z)$ and $\msf J_1(z)$ are defined in \eqref{def:JhatT} and \eqref{def:msfJ1}, respectively.
			
			\item[\rm (c)]As $z \to \infty$ with $z\in \mathbb{C} \setminus \Gamma_{\what S}$, we have
			\begin{align}\label{eq:asyhatS}
				\what S(z)=&\left( I+ \frac{\what T^{(1)}}{z} + \Boh(z^{-2}) \right) \operatorname{diag}\left(z^{\frac{1}{4}}, z^{-\frac{1}{4}},z^{\frac{1}{4}}, z^{-\frac{1}{4}}\right)    \operatorname{diag}\left(\begin{pmatrix} 1 & -1 \\ 1 & 1 \end{pmatrix},\begin{pmatrix} 1 & 1 \\ -1 & 1 \end{pmatrix}\right)  \nonumber\\ 
				&\quad \times\operatorname{diag}\left((-\sqrt{z})^{-\frac{1}{4}}, z^{-\frac{1}{8}},(-\sqrt{z})^{\frac{1}{4}}, z^{\frac{1}{8}}\right) A,
			\end{align}
			where $\what T^{(1)}$ is given in \eqref{eq:asyhatT} and $A$ is defined in \eqref{def:A}.
			\item[\rm (d)]
		As $z \to  1$ or $z \to  0$ outside the lens, $\what S(z)$ has the same local behaviors as $\what T(z)$.
		\end{enumerate}
	\end{rhp}
	
	\subsection{Global parametrix}
	As $s\to +\infty$, it is now readily seen that all the jump matrices of $\what S$ tend to the identity matrices exponentially fast except for those along $\mathbb{R}$. This leads to the following global parametrix.
	
	\begin{rhp}\label{rhp:hatN}
		\hfill
		\begin{enumerate}
			\item[\rm (a)] $\what N(z)$ is defined and analytic in $\mathbb{C} \setminus \mathbb R$.

			\item[\rm (b)] For $x\in \mathbb{R}$, $\what N(x)$ satisfies the jump condition
			\begin{equation}\label{eq:hatN-jump}
				\what N_+(x)=\what N_-(x)\left\{ \begin{array}{ll}
					\begin{pmatrix}
						0&0&1&0\\0&1&0&0\\-1&0&0&0\\0&0&0&1
					\end{pmatrix}, & \quad \hbox{$x>1$,}
					\\
					\begin{pmatrix}
						0&0&1-\gamma&0
						\\
						0&1&0&0
						\\
						\frac{1}{\gamma-1}&0&0&0
						\\0&0&0&1
					\end{pmatrix}, & \quad  \hbox{$x \in (0,1)$,}
					\\
					\begin{pmatrix}0&-\ii&0&0\\-\ii &0&0&0\\0&0&0&\ii\\0&0&\ii&0 \end{pmatrix}, & \quad  \hbox{$x<0$.}
				\end{array}
				\right.
			\end{equation}

			\item[\rm (c)]As $z \to \infty$,  we have
			\begin{align} \label{eq:asyhatN}
				\what N(z) & =\left( I+ \frac{\what N^{(1)}}{z}+\Boh(z^{-2}) \right) \operatorname{diag}\left(z^{\frac{1}{4}}, z^{-\frac{1}{4}},z^{\frac{1}{4}}, z^{-\frac{1}{4}}\right)    \operatorname{diag}\left(\begin{pmatrix} 1 & -1 \\ 1 & 1 \end{pmatrix},\begin{pmatrix} 1 & 1 \\ -1 & 1 \end{pmatrix}\right)  \nonumber\\ 
				&\quad \times \operatorname{diag}\left((-\sqrt{z})^{-\frac{1}{4}}, z^{-\frac{1}{8}},(-\sqrt{z})^{\frac{1}{4}}, z^{\frac{1}{8}}\right) A,
			\end{align}
			where $\what N^{(1)}$ is independent of $z$ and $A$ is defined in \eqref{def:A}.
		\end{enumerate}
	\end{rhp}
To solve the above RH problem,	we set
	\begin{equation}\label{def:lambda}
		\lambda(\zeta):= \left(\frac{\zeta-\ii}{\zeta+\ii}\right)^{\beta},\qquad \zeta \in \mathbb{C} \setminus [-\ii,\ii],
	\end{equation}
	where $\beta=\ln (1-\gamma)/2\pi \ii$  and the branch cut is chosen such that $\lambda(\zeta) \to  1$ as $\zeta \to \infty$ with the orientation from $\ii$ to $-\ii$. We then further define
	\begin{equation}\label{def:di}
		d_1(z)=\lambda((-\sqrt{z})^{\frac12}),\quad d_2(z)=\lambda(z^{\frac14}), \quad d_3(z)=\lambda(-(-\sqrt{z})^{\frac12}),\quad d_4(z)=\lambda(-z^{\frac14}).
	\end{equation}
Some properties of these functions are collected in the proposition below.
	\begin{proposition}\label{prop:di}
		The functions  $d_i$, $i=1,\ldots,4$, in \eqref{def:di} satisfy the following properties.
		\begin{enumerate}
			\item[\rm (i)] $d_1(z)$ and $d_3(z)$ are analytic in $\mathbb{C} \setminus \mathbb{R}$. $d_2(z)$ and $d_4(z)$ are analytic in $\mathbb{C} \setminus (-\infty,0]$. Moreover, we have
			\begin{align}
				d_{1,\pm}(x)&=d_{2,\mp}(x), \quad d_{3,\pm}(x)=d_{4,\mp}(x), && x<0, \label{eq:dijump1}
				\\
				d_{1,\pm}(x)&=d_{3,\mp}(x)e^{-2\beta \pi \ii},  && 0<x<1.
				\\
				d_{1,\pm}(x)&=d_{3,\mp}(x),  && x>1.
			\end{align}
			\item[\rm (ii)] As $z\to \infty$, $\Im z>0$, we have
			\begin{align}
				d_1(z)&=1-\frac{2\beta\ii}{(-\sqrt{z})^{1/2}}+\frac{2\beta^2}{\sqrt{z}}+\frac{2(\beta+2\beta^3)\ii}{3(-\sqrt{z})^{3/2}}+\Boh(z^{-1}),
				\\
				d_2(z)&=1-\frac{2\beta \ii}{z^{1/4}}-\frac{2\beta^2}{\sqrt{z}}+\frac{2(\beta+2\beta^3)\ii}{3z^{3/4}}+\Boh(z^{-1}),
				\\
				d_3(z)&=1+\frac{2\beta\ii}{(-\sqrt{z})^{1/2}}+\frac{2\beta^2}{\sqrt{z}}-\frac{2(\beta+2\beta^3)\ii}{3(-\sqrt{z})^{3/2}}+\Boh(z^{-1}),
				\\
				d_4(z)&=1+\frac{2\beta \ii}{z^{1/4}}-\frac{2\beta^2}{\sqrt{z}}-\frac{2(\beta+2\beta^3)\ii}{3z^{3/4}}+\Boh(z^{-1}).
			\end{align}
			\item[\rm (iii)] As $z \to 0$, $\Im z>0$, we have
			\begin{align}
				d_1(z) &= e^{-\beta \pi \ii}\bigg(1 + 2 \ii \beta (-\sqrt{z})^{\frac12}+ 2 \beta^2 \sqrt{z} - \frac{2 \ii (\beta + 2 \beta^3)}{3} (-\sqrt{z})^{\frac 32} + \frac{4\beta^2+2\beta^4}{3}z 
				\nonumber \\
				&\quad +\frac{2\ii \beta(3+10\beta^2+2\beta^4)}{15}(-\sqrt{z})^{\frac 52}+\frac{2\beta^2(23+20\beta^2+2\beta^4)}{45}z^{\frac{3}{2}}\nonumber \\
				&\quad-\frac{2\ii \beta(45+196\beta^2+70\beta^4+4\beta^6)}{315}(-\sqrt{z})^{\frac 72}+\Boh(z^2)\bigg), \label{eq:d10}
				\\
				d_2(z) &= e^{-\beta \pi \ii}\bigg(1 + 2 \ii \beta z^{1/4} - 2 \beta^2  \sqrt{z} - \frac{2 \ii (\beta + 2 \beta^3)}{3} z^{\frac 34} + \frac{4\beta^2+2\beta^4}{3}z \nonumber \\
				&\quad+\frac{2\ii \beta(3+10\beta^2+2\beta^4)}{15}z^{\frac 54}-\frac{2\beta^2(23+20\beta^2+2\beta^4)}{45}z^{\frac{3}{2}}\nonumber \\
				&\quad-\frac{2\ii \beta(45+196\beta^2+70\beta^4+4\beta^6)}{315}z^{\frac 74}+\Boh(z^2)\bigg), \label{eq:d20}
				\\
				d_3(z) &= e^{\beta \pi \ii}\bigg(1 - 2 \ii \beta (-\sqrt{z})^{\frac12}+ 2 \beta^2 \sqrt{z} + \frac{2 \ii (\beta + 2 \beta^3)}{3} (-\sqrt{z})^{\frac 32} + \frac{4\beta^2+2\beta^4}{3}z \nonumber \\
				&\quad -\frac{2\ii \beta(3+10\beta^2+2\beta^4)}{15}(-\sqrt{z})^{\frac 52}+\frac{2\beta^2(23+20\beta^2+2\beta^4)}{45}z^{\frac{3}{2}}\nonumber \\ 
				&\quad+\frac{2\ii \beta(45+196\beta^2+70\beta^4+4\beta^6)}{315}(-\sqrt{z})^{\frac 72}+\Boh(z^2)\bigg), \label{eq:d30}
				\\
				d_4(z) &= e^{\beta \pi \ii}\bigg(1 - 2 \ii \beta z^{1/4} - 2 \beta^2  \sqrt{z} + \frac{2 \ii (\beta + 2 \beta^3)}{3} z^{\frac 34} + \frac{4\beta^2+2\beta^4}{3}z \nonumber \\
				&\quad-\frac{2\ii \beta(3+10\beta^2+2\beta^4)}{15}z^{\frac 54}-\frac{2\beta^2(23+20\beta^2+2\beta^4)}{45}z^{\frac{3}{2}}\nonumber \\
				&\quad+\frac{2\ii \beta(45+196\beta^2+70\beta^4+4\beta^6)}{315}z^{\frac 74}+\Boh(z^2)\bigg). \label{eq:d40}
			\end{align}
			\item[\rm (iv)] As $z \to 1$, $\Im z>0$, we have
			\begin{align}
				d_1(z) & = 8^{\beta} (z-1)^{-\beta} \left(1+\frac{\beta}{2}(z-1)+\Boh \left((z-1)^2\right)\right), \label{eq:d31}
				\\
				d_2(z) & = e^{-\frac{\beta \pi \ii}{2}}\left(1 + \frac{\beta \ii}{2}(z-1) + \Boh\left((z-1)^2\right)\right), \label{eq:d21}
				\\
				d_3(z) & = 8^{-\beta} (z-1)^{\beta} \left(1-\frac{\beta}{2}(z-1)+\Boh \left((z-1)^2 \right)\right), 
				\\
				d_4(z) & = e^{\frac{\beta \pi \ii}{2}}\left(1 - \frac{\beta \ii}{2}(z-1) + \Boh\left((z-1)^2\right)\right). \label{eq:d41}
			\end{align}
		\end{enumerate}
	\end{proposition}
	\begin{proof}
		From the definition of $\lambda$ given in \eqref{def:lambda}, it is readily seen that
		\begin{equation}
			\lambda_{+}(\zeta)=\lambda_{-}(\zeta) e^{-2\beta \pi \ii}, \qquad \zeta \in (-\ii,\ii).
		\end{equation}
	Moreover, we have 
		\begin{align}
			\lambda(\zeta)&=1-\frac{2 \beta \ii}{\zeta}-\frac{2\beta^2}{\zeta^2}+\Boh(\zeta^{-3}), \quad \zeta \to \infty,\\
			\lambda(\zeta)&= e^{-\frac{\beta \pi \ii}{2}} \left(1 + \beta \ii (\zeta-1) + \Boh\left((\zeta-1)^2\right)\right), \quad \zeta \to 1,
		\end{align}
  and as $\zeta \to 0$,
  \begin{equation}
      \lambda(\zeta)=\begin{cases}
				e^{-\beta \pi \ii}\left(1 + 2\ii \beta \zeta-2\beta^2 \zeta^2 -\frac{2 \ii (\beta + 2\beta^3)}{3} \zeta^3 +\frac{4\beta^2+2\beta^2}{3}\zeta^4+ \Boh(\zeta^5)\right), & \quad \Re{\zeta}>0,\\
				e^{\beta \pi \ii}\left(1 + 2\ii \beta \zeta-2\beta^2 \zeta^2 -\frac{2 \ii (\beta + 2\beta^3)}{3} \zeta^3 +\frac{4\beta^2+2\beta^2}{3}\zeta^4+ \Boh(\zeta^5)\right), & \quad \Re{\zeta}<0.
    \end{cases}
  \end{equation}
This, together with \eqref{def:di}, gives us the claims in the proposition after straightforward calculations.
	\end{proof}
We are now ready to solve RH problem \ref{rhp:hatN} for $\what N$.
	\begin{proposition}\label{prop:solhatN}
		Let $d_i$, $i=1,\ldots,4$, be the functions defined in \eqref{def:di}. A solution of RH problem \ref{rhp:hatN} is given by
		\begin{align}\label{eq:hatNexp}
			\begin{aligned}
				\what N(z)& = C_0\operatorname{diag}\left(z^{\frac{1}{4}}, z^{-\frac{1}{4}},z^{\frac{1}{4}}, z^{-\frac{1}{4}}\right)    \operatorname{diag}\left(\begin{pmatrix} 1 & -1 \\ 1 & 1 \end{pmatrix},\begin{pmatrix} 1 & 1 \\ -1 & 1 \end{pmatrix}\right) \\ 
				&\quad \times \operatorname{diag}\left((-\sqrt{z})^{-\frac{1}{4}}, z^{-\frac{1}{8}},(-\sqrt{z})^{\frac{1}{4}}, z^{\frac{1}{8}}\right) A  \operatorname{diag}\left(d_1(z),d_2(z),d_3(z),d_4(z)\right),
			\end{aligned}
		\end{align}
		where
		\begin{equation}\label{def:C_0}
			C_0=\begin{pmatrix}
1&2\beta^2&0&-2\beta\\0&1&0&0\\2\beta&\frac{2(\beta+2\beta^3)}{3}&1&-2\beta^2\\0&-2\beta&0&1
			\end{pmatrix}.
		\end{equation}
		Moreover, as $z \to \infty$, we have
		\begin{align} \label{eq:asyhatN1}
			\what N(z)=&\left( I+ \frac{\what N^{(1)}}{z}+\Boh(z^{-2}) \right) \operatorname{diag}\left(z^{\frac{1}{4}}, z^{-\frac{1}{4}},z^{\frac{1}{4}}, z^{-\frac{1}{4}}\right)    \operatorname{diag}\left(\begin{pmatrix} 1 & -1 \\ 1 & 1 \end{pmatrix},\begin{pmatrix} 1 & 1 \\ -1 & 1 \end{pmatrix}\right)  \nonumber\\ 
			&\quad \times \operatorname{diag}\left((-\sqrt{z})^{-\frac{1}{4}}, z^{-\frac{1}{8}},(-\sqrt{z})^{\frac{1}{4}}, z^{\frac{1}{8}}\right) A,
		\end{align}
		where $A$ is defined in \eqref{def:A} and
		\begin{equation}
			\what N^{(1)}=\begin{pmatrix}
				*&*&*&*\\
				2\beta^2&*&*&*\\
				*&*&*&*\\
				*&*&2\beta^2&*
			\end{pmatrix}.
		\end{equation}
	\end{proposition}
	\begin{proof}
		To check the desired jump condition \eqref{eq:hatN-jump}, it suffices to show
		\begin{equation}
			\left\{ \begin{array}{ll}
				\begin{pmatrix}
					0&0&1&0\\0&1&0&0\\-1&0&0&0\\0&0&0&1
				\end{pmatrix} D_+(x) =D_-(x) \begin{pmatrix}
					0&0&1&0\\0&1&0&0\\-1&0&0&0\\0&0&0&1
				\end{pmatrix}
				, & \quad \hbox{$x>1$,}
				\\
				\begin{pmatrix}
					0&0&1&0\\0&1&0&0\\-1&0&0&0\\0&0&0&1
				\end{pmatrix} D_+(x) =D_-(x) \begin{pmatrix}
					0&0&1-\gamma&0
					\\
					0&1&0&0
					\\
					\frac{1}{\gamma-1}&0&0&0
					\\0&0&0&1
				\end{pmatrix}, & \quad  \hbox{$x \in (0,1)$,}
				\\
				\begin{pmatrix}0&-\ii&0&0\\-\ii &0&0&0\\0&0&0&\ii\\0&0&\ii&0 \end{pmatrix}D_+(x)=D_-(x) \begin{pmatrix}0&-\ii&0&0\\-\ii &0&0&0\\0&0&0&\ii\\0&0&\ii&0 \end{pmatrix}, & \quad  \hbox{$x<0$,}
			\end{array}
			\right.
		\end{equation} 
		where \begin{equation}
 D(z):=
 \operatorname{diag}\left(d_1(z),d_2(z),d_3(z),d_4(z)\right).
		\end{equation}
These relations can be verified by using item (i) of Proposition \ref{prop:di} directly. Finally, the large $z$ behavior of $\what N$ in \ref{eq:asyhatN} is a consequence of item (ii) of Proposition \ref{prop:di}.
	\end{proof}
 
	\subsection{Local parametrix near $z=0$}
In small neighborhoods of $z=0$ and $z=1$, we need to construct local parametrices as approximations of $\what{S}$ and start with the local parametrix near the origin.
	
	\begin{rhp}\label{rhp:hatP0}
		\hfill
		\begin{enumerate}
			\item[\rm (a)] $\what P^{(0)}(z)$ is defined and analytic in $D(0, \varepsilon)\setminus \Gamma_{\what S}$, where $\Gamma_{\what S}$ is defined in \eqref{def:gammahatS}.
			
			\item[\rm (b)] For $z \in D(0, \varepsilon) \cap \Gamma_{\what S}$, $\what P^{(0)}(z)$ satisfies the jump condition
			\begin{equation}\label{eq:hatP0-jump}
				\what P^{(0)}_+(z)=\what P^{(0)}_-(z)J_{\what S}(z),
			\end{equation}
			where $J_{\what S}(z)$ is given in \eqref{def:JhatS}.
			
			\item[\rm (c)]As $s \to \infty$,  we have the matching condition
			\begin{equation}\label{eq:matchhatP0}
				\what P^{(0)}(z)=\left( I+ \Boh(s^{-\frac 14}) \right) \what N(z),\quad z \in \partial D(0, \varepsilon),
			\end{equation}
			where $\what N(z)$ is given in \eqref{eq:hatNexp}.
		\end{enumerate}
	\end{rhp}
	
This RH problem can be solved by using the solution $\widehat{M}(z)$ of  RH problem \ref{rhp:hard edge tac}. More precisely, we define
	\begin{align}\label{def:hatP0}
		\what P^{(0)}(z) & = \what E_0(z) \widehat{M}(sz) \diag \left((1-\gamma)^{-\frac 12}e^{\theta_1(\sqrt{sz}) - \tau \sqrt{sz}},(1-\gamma)^{-\frac 12}e^{\theta_2(\sqrt{sz}) + \tau\sqrt{sz}},\right.\nonumber\\
		&\qquad  \left.(1-\gamma)^{\frac 12}e^{-\theta_1(\sqrt{sz}) - \tau \sqrt{sz}},(1-\gamma)^{\frac 12}e^{-\theta_2(\sqrt{sz}) + \tau \sqrt{sz}}\right)
	\end{align}
	with
	\begin{align}\label{def:hatE0z}
		\what E_0(z)& = \frac{1}{2}\what N(z)\diag\left((1-\gamma)^{\frac 12},(1-\gamma)^{\frac 12},(1-\gamma)^{-\frac 12},(1-\gamma)^{-\frac 12}\right) A^{-1}\nonumber\\
		&\quad \times \diag \left((-\sqrt{sz})^{\frac 14}, (sz)^{\frac 18},(-\sqrt{sz})^{-\frac 14},(sz)^{-\frac 18}\right) \nonumber\\
		&\quad \times \operatorname{diag}\left(\begin{pmatrix} 1 & 1 \\ -1 & 1 \end{pmatrix},\begin{pmatrix} 1 & -1 \\ 1 & 1 \end{pmatrix}\right) \operatorname{diag}\left({(sz)}^{-\frac{1}{4}}, {(sz)}^{\frac{1}{4}},{(sz)}^{-\frac{1}{4}}, {(sz)}^{\frac{1}{4}}\right) ,
	\end{align}
	where $A$, $\theta_1$ and $\theta_2$ are given in \eqref{def:A} -- \eqref{def:theta2}, respectively, and $\what N (z)$ is given in \eqref{eq:hatNexp}.
	
	\begin{proposition}\label{prop:whatP0}
		The local parametrix $\what P^{(0)}(z)$ defined in \eqref{def:hatP0} solves RH problem \ref{rhp:hatP0}.
	\end{proposition}
	\begin{proof}
		We first show the prefactor $\what E_0(z)$ is analytic in $D(0, \varepsilon)$. From its definition in \eqref{def:hatE0z}, the only possible jump is on $(-\varepsilon, \varepsilon)$. For $x \in (-\varepsilon, 0)$, recalling the jump matrix of $\what{N}(z)$ in \eqref{eq:hatN-jump}, we have
		\begin{align}
			&\what E_{0,-}(x)^{-1} \what E_{0,+}(x)\nonumber\\
			&= \frac{1}{2} \diag \left((sz)_-^{\frac 14}, (sz)_-^{-\frac 14},(sz)_-^{\frac 14},(sz)_-^{-\frac 14}\right)  \operatorname{diag}\left(\begin{pmatrix} 1 & -1 \\ 1 & 1 \end{pmatrix},\begin{pmatrix} 1 & 1 \\ -1 & 1 \end{pmatrix}\right) \nonumber\\ 
			&\quad \times\diag \left((-\sqrt{sz})_-^{-\frac 14}, (sz)_-^{-\frac 18},(-\sqrt{sz})_-^{\frac 14},(sz)_-^{\frac 18}\right) A \diag\left((1-\gamma)^{-\frac 12},(1-\gamma)^{-\frac 12},(1-\gamma)^{\frac 12},(1-\gamma)^{\frac 12}\right)\nonumber\\
			&\quad \times \begin{pmatrix}
				0&-\ii&0&0
				\\
				-\ii&0&0&0
				\\
				0&0&0&\ii
				\\0&0&\ii&0
			\end{pmatrix} \diag\left((1-\gamma)^{\frac 12},(1-\gamma)^{\frac 12},(1-\gamma)^{-\frac 12},(1-\gamma)^{-\frac 12}\right) A^{-1}\nonumber\\
			&\quad \times \diag \left((-\sqrt{sz})_+^{\frac 14}, (sz)_+^{\frac 18},(-\sqrt{sz})_+^{-\frac 14},(sz)_+^{-\frac 18}\right) \operatorname{diag}\left(\begin{pmatrix} 1 & 1 \\ -1 & 1 \end{pmatrix},\begin{pmatrix} 1 & -1 \\ 1 & 1 \end{pmatrix}\right)\nonumber\\
			&\quad \times  \operatorname{diag}\left((sz)_+^{-\frac{1}{4}}, (sz)_+^{\frac{1}{4}},(sz)_+^{-\frac{1}{4}}, (sz)_+^{\frac{1}{4}}\right) =I.
		\end{align}
		Similarly, for $x \in (0, \varepsilon)$, we have
		\begin{align}
			&\what E_{0,-}(x)^{-1} \what E_{0,+}(x)\nonumber\\
			&= \frac{1}{2} \diag \left((sz)^{\frac 14}, (sz)^{-\frac 14},(sz)^{\frac 14},(sz)^{-\frac 14}\right)  \operatorname{diag}\left(\begin{pmatrix} 1 & -1 \\ 1 & 1 \end{pmatrix},\begin{pmatrix} 1 & 1 \\ -1 & 1 \end{pmatrix}\right)  \nonumber\\ 
			&\quad \times\diag \left((-\sqrt{sz})_-^{-\frac 14}, (sz)^{-\frac 18},(-\sqrt{sz})_-^{\frac 14},(sz)^{\frac 18}\right) A \diag\left((1-\gamma)^{-\frac 12},(1-\gamma)^{-\frac 12},(1-\gamma)^{\frac 12},(1-\gamma)^{\frac 12}\right)\nonumber\\
			&\quad \times \begin{pmatrix}
				0&0&1-\gamma&0
				\\
				0&1&0&0
				\\
				\frac{1}{\gamma -1}&0&0&0
				\\0&0&0&1
			\end{pmatrix} \diag\left((1-\gamma)^{\frac 12},(1-\gamma)^{\frac 12},(1-\gamma)^{-\frac 12},(1-\gamma)^{-\frac 12}\right) A^{-1}\nonumber\\
			&\quad \times \diag \left((-\sqrt{sz})_+^{\frac 14}, (sz)^{\frac 18},(-\sqrt{sz}))_+^{-\frac 14},(sz)^{-\frac 18}\right) \operatorname{diag}\left(\begin{pmatrix} 1 & 1 \\ -1 & 1 \end{pmatrix},\begin{pmatrix} 1 & -1 \\ 1 & 1 \end{pmatrix}\right)\nonumber\\
			&\quad \times  \operatorname{diag}\left((sz)^{-\frac{1}{4}}, (sz)^{\frac{1}{4}},(sz)^{-\frac{1}{4}}, (sz)^{\frac{1}{4}}\right) =I.
		\end{align}
		Moreover, as $z \to 0$, one has
		\begin{equation}
			\what E_0(z) = \what E_0(0) + \what E_0'(0)z + \Boh(z^{2}),
		\end{equation}
		where
		\begin{equation}\label{def:hatE0}
			\what E_0(0) = C_0\begin{pmatrix}
				1 & 0 & -2\beta & 0\\
				2\beta^2 & 1 & -\frac{2(\beta +2\beta^3)}{3} & 2 \beta\\
				0 & 0 & 1 & 0\\
				2\beta & 0  & -2\beta^2 & 1
			\end{pmatrix} \diag \left(s^{-\frac 18}, s^{\frac 38},s^{-\frac 38},s^{\frac 18}\right)
		\end{equation}
		and
			\begin{align}
				\what E_0'(0) &=C_0 \begin{pmatrix}
					\frac{4\beta^2+2\beta^4}{3}& 2\beta^2 & -\frac{2\beta(3+10\beta^2+2\beta^4)}{15} & \frac{4\beta^3+2\beta}{3}\\
					\frac{2\beta^2(23+20\beta^2+2\beta^4)}{45} & \frac{4\beta^2+2\beta^4}{3} & -\frac{2\beta(45+196\beta^2+20\beta^4+4\beta^6)}{315} & \frac{2\beta(3+10\beta^2+2\beta^4)}{15}\\
					-\frac{4\beta^3+2\beta}{3} & -2\beta & \frac{4\beta^2+2\beta^4}{3} & -2\beta^2\\
					\frac{2\beta(3+10\beta^2+2\beta^4)}{15} & \frac{4\beta^3+2\beta}{3}  & -\frac{2\beta^2(23+20\beta^2+2\beta^4)}{45} & 	\frac{4\beta^2+2\beta^4}{3}
				\end{pmatrix} \nonumber \\
				&\quad\times\diag \left(s^{-\frac 18}, s^{\frac 38},s^{-\frac 38},s^{\frac 18}\right)
			\end{align}
with $C_0$ given in \eqref{def:C_0}. Thus, $\what E_0(z)$ is indeed analytic in $D(0, \varepsilon)$ and the jump condition \eqref{eq:hatP0-jump} can be verified easily from this fact, \eqref{eq:thetairelations} and
		the jump condition of $\widehat{M}$ given in \eqref{jumps:Mhat}.
		
		For the matching condition, it follows from the asymptotic behavior of $\widehat{M}$ at infinity given in \eqref{eq:asy:Mhat} that, for $z \in \partial D(0, \varepsilon)$,
		\begin{equation}\label{eq:match01}
			\what P^{(0)}(z)\what N(z)^{-1} = I + \frac{\what J^{(0)}_1(z)}{s^{1/4}} +  \frac{\what J^{(0)}_2(z)}{s^{1/2}}+\frac{\what J^{(0)}_3(z)}{s^{3/4}}+\Boh(s^{-1}),\quad s \to +\infty,
		\end{equation}
		where
		\begin{equation}\label{def:hatJ01}
			\what J^{(0)}_1(z) =  \widetilde E_0(z) \begin{pmatrix}
				0 &0 &0 &-M^{(1)}_{13}+M^{(1)}_{14}\\
				0 & 0 &(M^{(1)}_{13}+M^{(1)}_{14})z^{-1} & 0\\
				0 & 0 & 0& 0\\
				0 & 0 & 0 &0
			\end{pmatrix}\widetilde E_0(z)^{-1},
		\end{equation}
		\begin{align}\label{def:hatJ02}
			\what J^{(0)}_2(z)=
			\widetilde E_0(z) \begin{pmatrix}
				0&M^{(1)}_{11}+M^{(1)}_{12} &0 &0\\
				(M^{(1)}_{11}-M^{(1)}_{12})z^{-1} & 0 &0 & 0\\
				0 & 0 & 0 &-M^{(1)}_{33}+M^{(1)}_{34}\\
				0 & 0 & (-M^{(1)}_{33}-M^{(1)}_{34})z^{-1} &0
			\end{pmatrix}\widetilde E_0(z)^{-1}
		\end{align}
		and
		\begin{align}\label{def:hatJ03}
			\what J^{(0)}_3(z)=
			\widetilde E_0(z) \begin{pmatrix}
				0 &0 &(M^{(2)}_{13}+M^{(2)}_{14})z^{-1} &0\\
				0 & 0 &0 & -M^{(2)}_{13}+M^{(2)}_{14}\\
				0 & M^{(1)}_{31}+M^{(1)}_{32} & 0 & 0\\
				(-M^{(1)}_{31}+M^{(1)}_{32})z^{-1} & 0 & 0 &0
			\end{pmatrix}\widetilde E_0(z)^{-1}
		\end{align}
	with $\widetilde E_0(z):=\what E_0(z) \diag \left(s^{\frac 18}, s^{-\frac 38},s^{\frac 38},s^{-\frac 18}\right) $, the matrices $M^{(1)}$ and $M^{(2)}$ given in \eqref{eq:asy:Mhat}.

 This completes the proof of Proposition \ref{prop:whatP0}. 
	\end{proof}
	
	\subsection{Local parametrix near $z=1$}
	Near the endpoint $z=1$, the local parametrix reads as follows.
	
	\begin{rhp}\label{rhp:hatP1}
		\hfill
		\begin{enumerate}
			\item[\rm (a)] $\what P^{(1)}(z)$ is defined and analytic in $D(1, \varepsilon)\setminus \Gamma_{\what S}$, where $\Gamma_{\what S}$ is defined in \eqref{def:gammahatS}.
			
			\item[\rm (b)] For $z \in D(1, \varepsilon) \cap \Gamma_{\what S}$, $\what P^{(1)}(z)$ satisfies the jump condition
			\begin{equation}\label{eq:hatP1-jump}
				\what P^{(1)}_+(z)=\what P^{(1)}_-(z)J_{\what S}(z),
			\end{equation}
			where $J_{\what S}(z)$ is given in \eqref{def:JhatS}.
			
			\item[\rm (c)]As $s \to \infty$,  we have the matching condition
			\begin{equation}\label{eq:matchhatP1}
				\what P^{(1)}(z)=\left( I+ \Boh(s^{-\frac{3}{4}}) \right) \what N(z), \quad z \in \partial D(1, \varepsilon),
			\end{equation}
			where $\what N(z)$ is given in \eqref{eq:hatNexp}.
		\end{enumerate}
	\end{rhp}
	
The above RH problem can be solved by using the confluent hypergeometric parametrix $\Phi^{(\CHF)}$ in Appendix \ref{sec:CHF}. To do this, we introduce the function
	\begin{equation}\label{def:hatf1}
		\what f_{1}(z) = -2 \ii s^{-\frac 34} \begin{cases}
			\theta_1(\sqrt{sz}) - \theta_{1,+}(\sqrt{s}), & \quad\Im{z} > 0,\\
			-\theta_1(\sqrt{sz}) + \theta_{1,-}(\sqrt{s}), & \quad\Im{z} < 0.
		\end{cases}
	\end{equation}
	By \eqref{def:theta1}, it is easily seen that
	\begin{equation}\label{eq:hatf1}
		\what f_{1}(z) = \left(1 - \frac{\tilde{s}}{\sqrt{s}}\right)(z-1) - \left(\frac{\sqrt{s}-3\tilde{s}}{8\sqrt{s}}\right)(z-1)^2+\Boh((z-1)^3), \qquad z \to 1.
	\end{equation}
	Hence, $\what f_{1}$ is a conformal mapping near $z=1$ for large positive $s$.
	
	We then define
	\begin{align}\label{def:hatP1}
		\what P^{(1)}(z) &= \what E_{1}(z) \begin{pmatrix}
			\Phi^{(\CHF)}_{11}(s^{\frac 34} \what f_{1}(z); \beta) & 0 & \Phi^{(\CHF)}_{12}(s^{\frac 34} \what f_{1}(z); \beta) & 0\\
			0 & 1 &0&0\\
			\Phi^{(\CHF)}_{21}(s^{\frac 34} \what f_{1}(z); \beta) & 0 & \Phi^{(\CHF)}_{22}(s^{\frac 34} \what f_{1}(z); \beta) & 0\\
			0 & 0 & 0 & 1
		\end{pmatrix}\nonumber\\
		&\quad \times \diag \left(e^{\theta_1(\sqrt{sz}) - \frac{\beta \pi \ii}{2}}, 1, e^{-\theta_1(\sqrt{sz}) + \frac{\beta \pi \ii}{2}}, 1\right),
	\end{align}
	where $\beta$ is given in \eqref{def:beta}, $\what f_{1}$ is defined in \eqref{def:hatf1} and
	\begin{align}\label{def:hatE1}
		&\what E_{1}(z) = \what N(z)\nonumber\\
		&\times \begin{cases}
			\diag \left(e^{-\theta_{1,+}(\sqrt{s}) + \frac{\beta \pi \ii}{2}}s^{\frac{3\beta}{4}} \what f_{1}(z)^{\beta}, 1, e^{\theta_{1,+}(\sqrt{s}) -\frac{\beta \pi \ii}{2}}s^{-\frac{3\beta}{4}} \what f_{1}(z)^{-\beta}, 1\right), &\quad \Im{z}>0,\\
			\begin{pmatrix}
				0 & 0 & 1 & 0\\
				0 & 1 & 0 & 0\\
				-1 & 0 & 0 & 0\\
				0 & 0 & 0 & 1
			\end{pmatrix}\diag \left(e^{\theta_{1,-}(\sqrt{s}) +\frac{\beta \pi \ii}{2}}s^{\frac{3\beta}{4}} \what f_{1}(z)^{\beta}, 1, e^{-\theta_{1,-}(\sqrt{s}) -\frac{\beta \pi \ii}{2}}s^{-\frac{3\beta}{4}} \what f_{1}(z)^{-\beta},1\right),&\quad \Im{z}<0.
		\end{cases}
	\end{align}
	
	\begin{proposition}\label{prop:whatP1}
		The local parametrix $\what P^{(1)}(z)$ defined in \eqref{def:hatP1} solves RH problem \ref{rhp:hatP1}.
	\end{proposition}
	\begin{proof}
		We first show the analyticity of $\what E_{1}(z)$. From its definition in \eqref{def:hatE1}, the only possible jump is on $(1-\varepsilon, 1+\varepsilon)$. For $x \in (1-\varepsilon, 1)$, we have $\what f_{1,+}(x)^{\beta} = e^{2 \beta \pi \ii}\what f_{1,-}(x)^{\beta}$, it then follows from \eqref{eq:hatN-jump} that
		\begin{equation}
			\what E_{1,-}(x)^{-1}\what E_{1,+}(x) = I.
		\end{equation}
		Similarly, for $x \in (1, 1+\varepsilon)$, we also have
		\begin{equation}
			\what E_{1,-}(x)^{-1}\what E_{1,+}(x) = I.
		\end{equation}
	Moreover, as $z \to 1$, 
		\begin{equation}
			\what E_{1}(z) = \what E_{1}(1)\left(I + \what E_{1}(1)^{-1}\what E_{1}'(1)(z-1) +\Boh\left((z-1)^2\right)\right), 
		\end{equation}
		where
		\begin{equation}\label{eq:hatE11}
			\what E_{1}(1) =  \frac{C_0}{\sqrt{2}} \begin{pmatrix}
				e^{\frac{\pi \ii}{4}}\what {\textsf{a}} & -e^{-\frac{\beta \pi \ii}{2}} & -\ii e^{\frac{\pi \ii}{4}}\what {\textsf{a}}^{-1}& -\ii e^{\frac{\beta \pi \ii}{2}} \\
				e^{\frac{\pi \ii}{4}}\what {\textsf{a}} & e^{-\frac{\beta \pi \ii}{2}} & -\ii e^{\frac{\pi \ii}{4}}\what {\textsf{a}}^{-1} & \ii e^{\frac{\beta \pi \ii}{2}}\\
				-\ii e^{-\frac{\pi \ii}{4}}\what {\textsf{a}} & \ii e^{-\frac{\beta \pi \ii}{2}} & e^{\frac{-\pi \ii}{4}}\what {\textsf{a}}^{-1} & e^{\frac{\beta \pi \ii}{2}} \\
				\ii e^{-\frac{\pi \ii}{4}}\what {\textsf{a}} & \ii e^{-\frac{\beta \pi \ii}{2}} & -e^{\frac{-\pi \ii}{4}}\what {\textsf{a}}^{-1} &  e^{\frac{\beta \pi \ii}{2}}
			\end{pmatrix}
		\end{equation}
		and
		\begin{equation}\label{eq:hatE1'1}
			\what E_{1}(1)^{-1}\what E_{1}'(1)=\left(
			\begin{array}{cccc}
				\what {\textsf{b}} & -\frac{e^{-\frac{\beta \pi \ii  }{2}}}{4 \sqrt{2} \what {\textsf{a}} } & \frac{i}{8 \what {\textsf{a}}^2} & -\frac{e^{\frac{\beta \pi \ii }{2}}}{4 \sqrt{2} \what {\textsf{a}}} \\
				-\frac{\what {\textsf{a}} e^{\frac{\beta \pi \ii }{2}}}{4 \sqrt{2}} & \frac{ \beta \ii }{2} & -\frac{e^{\frac{\beta \pi \ii }{2}}}{4 \sqrt{2} \what {\textsf{a}}} & -\frac{1}{8} i e^{\beta \pi \ii } \\
				-\frac{\ii \what {\textsf{a}}^2}{8}  & -\frac{e^{-\frac{\beta \pi \ii }{2}}\what {\textsf{a}}}{4 \sqrt{2} } & -	\what {\textsf{b}} & \frac{\what {\textsf{a}} e^{\frac{\beta \pi \ii }{2}}}{4 \sqrt{2}} \\
				-\frac{e^{-\frac{\beta \pi \ii }{2}}\what {\textsf{a}}}{4 \sqrt{2} } & \frac{\ii}{8 e^{\beta \pi \ii }} & \frac{e^{-\frac{\beta \pi \ii }{2}}}{4 \sqrt{2} \what {\textsf{a}} } & -\frac{ \beta \ii}{2}  \\
			\end{array}
			\right)
		\end{equation}
		with
		\begin{equation}\label{def:whata}
			\what {\textsf{a}} = e^{-\theta_{1,+}(\sqrt{s})+\pi \ii \beta/2}\what f_{1}'(1)^{\beta} 8^{\beta} s^{3 \beta /4},\qquad 
			\what {\textsf{b}}=\frac{\beta}{2}\left(1-\frac{\sqrt{s}-3\tilde{s}}{8(\sqrt{s}-\tilde{s})}\right).
		\end{equation}
Thus, $\what E_1(z)$ is indeed analytic in $D(1, \varepsilon)$ and the jump condition \eqref{eq:hatP1-jump} can be verified easily from this fact, \eqref{def:hatP1} and
the jump condition of $\Phi^{(\CHF)}$ given in \eqref{CHFJumps}.

For the matching condition, it follows from  the asymptotic behavior of the confluent hypergeometric parametrix at infinity in \eqref{CHF at infinity} that, as $s \to \infty$,
		\begin{align}\label{eq:match11}
			\what P^{(1)}(z) \what N(z)^{-1 }= I+ \frac{\what J^{(1)}_3(z)}{s^{3/4}}+\Boh(s^{-1}),
		\end{align}
		where 
		\begin{align}\label{def:hatJ11}
			\what J^{(1)}_3(z)=\frac{1}{\what f_1(z)}\what E_{1}(1)\Phi_1 \what E_{1}(1)^{-1}
		\end{align}
		with
		\begin{align}\label{def:Phi1}
			\Phi_1=\begin{pmatrix}
				\beta^2 \ii &0 &-\frac{\Gamma(1-\beta)}{\Gamma(\beta)}e^{-\beta \pi \ii} \ii &0 \\
				0 &1 &0 &0 \\
				\frac{\Gamma(1+\beta)}{\Gamma(-\beta)}e^{\beta \pi \ii} \ii &0 &-\beta^2 \ii &0 \\
				0 &0 &0&1\\
			\end{pmatrix}.
		\end{align}
This completes the proof of Proposition \ref{prop:whatP1}. 		
	\end{proof}
 
	\subsection{Final transformation}
The final transformation is defined by
	\begin{equation}\label{def:hatR}
		\what R(z) = \begin{cases}
			\what S(z) \what P^{(0)}(z)^{-1}, & \quad z \in D(0, \varepsilon),\\
			\what S(z) \what P^{(1)}(z)^{-1}, & \quad z \in D(1, \varepsilon),\\
			\what S(z) \what N(z)^{-1}, & \quad \textrm{elsewhere}.
		\end{cases}
	\end{equation}
Similar to RH problem \ref{rhp:R}, it follows from the RH problems for $\what S$, $\what N$, $\what P^{(0)}$ and $\what P^{(1)}$ that $\what R$ satisfies the following RH problem.
	\begin{figure}[t]
		\begin{center}
			
			\tikzset{every picture/.style={line width=0.55pt}} 
			
			\begin{tikzpicture}[x=0.75pt,y=0.75pt,yscale=-1,xscale=1]
				
				\draw   (190,146) .. controls (190,132.19) and (201.19,121) .. (215,121) .. controls (228.81,121) and (240,132.19) .. (240,146) .. controls (240,159.81) and (228.81,171) .. (215,171) .. controls (201.19,171) and (190,159.81) .. (190,146) -- cycle ;
				\draw  [fill={rgb, 255:red, 0; green, 0; blue, 0 }  ,fill opacity=1 ] (214,145) .. controls (214,144.45) and (214.45,144) .. (215,144) .. controls (215.55,144) and (216,144.45) .. (216,145) .. controls (216,145.55) and (215.55,146) .. (215,146) .. controls (214.45,146) and (214,145.55) .. (214,145) -- cycle ;
				
				\draw   (300,145) .. controls (300,131.19) and (311.19,120) .. (325,120) .. controls (338.81,120) and (350,131.19) .. (350,145) .. controls (350,158.81) and (338.81,170) .. (325,170) .. controls (311.19,170) and (300,158.81) .. (300,145) -- cycle ;
				\draw  [fill={rgb, 255:red, 0; green, 0; blue, 0 }  ,fill opacity=1 ] (324,144) .. controls (324,143.45) and (324.45,143) .. (325,143) .. controls (325.55,143) and (326,143.45) .. (326,144) .. controls (326,144.55) and (325.55,145) .. (325,145) .. controls (324.45,145) and (324,144.55) .. (324,144) -- cycle ;
				
				\draw    (130.58,58.92) -- (200,126) ;
				\draw    (340.58,164.92) -- (410,232) ;
				\draw    (130,231) -- (200,166) ;
				\draw    (340,125) -- (410,60) ;
				\draw  [fill={rgb, 255:red, 0; green, 0; blue, 0 }  ,fill opacity=1 ] (171.11,97.8) -- (165.5,96.45) -- (169.43,92.28) -- cycle ;
				\draw  [fill={rgb, 255:red, 0; green, 0; blue, 0 }  ,fill opacity=1 ] (171.19,192.46) -- (169.24,197.88) -- (165.52,193.52) -- cycle ;
				\draw  [fill={rgb, 255:red, 0; green, 0; blue, 0 }  ,fill opacity=1 ] (375.29,198.46) -- (369.68,197.11) -- (373.61,192.94) -- cycle ;
				\draw  [fill={rgb, 255:red, 0; green, 0; blue, 0 }  ,fill opacity=1 ] (375,92.5) -- (373.05,97.92) -- (369.33,93.56) -- cycle ;
				\draw  [draw opacity=0] (227.81,123.69) .. controls (234.13,115.74) and (250.41,110.08) .. (269.5,110.08) .. controls (288.59,110.08) and (304.87,115.74) .. (311.19,123.69) -- (269.5,131.04) -- cycle ; \draw   (227.81,123.69) .. controls (234.13,115.74) and (250.41,110.08) .. (269.5,110.08) .. controls (288.59,110.08) and (304.87,115.74) .. (311.19,123.69) ;  
				\draw  [fill={rgb, 255:red, 0; green, 0; blue, 0 }  ,fill opacity=1 ] (271.77,110.08) -- (265.41,112.95) -- (265.41,107.21) -- cycle ;
				
				\draw  [draw opacity=0] (312.19,167.52) .. controls (305.87,175.48) and (289.59,181.13) .. (270.5,181.13) .. controls (251.41,181.13) and (235.13,175.48) .. (228.81,167.52) -- (270.5,160.17) -- cycle ; \draw   (312.19,167.52) .. controls (305.87,175.48) and (289.59,181.13) .. (270.5,181.13) .. controls (251.41,181.13) and (235.13,175.48) .. (228.81,167.52) ;  
				\draw  [fill={rgb, 255:red, 0; green, 0; blue, 0 }  ,fill opacity=1 ] (274.59,181.13) -- (268.23,184) -- (268.23,178.27) -- cycle ;
				
				\draw  [fill={rgb, 255:red, 0; green, 0; blue, 0 }  ,fill opacity=1 ] (217.5,121) -- (212.5,123.87) -- (212.5,118.13) -- cycle ;
				\draw  [fill={rgb, 255:red, 0; green, 0; blue, 0 }  ,fill opacity=1 ] (328,120) -- (323,122.87) -- (323,117.13) -- cycle ;
				
				\draw (210,148) node [anchor=north west][inner sep=0.75pt]   [align=left] {$0$};
				\draw (321,148) node [anchor=north west][inner sep=0.75pt]   [align=left] {$1$};
				\draw (105,50) node [anchor=north west][inner sep=0.75pt]   [align=left] {$\Gamma_2^{(1)}$};
				\draw (105,225) node [anchor=north west][inner sep=0.75pt]   [align=left] {$\Gamma_4^{(1)}$};
				\draw (415,50) node [anchor=north west][inner sep=0.75pt]   [align=left] {$\Gamma_1^{(1)}$};
				\draw (415,225) node [anchor=north west][inner sep=0.75pt]   [align=left] {$\Gamma_5^{(1)}$};
				\draw (257,90) node [anchor=north west][inner sep=0.75pt]   [align=left] {$\partial \Omega_{R,+}$};
				\draw (257,190) node [anchor=north west][inner sep=0.75pt]   [align=left] {$\partial \Omega_{R,-}$};

			\end{tikzpicture}
			
			\caption{The jump contour $\Gamma_{\what R}$ of the RH problem for $\what R$.}
			\label{fig:hatR}
		\end{center}
	\end{figure}
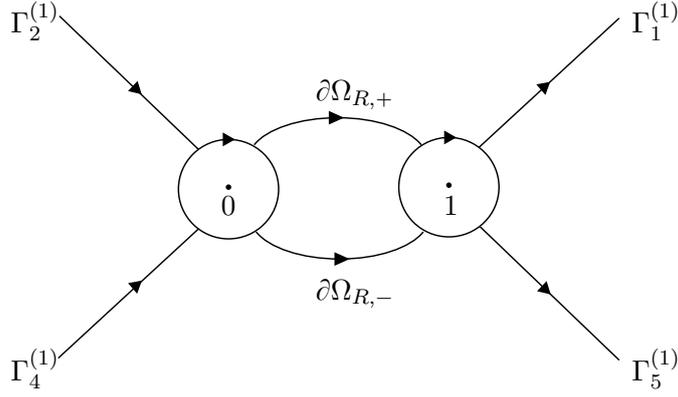

	\begin{rhp}
		\hfill
		\begin{itemize}
			\item [\rm{(a)}] $\what R(z)$ is defined and analytic in $\mathbb{C} \setminus \Gamma_{\what R}$, where
			\begin{equation}
				\Gamma_{\what R}:=\Gamma_{\what S} \cup \partial D(0,\varepsilon) \cup \partial D(1,\varepsilon) \setminus \{\mathbb{R}
				\cup D(0,\varepsilon) \cup D(1,\varepsilon) \};
			\end{equation}
			see Figure \ref{fig:hatR} for an illustration of $\Gamma_{\what R}$.
			\item [\rm{(b)}] For $z \in \Gamma_{\what R}$, we have
			\begin{equation}\label{eq:hatRjump}
				\what R_+(z) = \what R_-(z) J_{\what R} (z),
			\end{equation}
			where
			\begin{equation}\label{def:JhatR}
				J_{\what R}(z) = \begin{cases}
					\what P^{(0)}(z) \what N(z)^{-1}, & \quad z \in \partial D(0, \varepsilon),\\
					\what P^{(1)}(z) \what N(z)^{-1}, & \quad z \in \partial D(1, \varepsilon),\\
					\what N(z) J_{\what S}(z) \what N(z)^{-1}, & \quad \Gamma_{\what R} \setminus \left\{\partial D(0, \varepsilon) \cup \partial D( 1, \varepsilon) \right\},
				\end{cases}
			\end{equation}
			with $J_{\what S}(z)$ defined in \eqref{def:JhatS}.
			\item [\rm{(c)}] As $z \to \infty$, we have
			\begin{equation}\label{eq:asy:hatR}
				\what R(z) = I + \frac{\what R^{(1)}}{z} + \Boh (z^{-2}),
			\end{equation}
			where $\what R^{(1)}$ is independent of $z$.
		\end{itemize}
	\end{rhp}
	As $s\to +\infty$, we have the following estimate of $J_{\what R}(z)$ in \eqref{def:JhatR}. For $z \in \Gamma_{\what R} \setminus \left\{\partial D(0, \varepsilon) \cup \partial D( 1, \varepsilon) \right\}$, it is readily seen from \eqref{def:JhatS} and \eqref{eq:hatNexp} that there exists a positive constant $c$ such that
	\begin{equation}\label{eq:estJhatR1}
		J_{\what R}(z) = I + \Boh\left(e^{-c s^{3/4}}\right).
	\end{equation}
	For $z \in \partial D( 1, \varepsilon)$, it follows from  \eqref{eq:matchhatP1} that
	\begin{equation}	
		J_{\what R}(z) = I +\frac{\what J^{(1)}_3(z)}{s^{3/4}}+ \Boh\left(s^{-1}\right),
	\end{equation}
where $\what J^{(1)}_3(z)$ is given in \eqref{def:hatJ11}.
	For $z \in \partial D(0, \varepsilon)$, we obtain from \eqref{eq:match01} that
	\begin{equation}\label{eq:estJhatR3}
		J_{\what R}(z) = I + \frac{\what J^{(0)}_1(z)}{s^{1/4}} + \frac{\what J^{(0)}_2(z)}{s^{1/2}}+\frac{\what J^{(0)}_3(z)}{s^{3/4}}+ \Boh\left(s^{-1}\right),
	\end{equation}
	where $\what J^{(0)}_1(z)$, $\what J^{(0)}_2(z)$ and $\what J^{(0)}_3(z)$ are given in \eqref{def:hatJ01}, \eqref{def:hatJ02} and \eqref{def:hatJ03}.
	
	By \cite{Deift1999, Deift1993}, the estimates \eqref{eq:estJhatR1}--\eqref{eq:estJhatR3} imply that
	\begin{equation}\label{eq:hatRexpansion}
		\what R(z) = I + \frac{\what R_1(z)}{s^{1/4}} + \frac{\what R_2(z)}{s^{1/2}}+\frac{\what R_3(z)}{s^{3/4}}+ \Boh\left(s^{-1}\right), \qquad s \to +\infty,
	\end{equation}
	uniformly for $z \in \mathbb{C} \setminus \Gamma_{\what R}$. Moreover, by inserting \eqref{eq:estJhatR3} and  \eqref{eq:hatRexpansion} into \eqref{eq:hatRjump}, it follows that $\what R_1$ satisfies the following RH problem.
	\begin{rhp}
		\hfill
		\begin{itemize}
			\item [\rm{(a)}] $\what R_1(z)$ is analytic in $\mathbb{C} \setminus \partial D(0, \varepsilon)$.
			\item [\rm{(b)}] For $z \in \partial D(0, \varepsilon)$, we have
			$
				\what R_{1,+}(z)-\what R_{1,-}(z)=\what J^{(0)}_1(z),
			$
			where $\what J^{(0)}_1(z)$ is given in \eqref{def:hatJ01}.
			\item [\rm{(c)}] As $z \to \infty$, we have
			$
				\what R_1(z) = \Boh (z^{-1}).
			$
		\end{itemize}
	\end{rhp}
	By Cauchy's residue theorem, we have
	\begin{align}\label{eq:hatR1}
		\what R_1(z) = \frac{1}{2 \pi \ii} \int_{\partial D(0, \varepsilon)} \frac{\what J^{(0)}_1(\zeta)}{\zeta -z} \ud \zeta
		=\begin{cases}
			\frac{\Res_{\zeta =0}\what J^{(0)}_1(\zeta)}{z} - \what J^{(0)}_1(z), & \quad z \in D(0, \varepsilon),\\
			\frac{\Res_{\zeta =0}\what J^{(0)}_1(\zeta)}{z}, & \quad \textrm{elsewhere}.
		\end{cases}
	\end{align} 
	Similarly, we have that $\what R_2$ and $\what R_3$ in \eqref{eq:hatRexpansion} solve the following two RH problems respectively.
	\begin{rhp}
		\hfill
		\begin{itemize}
			\item [\rm{(a)}] $\what R_2(z)$ is analytic in $\mathbb{C} \setminus \partial D(0, \varepsilon)$.
			\item [\rm{(b)}] For $z \in \partial D(0, \varepsilon)$, we have
			$\what R_{2,+}(z)-\what R_{2,-}(z)=\what R_{1,-}(z) \what J^{(0)}_1(z)+\what J^{(0)}_2(z)$,
			where $\what J^{(0)}_1(z)$ and $\what J^{(0)}_2(z)$ are given in \eqref{def:hatJ01} and \eqref{def:hatJ02}, respectively.
			\item [\rm{(c)}] As $z \to \infty$, we have
			$\what R_2(z) = \Boh (z^{-1})$.
		\end{itemize}
	\end{rhp}
 \begin{rhp}
		\hfill
		\begin{itemize}
			\item [\rm{(a)}] $\what R_3(z)$ is analytic in $\mathbb{C} \setminus \left\{\partial D(0, \varepsilon) \cup \partial D( 1, \varepsilon) \right\}$.
			\item [\rm{(b)}] For $z \in \partial D(0, \varepsilon)$, we have
			\begin{equation*}
				\what R_{3,+}(z)-\what R_{3,-}(z)=\what R_{1,-}(z) \what J^{(0)}_2(z)+\what R_{2,-}(z) \what J^{(0)}_1(z)+\what J^{(0)}_3(z),
			\end{equation*}
			where $\what J^{(0)}_i(z)$, $ i=1,2,3$, are given in \eqref{def:hatJ01}--\eqref{def:hatJ03}, respectively. For $z \in \partial D(1, \varepsilon)$, we have
		$\what R_{3,+}(z)-\what R_{3,-}(z)=\what J^{(1)}_3(z)$,
     where $\what J^{(1)}_3(z)$ is given in \eqref{def:hatJ11}.
			
			\item [\rm{(c)}] As $z \to \infty$, we have
			$\what R_3(z) = \Boh (z^{-1})$.
		\end{itemize}
	\end{rhp}
	Again by Cauchy's residue theorem, we have
	\begin{align}\label{eq:hatR2}
		\what R_2(z) &= \frac{1}{2 \pi \ii} \int_{\partial D(0, \varepsilon)} \frac{\what R_{1,-}(\zeta) \what J^{(0)}_1(\zeta)+\what J^{(0)}_2(\zeta)}{\zeta -z} \ud \zeta
		\nonumber
		\\
		&=\begin{cases}
			\frac{\Res_{\zeta =0}(\what R_{1,-}(\zeta) \what J^{(0)}_1(\zeta)+\what J^{(0)}_2(\zeta))}{z} - \what R_{1,-}(z) \what J^{(0)}_1(z)-\what J^{(0)}_2(z), & \quad z \in D(0, \varepsilon),\\
			\frac{\Res_{\zeta =0}(\what R_{1,-}(\zeta) \what J^{(0)}_1(\zeta)+\what J^{(0)}_2(\zeta))}{z}, & \quad \textrm{elsewhere},
		\end{cases}
	\end{align}
	and
	\begin{align}\label{eq:hatR3}
		\what R_3(z) &= \frac{1}{2 \pi \ii} \int_{\partial D(0, \varepsilon)} \frac{\what R_{1,-}(\zeta) \what J^{(0)}_2(\zeta)+\what R_{2,-}(\zeta) \what J^{(0)}_1(\zeta)+\what J^{(0)}_3(\zeta)}{\zeta -z} \ud \zeta+\frac{1}{2 \pi \ii} \int_{\partial D(1, \varepsilon)} \frac{\what J^{(1)}_3(\zeta)}{\zeta -z} \ud \zeta
		\nonumber
		\\
		&=\begin{cases}
			\frac{\Res_{\zeta =0}(\what R_{1,-}(\zeta) \what J^{(0)}_2(\zeta)+\what R_{2,-}(\zeta) \what J^{(0)}_1(\zeta)+\what J^{(0)}_3(\zeta))}{z}+\frac{\Res_{\zeta =1}(\what J^{(1)}_3(\zeta))}{z-1} \nonumber  \\ 
			- \what R_{1,-}(z) \what J^{(0)}_2(z)-\what R_{2,-}(z) \what J^{(0)}_1(z)-\what J^{(0)}_3(z), & z \in D(0, \varepsilon),\\
			\frac{\Res_{\zeta =0}(\what R_{1,-}(\zeta) \what J^{(0)}_2(\zeta)+\what R_{2,-}(\zeta) \what J^{(0)}_1(\zeta)+\what J^{(0)}_3(\zeta))}{z}+\frac{\Res_{\zeta =1}(\what J^{(1)}_3(\zeta))}{z-1}  -\what J^{(1)}_3(z), & z \in D(1, \varepsilon),\\
			\frac{\Res_{\zeta =0}(\what R_{1,-}(\zeta) \what J^{(0)}_2(\zeta)+\what R_{2,-}(\zeta) \what J^{(0)}_1(\zeta)+\what J^{(0)}_3(\zeta))}{z}+\frac{\Res_{\zeta =1}(\what J^{(1)}_3(\zeta))}{z-1}, &  \textrm{elsewhere}.
		\end{cases}
	\end{align}
 
	\section{Small $s$ asymptotic analysis of the RH problem for $X$}\label{sec:AsyX0}
	
	In this section, we analyze the RH problem for $X$ as $s \to 0^+$ and assume that $0 < \gamma \le 1$ throughout this section.
	
\subsection*{Global parametrix}
	As $s \to 0^+$, the interval $(0, s)$ vanishes, it is then easily seen that the RH problem for $X$ is approximated by following global parametrix $\widecheck{N}$ for $|z|>\delta > s$:
	\begin{equation}\label{eq:widecheckN}
		\widecheck{N}(z) = \what M(z) \begin{cases}
			J_1(z), \qquad & \textrm{$\arg z < \varphi$ and $\arg (z-s) > \varphi$,} \\
			J_5(z)^{-1}, \qquad & \textrm{$\arg z >- \varphi$ and $\arg (z-s) <- \varphi$,} \\
			I, \qquad & \textrm{elsewhere,}
		\end{cases}
	\end{equation}
	where $\what M$ solves  RH problem \ref{rhp:hard edge tac}, $\varphi$ is given in \eqref{phi} and $J_k$, $k=1, 5$, denote the jump matrix of $\what M$ on the ray $\Gamma_k$ as shown in Figure \ref{fig:hard edge tacnode}.
	
	\subsection*{Local parametrix}
	For $|z|<\delta$, which particularly includes a neighborhood of $(0, s)$, we approximate $X$ by the following local parametrix.
	\begin{rhp}\label{rhp:widecheckP0}
		\hfill
		\begin{itemize}
			\item [\rm{(a)}] $\widecheck P^{(0)}(z)$ is defined and analytic in $D(0, \delta) \setminus \Gamma_{X}$, where $\Gamma_{X}$ is defined in \eqref{def:gammaX}.
			\item [\rm{(b)}]  $\widecheck P^{(0)}$ satisfies 
				$\widecheck P^{(0)}_+(z) = \widecheck P^{(0)}_-(z) J_{X}(z)$
		for $z \in D(0, \delta) \cap \Gamma_{X}$,
			where $J_{X}(z)$ is given in \eqref{def:JX}.
			\item [\rm{(c)}] As $s \to 0^+$, we have the matching condition
			\begin{equation}\label{eq:widecheckP0:matching}
				\widecheck P^{(0)}(z) = (I+ \Boh(s^{\frac{2\nu+1}{2}})) \widecheck{N}(z), \qquad z \in \partial D(0, \delta),
			\end{equation}
			where $\widecheck{N}(z)$ is given in \eqref{eq:widecheckN}.
		\end{itemize}
	\end{rhp}
	Recall that  $\Omega_k^{(s)}$, $k=1,\ldots,6$, are six regions as shown in Figure \ref{fig:X}, we look for a solution to the above RH problem of the following form:
	\begin{align}\label{eq:widecheckP0}
			\widecheck P^{(0)}(z) &= M_0(z) 
			\begin{pmatrix}
				z^{\frac{2\nu-1}{4}} & \delta_{\nu}(z) & e^{-\nu\pi \mathrm{i}}\delta_{\nu}(z)+z^{-\frac{2\nu-1}{4}}\eta \left(\frac{z}{s}\right)  & 0 \\
				0 & z^{-\frac{2\nu-1}{4}} & 0 & 0 \\
				0 & 0 & z^{-\frac{2\nu-1}{4}} & 0 \\
				0 & 0 & -\delta_{\nu}(z) & z^{\frac{2\nu-1}{4}}
			\end{pmatrix} \nonumber\\
			&\quad\times\begin{cases}
				J_1(z)^{-1}, \quad & z \in \Omega_0^{(s)},\\
				I, \quad & z \in \Omega_1^{(s)},\\
				J_2(z)^{-1}, \quad & z \in \Omega_2^{(s)},\\
				J_1(z)^{-1}J_0(z)^{-1}J_5(z)^{-1}J_4(z), \quad & z \in \Omega_3^{(s)},\\
				J_1(z)^{-1}J_0(z)^{-1}J_5(z)^{-1}, \quad & z \in \Omega_4^{(s)},\\
				J_1^{-1}(z)J_0^{-1}(z), \quad & z \in \Omega_5^{(s)},
			\end{cases}
	\end{align}
where $M_0(z)$ and $\delta_{\nu}(z)$ are given in \eqref{def:M_0} and \eqref{def:delta}, respectively, and $\eta(z)$ is a function to be determined later. In view of RH problem \ref{rhp:widecheckP0}, it follows that $\eta$ solves the following scalar RH problem.
	\begin{rhp}
		\hfill
		\begin{itemize}
			\item [\rm{(a)}] $\eta (z)$ is defined and analytic in $\mathbb{C} \setminus [0, 1]$.
			\item [\rm{(b)}] For $x \in (0, 1)$, we have $\eta_+(x) = \eta_-(x) -\gamma\left(sx\right)^{\frac{2\nu-1}{2}}$.
			\item [\rm{(c)}] As $z \to \infty$, we have
			$\eta(z) = \Boh (z^{-1})$.
		\end{itemize}
	\end{rhp}
	By the Sokhotske-Plemelj formula, it is easy to find
	\begin{equation}\label{def:eta}
		\eta (z) = - \frac{\gamma s^{\frac{2\nu-1}{2}}}{2 \pi \ii} \int_{0}^{1} \frac{x^{\frac{2\nu-1}{2}}}{x-z} \, dx.
	\end{equation}
	
	Since $\eta(z/s) = \Boh(s^{\frac{2\nu+1}{2}})$ as $ s \to 0^+$ for $z \in \partial D(0, \delta)$, we deduce the matching condition \eqref{eq:widecheckP0:matching} from \eqref{eq:widecheckN}, \eqref{eq:widecheckP0} and the fact that $M_0(z)$ is bounded near the origin.

	\subsection*{Final transformation}
	We define the final transformation as
	\begin{equation}\label{def:widecheckR}
		\widecheck R(z) = \begin{cases}
			X(z) \widecheck{N}(z)^{-1}, \qquad & z \in \mathbb{C} \setminus D(0, \delta),\\
			X(z) \widecheck P^{(0)}(z)^{-1}, \qquad & z \in D(0, \delta).
		\end{cases}
	\end{equation}
	From the RH problems for $X$ and $\widecheck P^{(0)}$, it is readily seen $\widecheck R$ is analytic in $D(0, \delta) \setminus \{0, s\}$.
	On account of the facts that
	\begin{equation}
		\eta (z/s) = \Boh(\ln{(z-s)}),  \quad z \to s,
	\end{equation}
 and
	\begin{equation}
		\eta (z/s) =\begin{cases}
			\gamma e^{-\nu \pi \ii} \frac{z^{\nu-1/2}}{2\sin (\nu-\frac{1}{2})\pi}+\Boh(1) , &\quad \text{if} \quad \nu-\frac{1}{2} \neq 0,\\
			\frac{ \gamma}{2\pi \ii}\ln z+\Boh(1), &\quad \text{if} \quad \nu-\frac{1}{2} =0,
		\end{cases} \qquad z \to 0,
	\end{equation}
we conclude from \eqref{eq:X-near-s}, \eqref{eq:X-near-0} and \eqref{eq:widecheckP0} that  $0$ and $s$ are removable singularities. Moreover, it is readily seen that $\widecheck R$ solves the following RH problem.
	\begin{rhp}
		\hfill
		\begin{itemize}
			\item [\rm{(a)}] $\widecheck R(z)$ is defined and analytic in $\mathbb{C} \setminus \partial D(0, \delta)$.
			\item [\rm{(b)}] For $z \in \partial D(0, \delta)$, we have
			$\widecheck R_+(z) = \widecheck R_-(z) J_{\widecheck{R}}(z)$,
			where
			\begin{equation}\label{def:JwidecheckR}
				J_{\widecheck{R}}(z) = \widecheck P^{(0)}(z) \widecheck{N}(z)^{-1}.
			\end{equation}
			\item [\rm{(c)}] As $z \to \infty$, we have
			$\widecheck R(z) = I + \Boh(z^{-1})$.
		\end{itemize}
	\end{rhp}
With $J_{\widecheck R}$ defined in \eqref{def:JwidecheckR}, we have, as $s\to 0^+$, 
	\begin{equation}\label{eq:estJwidecheckR}
		J_{\widecheck R}(z) = I + \widecheck{J_1}(z) s^{\frac{2\nu+1}{2}}+ \Boh\left(s^{\frac{2\nu+3}{2}}\right), \qquad z \in \partial D(0, \delta),
	\end{equation}
	where 
	\begin{align}\label{def:widecheckJ1}
		\widecheck J_1(z)=\frac{\gamma}{(2\nu+1)\pi \ii z}M_0(z)\begin{pmatrix}
			0 & 0 & 1  & 0 \\
			0 & 0 & 0 & 0 \\
			0 & 0 & 0 & 0 \\
			0 & 0 & 0 & 0
		\end{pmatrix}M_0(z)^{-1}.
	\end{align}
	
	By \cite{Deift1999, Deift1993}, the estimate \eqref{eq:estJwidecheckR} implies that
	\begin{equation}\label{eq:widecheckRexpansion}
		\widecheck  R(z) = I + \widecheck  R_1(z) s^{\frac{2\nu+1}{2}}+ \Boh \left(s^{\frac{2\nu+3}{2}}\right), \qquad s \to 0^+,
	\end{equation}
	uniformly for $z \in \mathbb{C}\setminus \partial D(0, \delta)$. Moreover, $\widecheck R_1$ satisfies the following RH problem.
	\begin{rhp}
		\hfill
		\begin{itemize}
			\item [\rm{(a)}] $\widecheck  R_1(z)$ is analytic in $\mathbb{C} \setminus \partial D(0, \delta)$.
			\item [\rm{(b)}] For $z \in \partial D(0, \varepsilon)$, we have
			$\widecheck  R_{1,+}(z)-\widecheck  R_{1,-}(z)=\widecheck J_1(z)$,
			where $\widecheck J_1(z)$ is given in \eqref{def:widecheckJ1}.
			\item [\rm{(c)}] As $z \to +\infty$, we have
			$\widecheck  R_1(z) = \Boh (z^{-1})$.
		\end{itemize}
	\end{rhp}
	By Cauchy's residue theorem, we have
	\begin{align}
		\widecheck R_1(z) = \frac{1}{2 \pi \ii} \int_{\partial D(0, \delta)} \frac{\widecheck J_1(\zeta)}{\zeta -z} \ud \zeta
		=\begin{cases}
			\frac{\Res_{\zeta =0}\widecheck J_1(\zeta)}{z} - \widecheck J_1(z), & \quad z \in D(0, \delta),\\
			\frac{\Res_{\zeta =0}\widecheck J_1(\zeta)}{z}, & \quad \textrm{elsewhere}.
		\end{cases}
	\end{align} 
	Particularly, we note that, as $s \to 0^+$, 
	\begin{align}\label{eq:estwidecheckR1}
		\widecheck R_1(z) =\frac{\gamma}{(2\nu+1)\pi \ii}\left(-M_0'(0)E_{13} M_0(0)^{-1}+M_0(0) E_{13} M_0(0)^{-1}M_0'(0)M_0(0)^{-1}+\Boh(s)\right).
	\end{align} 
	
\section{Large and small $s$ asymptotics of $\sum p_k(s) \frac{\partial q_k(s)}{\partial \beta}$} \label{sec:pq}
	
 We have defined the functions $p_k(s)$ and $q_k(s)$, $k=1,\ldots,12$, in \eqref{def:q1p1}, \eqref{def:p5} and\eqref{def:q5}, which satisfy the equations  \eqref{def:pq's} and \eqref{eq:sumpq}. It is the aim of this section to derive large and small $s$ asymptotics of $\sum p_k(s) \frac{\partial q_k(s)}{\partial \beta}$ for $0< \gamma < 1$. These asymptotic formulas are essential in the proof of large gap asymptotics of $F(s;\gamma)$.
	\begin{proposition}\label{th:pq}
		For the purely imaginary parameter $\beta$ given in \eqref{def:beta}, we have, as $s\to +\infty$, 
		\begin{align}
			\sum_{k=1}^{4} p_k(s) \frac{\partial q_k(s)}{\partial \beta} = 2\beta \ii \frac{\partial \vartheta(s)}{\partial \beta}-\frac{4\beta \ii}{3}\sin\left(2\vartheta(s)\right)+\frac{10}{3}\beta+\Boh(s^{-\frac{1}{4}}),
		\end{align}
		and
		\begin{multline}
		    \sum_{k=5}^{12} p_k(s) \frac{\partial q_k(s)}{\partial \beta} = \frac{\ii}{3}s^{3/4}-3\ii \tilde{s} s^{1/4}+C_{\nu}(\tilde{s},\tau)
      \\ +\frac{4\beta \ii}{3} \sin(2\vartheta(s))+\frac{\partial\left(\beta \ii \cos(2\vartheta(s))-\beta^2\right)}{4\partial \beta}+\Boh(s^{-\frac{1}{4}}),
		\end{multline}
where the constnat $C_{\nu}(\tilde{s},\tau)$ is given in \eqref{def:C} below.
		
  As $s \to 0^+$, we have 
		\begin{align}
			\sum_{k=1}^{12} p_k(s) \frac{\partial q_k(s)}{\partial \beta} =\Boh(s^{\frac{2\nu+1}{2}}).
		\end{align}
		
	\end{proposition}
	\begin{proof}
		We split the proof into two parts, which deal with the large and small $s$ asymptotics, respectively.
  
		\subsubsection*{Asymptotics of $\sum p_k(s) \frac{\partial q_k(s)}{\partial \beta}$ as $s \to +\infty$}
		We first consider the asymptotics of $p_k$ and $q_k$ for $k=1, \ldots, 4$.
		Recall that
		\begin{equation}\label{pq2}
			\begin{pmatrix}
				q_1(s)\\q_2(s)\\q_3(s)\\q_4(s)
			\end{pmatrix}=X_{R,0}(s) \begin{pmatrix}
				1\\0\\0\\0
			\end{pmatrix} \quad \textrm{and} \quad \begin{pmatrix}
				p_1(s)\\p_2(s)\\p_3(s)\\p_4(s)
			\end{pmatrix}=-\frac{\gamma}{2 \pi \ii}X_{R,0}(s)^{-\msf T} \begin{pmatrix}
				0\\0\\1\\0
			\end{pmatrix},
		\end{equation}
		where
		\begin{align}\label{eq:XR0s}
			X_{R,0}(s) = \lim_{z \to 1 ~ \textrm{from} ~  \Omega_1^{(1)}} X(sz)
			\begin{pmatrix}
				1 & 0 & \frac{\gamma}{2 \pi \ii} \ln(sz-s) & 0\\	
				0 & 1 & 0 & 0\\
				0 & 0 & 1 & 0\\
				0 & 0 & 0 & 1
			\end{pmatrix}.
		\end{align}
		
		Tracing back the transformations $X \to \what{T} \to \what S \to \what R$ in \eqref{def:XToTgamma}, \eqref{def:hatS} and \eqref{def:hatR}, it follows that, as $z\to 1$ from $\Omega_1^{(1)} \setminus \Omega_{\mathsf{R},+}$, 
		\begin{align}
			X(sz) & =  \operatorname{diag}\left(s^{\frac{1}{8}}, s^{-\frac{3}{8}},s^{\frac{3}{8}}, s^{-\frac{1}{8}}\right) \what T(z)  \nonumber \\ 
			&\quad \times
              \diag \left(e^{-\theta_1(\sqrt{sz})+\tau \sqrt{sz}}, e^{-\theta_2(\sqrt{sz})-\tau \sqrt{sz}}, e^{\theta_1(\sqrt{sz})+\tau \sqrt{sz}}, e^{\theta_2(\sqrt{sz})-\tau \sqrt{sz}} \right) \nonumber\\
			& = \operatorname{diag}\left(s^{\frac{1}{8}}, s^{-\frac{3}{8}},s^{\frac{3}{8}}, s^{-\frac{1}{8}}\right) \what S(z) 
              \nonumber \\ 
			&\quad \times \diag \left(e^{-\theta_1(\sqrt{sz})+\tau \sqrt{sz}}, e^{-\theta_2(\sqrt{sz})-\tau \sqrt{sz}}, e^{\theta_1(\sqrt{sz})+\tau \sqrt{sz}}, e^{\theta_2(\sqrt{sz})-\tau \sqrt{sz}} \right)\nonumber\\
			& = \operatorname{diag}\left(s^{\frac{1}{8}}, s^{-\frac{3}{8}},s^{\frac{3}{8}}, s^{-\frac{1}{8}}\right)\what R(z) \what P^{(1)}(z) \nonumber \\ 
			&\quad \times  \diag \left(e^{-\theta_1(\sqrt{sz})+\tau \sqrt{sz}}, e^{-\theta_2(\sqrt{sz})-\tau \sqrt{sz}}, e^{\theta_1(\sqrt{sz})+\tau \sqrt{sz}}, e^{\theta_2(\sqrt{sz})-\tau \sqrt{sz}} \right).
		\end{align}
		This, together with \eqref{def:hatP1} and \eqref{eq:XR0s}, implies that
		\begin{align}
			X_{R,0}(s) & =  \operatorname{diag}\left(s^{\frac{1}{8}}, s^{-\frac{3}{8}},s^{\frac{3}{8}}, s^{-\frac{1}{8}}\right) \what{R}(1) \what E_1(1)\nonumber\\
			&\quad \times\lim_{z \to 1 ~ \textrm{from} ~  \Omega_1^{(1)}} \left[\begin{pmatrix}
				\Phi^{(\CHF)}_{11}(s^{\frac 34} \what f_{1}(z); \beta) & 0 & \Phi^{(\CHF)}_{12}(s^{\frac 34} \what f_{1}(z); \beta) & 0\\
				0 & 1 &0&0\\
				\Phi^{(\CHF)}_{21}(s^{\frac 34} \what f_{1}(z); \beta) & 0 & \Phi^{(\CHF)}_{22}(s^{\frac 34} \what f_{1}(z); \beta) & 0\\
				0 & 0 & 0 & 1
			\end{pmatrix}\nonumber\right.\\
			&\quad \times \diag \left(e^{\theta_1(\sqrt{sz}) - \frac{\beta \pi \ii}{2}}, 1, e^{-\theta_1(\sqrt{sz}) + \frac{\beta \pi \ii}{2}}, 1\right) \nonumber\\
			&\quad \times  \diag \left(e^{-\theta_1(\sqrt{sz})+\tau \sqrt{sz}}, e^{-\theta_2(\sqrt{sz})-\tau \sqrt{sz}}, e^{\theta_1(\sqrt{sz})+\tau \sqrt{sz}}, e^{\theta_2(\sqrt{sz})-\tau \sqrt{sz}} \right)\nonumber\\
			&\quad \times \left.\begin{pmatrix}
				1 & 0 & \frac{\gamma}{2 \pi \ii} \ln(sz-s) & 0\\
				0 & 1 & 0 & 0\\
				0 & 0 & 1 & 0\\
				0 & 0 & 0 & 1
			\end{pmatrix}\right],
		\end{align}
		where $ \what E_1$ and $\what f_{1}$ are defined in \eqref{def:hatE1} and \eqref{def:hatf1}, respectively.  On account of the behavior of $\Phi^{(\CHF)}$ near the origin given in \eqref{eq:H-expand-2}, it is readily seen that
		\begin{align}\label{eq:XR0}
			X_{R,0}(s) &=  \operatorname{diag}\left(s^{\frac{1}{8}}, s^{-\frac{3}{8}},s^{\frac{3}{8}}, s^{-\frac{1}{8}}\right)\what{R}(1) \what E_1(1)\what{\Upsilon}_0\nonumber\\
			& \quad\times \begin{pmatrix}
				e^{\tau \sqrt{s}} & 0 & e^{\tau \sqrt{s}} \frac{\gamma}{2 \pi \ii}\left(\ln s - \ln (e^{-\frac{\pi \ii}{2}} s^{\frac 34}\what f_{1}'(1))\right) & 0\\
				0 & e^{-\theta_2(\sqrt{s})-\tau \sqrt{s}} & 0 & 0\\
				0 & 0 & e^{\tau \sqrt{s}} & 0\\
				0 & 0 & 0 & e^{\theta_2(\sqrt{s})-\tau \sqrt{s}}
			\end{pmatrix},
		\end{align}
		where
		\begin{equation}\label{def:whatUpsilon}
			\what{\Upsilon}_0 = \begin{pmatrix}
				\Upsilon_{0,11} & 0 & \Upsilon_{0,12} & 0\\
				0 & 1 & 0 & 0\\
				\Upsilon_{0,21} & 0 & \Upsilon_{0,22} & 0\\
				0 & 0 & 0 & 1
			\end{pmatrix}
		\end{equation}
		with $\Upsilon_0$ defined in \eqref{eq:H-expand-coeff-0}.
		
Inserting \eqref{eq:XR0} into  \eqref{pq2}, it follows that
		\begin{align}\label{def:q1234}
			\begin{pmatrix}
				q_1(s)\\q_2(s)\\q_3(s)\\q_4(s)
			\end{pmatrix}&=e^{\tau \sqrt{s} } \operatorname{diag}\left(s^{\frac{1}{8}}, s^{-\frac{3}{8}},s^{\frac{3}{8}}, s^{-\frac{1}{8}}\right)\what{R}(1) \what E_1(1)\what{\Upsilon}_0\begin{pmatrix}
				1 \\0\\0\\0
			\end{pmatrix},
   \\
   \begin{pmatrix}
				p_1(s)\\p_2(s)\\p_3(s)\\p_4(s)
			\end{pmatrix}&=-e^{-\tau \sqrt{s} }\frac{\gamma}{2 \pi \ii} \operatorname{diag}\left(s^{-\frac{1}{8}}, s^{\frac{3}{8}},s^{-\frac{3}{8}}, s^{\frac{1}{8}}\right)\what{R}(1)^{-\msf T}\what E_1(1)^{-\msf T}\what{\Upsilon}_0^{-\msf T}\begin{pmatrix}
				0\\0\\1\\0
			\end{pmatrix}. \label{def:p1234}
\end{align}
Moreover, we see from \eqref{eq:hatRexpansion} that
		\begin{align}\label{eq:hatR-1}
			\what R(1) = I + \Boh(s^{-\frac{1}{4}}).
		\end{align}
This, together with \eqref{eq:hatE11} and \eqref{def:whatUpsilon}--\eqref{def:p1234}, implies that
		\begin{align}\label{eq:p1q1infinite}
			&\sum_{k=1}^{4} p_k(s) \frac{\partial q_k(s)}{\partial \beta} =-\frac{\gamma}{2 \pi \ii}\begin{pmatrix}
				0 &0 &1 &0 
			\end{pmatrix}
			\what{\Upsilon}_0^{-1} \what E_1(1)^{-1} \frac{\partial \left(\what E_1(1) \what{\Upsilon}_0\right)}{\partial \beta} \begin{pmatrix}
				1 \\ 0 \\ 0 \\0
			\end{pmatrix}+\Boh(s^{-\frac{1}{4}}) \nonumber \\
			&= \frac{\gamma}{2 \pi \ii}\bigg[2\frac{\beta \pi}{\sin (\beta \pi) }e^{-\beta \pi \ii}\left(\ln8(1-\frac{\tilde{s}}{\sqrt{s}})+\frac{3}{4}\ln s+\frac{2\ii}{3}\sin\left(2\vartheta(s)\right)-\frac{5}{3}\right)\nonumber \\
			&\quad -\Gamma^{'}(1+\beta)\Gamma(1-\beta)e^{-\beta \pi \ii}+\Gamma(1+\beta)\Gamma^{'}(1-\beta)e^{-\beta \pi \ii}\bigg]+\Boh(s^{-\frac{1}{4}}) \nonumber \\
			&= -2\beta\left(\ln8(1-\frac{\tilde{s}}{\sqrt{s}})+\frac{3}{4}\ln s+\frac{2\ii}{3}\sin\left(2\vartheta(s)\right)-\frac{5}{3}\right)+2\beta\ii \frac{\partial \arg\Gamma(1+\beta)}{\partial \beta}+\Boh(s^{-\frac{1}{4}})\nonumber \\
			&= 2\beta \ii \frac{\partial \vartheta(s)}{\partial \beta}-\frac{4\beta \ii}{3}\sin\left(2\vartheta(s)\right)+\frac{10}{3}\beta+\Boh(s^{-\frac{1}{4}}),
		\end{align}
		where $\vartheta(s)$ is given in \eqref{def:theta}.
		
		Next, recall the definitions of $p_k$ and $q_k$, $k= 5, \ldots, 12$, in \eqref{def:p5} and \eqref{def:q5}, where
		\begin{align}\label{eq:XL0s}
			X_{L,0}(s) = \lim_{z \to 0~ \textrm{from} ~ \Omega_1^{(1)}} X(z)
			\begin{pmatrix}
				z^{-\frac{2\nu-1}{4}} & -\delta_{\nu}(z) & (\gamma-1) e^{-\nu\pi \mathrm{i}}\delta_{\nu}(z) & 0 \\
				0 & z^{\frac{2\nu-1}{4}} & 0 & 0 \\
				0 & 0 & z^{\frac{2\nu-1}{4}} & 0 \\
				0 & 0 & \delta_{\nu}(z) & z^{-\frac{2\nu-1}{4}}
			\end{pmatrix}
		\end{align}
		with $\delta_{\nu}(z)$  shown in $\eqref{def:delta}$. From now on, it is assumed that $\nu \neq 1/2$. 
		Tracing back the transformations $X \to \what{T} \to \what S \to \what R$ in \eqref{def:XToTgamma}, \eqref{def:hatS} and \eqref{def:hatR}, it follows that, as $z \to 0$ from $ \Omega_1^{(1)} \setminus \Omega_{\mathsf{R},+}$, 
		\begin{align}
			X(sz) & =\operatorname{diag}\left(s^{\frac{1}{8}}, s^{-\frac{3}{8}},s^{\frac{3}{8}}, s^{-\frac{1}{8}}\right) \what T(z)  \nonumber\\
			&\quad \times 
   \diag \left(e^{-\theta_1(\sqrt{sz})+\tau \sqrt{sz}}, e^{-\theta_2(\sqrt{sz})-\tau \sqrt{sz}}, e^{\theta_1(\sqrt{sz})+\tau \sqrt{sz}}, e^{\theta_2(\sqrt{sz})-\tau \sqrt{sz}} \right)\nonumber\\
			& =\operatorname{diag}\left(s^{\frac{1}{8}}, s^{-\frac{3}{8}},s^{\frac{3}{8}}, s^{-\frac{1}{8}}\right) \what S(z) 
   \nonumber\\
			&\quad \times \diag \left(e^{-\theta_1(\sqrt{sz})+\tau \sqrt{sz}}, e^{-\theta_2(\sqrt{sz})-\tau \sqrt{sz}}, e^{\theta_1(\sqrt{sz})+\tau \sqrt{sz}}, e^{\theta_2(\sqrt{sz})-\tau \sqrt{sz}} \right)\nonumber\\
			& =\operatorname{diag}\left(s^{\frac{1}{8}}, s^{-\frac{3}{8}},s^{\frac{3}{8}}, s^{-\frac{1}{8}}\right)\what R(z) \what P^{(0)}(z) \nonumber\\
			&\quad \times \diag \left(e^{-\theta_1(\sqrt{sz})+\tau \sqrt{sz}}, e^{-\theta_2(\sqrt{sz})-\tau \sqrt{sz}}, e^{\theta_1(\sqrt{sz})+\tau \sqrt{sz}}, e^{\theta_2(\sqrt{sz})-\tau \sqrt{sz}} \right).
		\end{align}
		This, together with \eqref{def:hatP0} and \eqref{eq:XL0s}, implies that
		\begin{align}\label{eq:XL0}
			&X_{L,0}(s)= \operatorname{diag}\left(s^{\frac{1}{8}}, s^{-\frac{3}{8}},s^{\frac{3}{8}}, s^{-\frac{1}{8}}\right) \what{R}(0) \what E_0(0)\lim_{z \to 0 ~ \textrm{from} ~  \Omega_1^{(1)}} \what M(z) \nonumber\\
			& \times\diag\left((1-\gamma)^{-\frac{1}{2}},(1-\gamma)^{-\frac{1}{2}},(1-\gamma)^{\frac{1}{2}},(1-\gamma)^{\frac{1}{2}}\right)    \begin{pmatrix}
				z^{-\frac{2\nu-1}{4}} & -\delta_{\nu}(z) & (\gamma-1) e^{-\nu\pi \mathrm{i}}\delta_{\nu}(z) & 0 \\
				0 & z^{\frac{2\nu-1}{4}} & 0 & 0 \\
				0 & 0 & z^{\frac{2\nu-1}{4}} & 0 \\
				0 & 0 & \delta_{\nu}(z) & z^{-\frac{2\nu-1}{4}}
			\end{pmatrix} \nonumber \\ 
			&= \operatorname{diag}\left(s^{\frac{1}{8}}, s^{-\frac{3}{8}},s^{\frac{3}{8}}, s^{-\frac{1}{8}}\right)  \what{R}(0) \what E_0(0) M_0(0) \diag\left((1-\gamma)^{-\frac{1}{2}},(1-\gamma)^{-\frac{1}{2}},(1-\gamma)^{\frac{1}{2}},(1-\gamma)^{\frac{1}{2}}\right).
		\end{align}
	Then, we have
		\begin{align}
		&	X_{L,0}(s)^{-1}\frac{\partial X_{L,0}(s)}{\partial  \beta} \nonumber 
  \\
  &=\diag\left((1-\gamma)^{\frac{1}{2}},(1-\gamma)^{\frac{1}{2}},(1-\gamma)^{-\frac{1}{2}},(1-\gamma)^{-\frac{1}{2}}\right)M_0(0)^{-1}\what E_0(0)^{-1} \what R(0)^{-1} \nonumber \\ 
			&\quad \times \frac{\partial\left( \what R(0) \what E_0(0)\right)}{\partial \beta} M_0(0) \diag\left((1-\gamma)^{-\frac{1}{2}},(1-\gamma)^{-\frac{1}{2}},(1-\gamma)^{\frac{1}{2}},(1-\gamma)^{\frac{1}{2}}\right) \nonumber \\ 
			&\quad +\diag \left(-\pi \ii,-\pi \ii, \pi \ii ,\pi \ii \right) .
		\end{align}
To further evaluate right hand side of the above formula, we first note from \eqref{eq:hatR1} that 
		\begin{align}\label{8.18}
			&\what E_0(0)^{-1} \what R(0)^{-1} \frac{\partial\left( \what R(0) \what E_0(0)\right)}{\partial \beta}= \what E_0(0)^{-1} \what R(0)^{-1} \frac{\partial \what R(0)  }{\partial \beta}\what E_0(0)+\what E_0(0)^{-1} \frac{\partial \what E_0(0) }{\partial \beta} \nonumber \\
			&=\what E_0(0)^{-1}\bigg[s^{-\frac{1}{4}}\frac{\partial \what R_1(0)}{\partial \beta}+s^{-\frac{1}{2}}\left(\frac{\partial \what R_2(0)}{\partial \beta}-\what R_1(0)\frac{\partial R_1(0)}{\partial \beta} \right)+s^{-\frac{3}{4}}  \nonumber \\
			& \quad \times \left(\frac{\partial \what R_3(0)}{\partial \beta}-\what R_1(0)\frac{\partial \what R_2(0)}{\partial \beta}- (\what R_2(0)-\what R_1(0)^2)\frac{\partial \what R_1(0)}{\partial \beta}\right)\bigg]
			\what E_0(0) +\what E_0(0)^{-1} \frac{\partial \what E_0(0) }{\partial \beta}
       \nonumber 
       \\
       & \quad +\Boh(s^{-\frac{1}{4}}).
		\end{align}
We then calculate right hand side of \eqref{8.18} term by term. For the last one, we see from \eqref{def:hatE0}
that
\begin{multline}\label{eq:lastcal}
			\what E_0(0)^{-1} \frac{\partial \what E_0(0) }{\partial \beta}=\begin{pmatrix}
				\frac{8 \beta^3}{3} & \frac{4 \beta }{3}s^{\frac{1}{2}}& 0 & \left(\frac{8 \beta^2}{3}-2\right) s^{\frac{1}{4}} \\
				0 & \frac{8}{9} \beta \left(\beta^2+5\right) & 0 & 0 \\
				\left(\frac{4 \beta^2}{3}+2\right) s^{\frac{1}{4}} & \frac{2}{3}s^{\frac{3}{4}} & -\frac{8}{9} \beta \left(\beta^2+5\right) & \frac{4\beta}{3}s^{\frac{1}{2}} \\
				0 & -\frac{2}{3} \left(2 \beta^2+3\right) s^{\frac{1}{4}} & 0& -\frac{8 \beta^3}{3}  \\
			\end{pmatrix}
   \\ +\Boh(s^{-\frac{1}{4}}).
		\end{multline}
For the first term, it follows from \eqref{def:hatE0} and \eqref{eq:hatR1} that 
		\begin{align}
		& s^{-\frac{1}{4}}\what E_0(0)^{-1}\frac{\partial \what R_1(0)}{\partial \beta}\what E_0(0)
   \nonumber \\
   & =2s^{\frac{1}{2}}(M^{(1)}_{13}+M^{(1)}_{14}) E_{32}\Xi_{32} \Psi_{32} +s^{\frac{1}{4}}(M^{(1)}_{13}-M^{(1)}_{14})\left(E_{34}\Xi_{31}-E_{12} \Xi_{42}\right) \nonumber \\
			& \quad+s^{\frac{1}{4}}(M^{(1)}_{13}+M^{(1)}_{14})E_{12}\left(\Xi_{12} 
			\Psi_{32} +\Xi_{32} \Psi_{12}\right) + s^{\frac{1}{4}}(M^{(1)}_{13}+M^{(1)}_{14})E_{34}\left(\Xi_{32} 
			\Psi_{34} +\Xi_{34} \Psi_{32}\right)\nonumber \\
			&\quad+(M^{(1)}_{13}-M^{(1)}_{14})E_{14} \left(\Xi_{11} 
			-\Xi_{44}\right) +(M^{(1)}_{13}+M^{(1)}_{14})E_{14}\left(\Xi_{12} 
			\Psi_{34} +\Xi_{34} \Psi_{12}\right) \nonumber \\
			&\quad +(M^{(1)}_{13}+M^{(1)}_{14})E_{31}\left(\Xi_{32} \Psi_{31}+\Xi_{31} \Psi_{32}\right) +(M^{(1)}_{13}+M^{(1)}_{14})E_{42}\left(\Xi_{32} \Psi_{42}+\Xi_{42} \Psi_{32}\right) \nonumber \\
			&=-\frac{8\beta}{3}s^{\frac{1}{2}}(M^{(1)}_{13}+M^{(1)}_{14}) E_{32}+\frac{4\beta^2+6}{3}s^{\frac{1}{4}}(M^{(1)}_{13}-M^{(1)}_{14})(E_{34}+E_{12})\nonumber \\
			&\quad-4\beta^2s^{\frac{1}{4}}(M^{(1)}_{13}+M^{(1)}_{14})(E_{34}+E_{12})+ \frac{16 \beta^3}{3} (M^{(1)}_{13}-M^{(1)}_{14}) E_{14}-\frac{16\beta^3}{3}(M^{(1)}_{13}+M^{(1)}_{14}) E_{14}\nonumber \\
			&\quad-\frac{8}{9}(4\beta^3+5\beta)(M^{(1)}_{13}+M^{(1)}_{14}) (E_{31}-E_{42})+\Boh(s^{-\frac{1}{4}}),
		\end{align}
  where $\Xi:=\widetilde E_0(0)^{-1} \frac{\partial \widetilde E_0(0)}{\partial \beta}$ and $\Psi:=\widetilde E_0(0)^{-1}  \widetilde E'_0(0)$.
  
		Similarly, we have 
		\begin{align}
			&s^{-\frac{1}{2}}\what E_0(0)^{-1}\left(\frac{\partial \what R_2(0)}{\partial \beta}-\what R_1(0)\frac{\partial \what R_1(0)}{\partial \beta} \right)\what E_0(0)
   \nonumber \\
   & =-\frac{4\beta}{9}\bigg[2(-5+2\beta^2)(M^{(1)}_{11}+M^{(1)}_{12})
   +30\beta^2(M^{(1)}_{13}+M^{(1)}_{14})^2
   \nonumber \\
   & \quad +2(5+4\beta^2)(M^{(1)}_{13}+M^{(1)}_{14})(-M^{(1)}_{13}+M^{(1)}_{14})+2(5+8\beta^2)(M^{(1)}_{33}+M^{(1)}_{34})\bigg]E_{12}
   \nonumber \\
   & \quad -\frac{2}{3}\bigg[(3+2\beta^2)(M^{(1)}_{34}-M^{(1)}_{33})-12\beta^2(M^{(1)}_{13}+M^{(1)}_{14})^2
   \nonumber\\
   &\quad-6\beta^2(M^{(1)}_{33} 
			+M^{(1)}_{34})+(3+2\beta^2)(M^{(1)}_{11}+M^{(1)}_{12})
			+(-3+22\beta^2+4\beta^4) (M^{(1)}_{11}-M^{(1)}_{12}) \bigg]s^{\frac{1}{4}} E_{32}\nonumber \\
			&\quad+\frac{8\beta}{45}\bigg[-3(7+20\beta^2+8\beta^4)(M^{(1)}_{11}-M^{(1)}_{12})-30\beta^2\left(3(M^{(1)}_{13}+M^{(1)}_{14})^2+(M^{(1)}_{33}+M^{(1)}_{34})\right)\nonumber \\
			&\quad+5(5-2\beta^2)(M^{(1)}_{34}-M^{(1)}_{33})+5(5+4\beta^2)((M^{(1)}_{13})^2-(M^{(1)}_{14})^2)\bigg]E_{34}+\Boh(s^{-\frac{1}{4}}),
		\end{align}
		and
		\begin{align}\label{eq:middcal}
		&s^{-\frac{3}{4}}\what E_0(0)^{-1}\left(\frac{\partial \what R_3(0)}{\partial \beta}-\what R_1(0)\frac{\partial \what R_2(0)}{\partial \beta}- (\what R_2(0)-\what R_1(0)^2)\frac{\partial \what R_1(0)}{\partial \beta}\right)\what E_0(0)
   \nonumber \\
   & =\bigg[\frac{8}{9} \beta \bigg(10 (M^{(1)}_{31}+M^{(1)}_{32}) \nonumber \\
			&\quad+ (14+6\beta^2)(M^{(1)}_{13}+M^{(1)}_{14})(M^{(1)}_{33}-M^{(1)}_{34}) + (5+4\beta^2) (M^{(1)}_{13}-M^{(1)}_{14}) (M^{(1)}_{33}+M^{(1)}_{34})  \nonumber \\
			&\quad+ 8\beta^2 (M^{(1)}_{31}-M^{(1)}_{32}) + 10\beta^2 (M^{(1)}_{13}+M^{(1)}_{14}) (M^{(1)}_{33}-M^{(1)}_{34})  + 4\beta^2 (M^{(1)}_{13}-M^{(1)}_{14})(M^{(1)}_{33}+M^{(1)}_{34})  \nonumber \\
			&\quad  -(5 + 4 \beta^2) (M^{(2)}_{13}+M^{(2)}_{14})  +
			(2 + 4 \beta^2)M^{(1)}_{12}(M^{(1)}_{13}+M^{(1)}_{14})+ 12\beta^2 (M^{(1)}_{13}+M^{(1)}_{14})^3  \bigg)\nonumber\\
			&\quad+\frac{8\beta \ii}{3} \sin(2\vartheta(s))+\frac{\partial\left(\beta \ii \cos(2\vartheta(s)\right)-\beta^2)}{2\partial \beta} \bigg]E_{32}+\Boh(s^{-\frac{1}{4}}).
		\end{align}
In the above formulas, the matrices $M^{(k)}$, $k=1,2$, are given in \eqref{eq:asy:Mhat}.

Inserting \eqref{eq:lastcal}--\eqref{eq:middcal} into \eqref{8.18}, it is readily seen that
		\begin{align}
			&\what E_0(0)^{-1} \what R(0)^{-1} \frac{\partial\left( \what R(0) \what E_0(0)\right)}{\partial \beta} \nonumber\\
			&=(cs^{\frac{1}{4}}+d)E_{12}+eE_{14}+fE_{31}+(gs^{\frac{1}{2}}+hs^\frac{1}{4}+j)E_{32}
   \nonumber \\ 
   & \quad +(cs^\frac{1}{4}+k)E_{34}-fE_{42}+\Boh(s^{-\frac{1}{4}}),
		\end{align}
where the functions $c, d, e, f, g, h, j, k$ are  independent of $s$ and can be given explicitly in terms of  $\beta$ and the entries of $M^{(k)}$, $k=1,2$. For instance, we have
  \begin{align}
  c&=-4\beta^2(M^{(1)}_{13}+M^{(1)}_{14}),\\
  d&=-\frac{4\beta}{9}\bigg[2(-5+2\beta^2)(M^{(1)}_{11}+M^{(1)}_{12})+30\beta^2(M^{(1)}_{13}+M^{(1)}_{14})^2 \nonumber\\
  &\quad+2(5+4\beta^2)(M^{(1)}_{13}+M^{(1)}_{14})(-M^{(1)}_{13}+M^{(1)}_{14})+2(5+8\beta^2)(M^{(1)}_{33}+M^{(1)}_{34})\bigg], \label{def:d}
  \\
  k&=\frac{8\beta}{45}\bigg[-3(7+20\beta^2+8\beta^4)(M^{(1)}_{11}-M^{(1)}_{12})-30\beta^2\left(3(M^{(1)}_{13}+M^{(1)}_{14})^2+(M^{(1)}_{33}+M^{(1)}_{34})\right)\nonumber\\
  &\quad+5(5-2\beta^2)(M^{(1)}_{34}-M^{(1)}_{33})+5(5+4\beta^2)((M^{(1)}_{13})^2-(M^{(1)}_{14})^2)\bigg]. \label{def:k}
   \end{align}
		
  
As a consequence, it follows from  the definitions of $p_k(s)$ and $q_k(s)$, $k=5,\ldots,12$, in \eqref{def:p5}, \eqref{def:q5} and direct calculations that
		\begin{multline}\label{8.25}
			\sum_{k=5}^{12} p_k(s) \frac{\partial q_k(s)}{\partial \beta} 
			=\frac{1-2\nu}{2}\bigg[\left(\frac{2}{3}s^\frac{3}{4}+gs^\frac{1}{2}+hs^\frac{1}{4}+j\right)\msf{u}+\left(\frac{4\beta}{3}s^\frac{1}{2}+cs^\frac{1}{4}+d\right)\msf{x}  \\
			 +\left(\frac{4\beta}{3}s^\frac{1}{2}+cs^\frac{1}{4}+k\right)\msf{y}-\left(\left(\frac{8\beta^2}{3}-2\right)s^\frac{1}{4}+e\right)\msf{z}+\frac{16\beta^3-40\beta}{9}\bigg]+\Boh(s^{-\frac{1}{4}}),
		\end{multline}
where
  \begin{align}
			\msf{u}&=\frac{ m_{13} m_{22} - m_{12} m_{23}}{ m_{13} m_{32} -  m_{12} m_{33}}, \qquad \msf{x}=\frac{ m_{23} m_{32} - m_{22} m_{33}}{ m_{13} m_{32} -  m_{12} m_{33}},\nonumber\\
			\msf{y}&=\frac{ m_{13} m_{42} - m_{12} m_{43}}{ m_{13} m_{32} -  m_{12} m_{33}}, \qquad \msf{z}=\frac{ m_{33} m_{42} - m_{32} m_{43}}{ m_{13} m_{32} -  m_{12} m_{33}},
		\end{align}
with $m_{ij}$ being the entries of $M_0(0)$ given in  \eqref{defi:M00} and independent of $\beta$.
  
In Section \ref{sec:proof}, we will adopt an alternative approach to determine the asymptotics of $H(s)$. Consequently, by substituting \eqref{asy:H}, \eqref{eq:p1q1infinite} and \eqref{8.25} into \eqref{eq:asy:H3}, and  comparing the coefficients of various orders of $s$, it follows that
		\begin{align}
			&\frac{1-2\nu}{2} \msf{u}= \frac{\ii}{2}, \\
			& g \msf{u}+\frac{4\beta}{3}(\msf{x+y})=0, \\
			& \frac{1-2\nu}{2}\bigg[h\msf{u}+c(\msf{x+y})-\left(\frac{8\beta^2}{3}-2\right)\msf{z}\bigg] = -3\ii \tilde{s}. \label{eq:s1}
		\end{align}
With the aid of explicit formulas for $c,g,h$, we can solve the above equations by analyzing the powers of $\beta$ and obtain
		\begin{align}
			& \msf{u}=\frac{\ii}{1-2\nu}, \label{eq:uexplicit} \\
			&\msf{x+y}=\frac{2\ii}{1-2\nu}(M^{(1)}_{13}+M^{(1)}_{14}), \\
			&\msf{z}=\frac{\ii}{1-2\nu}(-M^{(1)}_{12} -(M^{(1)}_{14})^2+(M^{(1)}_{13})^2+2M^{(1)}_{33}+M^{(1)}_{34}  ) ,\label{def:z}
		\end{align} 
 where we have made use of \cite[Equation (149)]{Del13} in the derivation of \eqref{def:z}. We actually even get an unexpected relation, namely, 
 \begin{equation}\label{eq:M11M12}
 M^{(1)}_{11}-M^{(1)}_{12}=0,
 \end{equation}
 since the coefficient of $\beta^4$ in \eqref{eq:s1} must vanish.
		
 A combination of \eqref{8.25} and 
 \eqref{eq:uexplicit}--\eqref{def:z} yields that 
		\begin{multline}\label{eq:p5q5infinite}
			\sum_{k=5}^{12} p_k(s) \frac{\partial q_k(s)}{\partial \beta} = \frac{\ii}{3}s^{3/4}-3\ii \tilde{s} s^{1/4}+C_{\nu}(\tilde{s},\tau)+\frac{4\beta \ii}{3} \sin(2\vartheta(s)) \\ 
            +\frac{\partial(\beta \ii \cos(2\vartheta(s))-\beta^2)}{4\partial \beta}+\Boh\left(s^{-\frac{1}{4}}\right),
		\end{multline}
		where 
		\begin{align}\label{def:C}
			C_{\nu}(\tilde{s},\tau)&=\frac{4}{9} \beta \ii \bigg(10 (M^{(1)}_{31}+M^{(1)}_{32}) + (14+6\beta^2)(M^{(1)}_{13}+M^{(1)}_{14})(M^{(1)}_{33}-M^{(1)}_{34}) \nonumber \\
			&\quad+ (5+4\beta^2) (M^{(1)}_{13}-M^{(1)}_{14}) (M^{(1)}_{33}+M^{(1)}_{34})+ 12\beta^2 (M^{(1)}_{13}+M^{(1)}_{14})^3   + 8\beta^2 (M^{(1)}_{31}-M^{(1)}_{32}) \nonumber \\
			&\quad+ 10\beta^2 (M^{(1)}_{13}+M^{(1)}_{14}) (M^{(1)}_{33}-M^{(1)}_{34})  + 4\beta^2 (M^{(1)}_{13}-M^{(1)}_{14})(M^{(1)}_{33}+M^{(1)}_{34}) \nonumber \\
			& \quad-(5 + 4 \beta^2)  (M^{(2)}_{13}+M^{(2)}_{14}) +
			(2 + 4 \beta^2)M^{(1)}_{12}(M^{(1)}_{13}+M^{(1)}_{14}) \bigg)\nonumber\\
			&\quad+\frac{1-2\nu}{2}(d\msf{x}+k\msf{y})-\frac{16\beta^3\ii}{3}M^{(1)}_{14}
			\left(-M^{(1)}_{12} -(M^{(1)}_{14})^2+(M^{(1)}_{13})^2+2M^{(1)}_{33}+M^{(1)}_{34}\right).
		\end{align}
The case when $\nu=1/2$ can be handled in a similar manner, and we omit the details here. 

		\subsubsection*{Asymptotics of $\sum p_k(s) \frac{\partial q_k(s)}{\partial \beta}$ as $s \to 0^+$}
		The small $s$  asymptotics of $\sum p_k(s) \frac{\partial q_k(s)}{\partial \beta}$ follows from the asymptotic analysis performed in Section \ref{sec:AsyX0}. 		
	If $|z| < \delta$, we obtain from \eqref{eq:widecheckP0}--\eqref{def:widecheckR} that for $z \in \Omega_1^{(s)}$,
		\begin{align}\label{eq:Xs0}
			X(z)=\widecheck R(z) \widecheck P^{(0)}(z)
			=\widecheck R(z) M_0(z)\begin{pmatrix}
				z^{-\frac{2\nu-1}{4}} & -\delta_{\nu}(z) & e^{-\nu\pi \mathrm{i}}\delta_{\nu}(z)+z^{-\frac{2\nu-1}{4}}\eta \left(\frac{z}{s}\right)  & 0 \\
				0 & z^{\frac{2\nu-1}{4}} & 0 & 0 \\
				0 & 0 & z^{\frac{2\nu-1}{4}} & 0 \\
				0 & 0 & \delta_{\nu}(z) & z^{-\frac{2\nu-1}{4}}
			\end{pmatrix}.
		\end{align}
		This, together with  \eqref{eq:X-near-s}, \eqref{eq: X-expand-s} and \eqref{eq:widecheckRexpansion}, implies that
		\begin{align} \label{def:R00}
			X_{R,0}(s)&=\widecheck R(s) M_0(s)\lim_{z \to s}  \begin{pmatrix}
				z^{\frac{2\nu-1}{4}} & \delta_{\nu}(z) & e^{-\nu \pi \mathrm{i}}\delta_{\nu}(z)+z^{-\frac{2\nu-1}{4}}\eta \left(\frac{z}{s}\right)  & 0 \\
				0 & z^{-\frac{2\nu-1}{4}} & 0 & 0 \\
				0 & 0 & z^{-\frac{2\nu-1}{4}} & 0 \\
				0 & 0 & -\delta_{\nu}(z) & z^{\frac{2\nu-1}{4}}
			\end{pmatrix} \nonumber\\
   &\quad \times \left(I+\frac{\gamma}{2 \pi \ii} \ln (z-s) E_{13}\right)\nonumber\\
			&=\left(I+\Boh(s^{\frac{2\nu+1}{2}})\right)M_0(s) \nonumber\\
			&\quad \times \lim_{z \to s} \begin{pmatrix}
				z^{\frac{2\nu-1}{4}} & \delta_{\nu}(z) & e^{-\nu\pi \mathrm{i}}\delta_{\nu}(z)+z^{-\frac{2\nu-1}{4}}\eta \left(\frac{z}{s}\right)+ \frac{\gamma}{2 \pi \ii} \ln (z-s) z^{\frac{2\nu-1}{4}} & 0 \\
				0 & z^{-\frac{2\nu-1}{4}} & 0 & 0 \\
				0 & 0 & z^{-\frac{2\nu-1}{4}} & 0 \\
				0 & 0 & -\delta_{\nu}(z) & z^{\frac{2\nu-1}{4}}
			\end{pmatrix}.
		\end{align}
		As a consequence, it follows from  the definitions of $p_k(s)$ and $q_k(s)$, $k=1,\ldots,4$, in \eqref{def:q1p1} and \eqref{eq:estwidecheckR1} that 
		\begin{align}\label{eq:p1q1zero}
			\sum_{k=1}^{4} p_k(s) \frac{\partial q_k(s)}{\partial \beta} &=-\frac{\gamma}{2 \pi \ii} \begin{pmatrix}
				0 & 0 & s^{\frac{2\nu-1}{4}} & 0 \end{pmatrix}  M_0(s)^{-1}\left(I-\widecheck R_1(s) s^{\frac{2\nu+1}{2}}+\Boh(s^{\frac{2\nu+3}{2}})\right)\nonumber\\
			&\quad \times  \left(\frac{\partial \widecheck R_1(s) }{\partial \beta}s^{\frac{2\nu+1}{2}}+ \Boh(s^{\frac{2\nu+3}{2}})\right) M_0(s) \begin{pmatrix}
				s^{\frac{2\nu-1}{4}} \\ 0 \\ 0 \\0
			\end{pmatrix}  \nonumber\\
			&= (1-\gamma) \begin{pmatrix}
				0 & 0 & s^{\frac{2\nu-1}{4}} & 0 \end{pmatrix} \left(M_0(0)^{-1}+\Boh(s)\right)\left(\widecheck R_1(s)-\widecheck R_1(s)^2+\Boh(s)\right)\nonumber\\
			&\quad \times 
			\left(M_0(0)+\Boh(s)\right) \begin{pmatrix}
				s^{\frac{2\nu-1}{4}} \\ 0 \\ 0 \\0
			\end{pmatrix}
			= \Boh(s^{\frac{2\nu+1}{2}}), \qquad s \to 0^+.
		\end{align}
We next consider the asymptotics of $p_k$ and $q_k$ for $k= 5, \ldots, 12$. Recall \eqref{eq:X-near-0}, \eqref{eq:X-expand-0}, \eqref{eq:widecheckRexpansion} and \eqref{eq:Xs0},  we have
		\begin{align}\label{def:L00}
			X_{L,0}(s)&=\widecheck R(0) M_0(0)\lim_{z \to 0}  \begin{pmatrix}
				z^{\frac{2\nu-1}{4}} & \delta_{\nu}(z) & e^{-\nu\pi \mathrm{i}}\delta_{\nu}(z)+z^{-\frac{2\nu-1}{4}}\eta \left(\frac{z}{s}\right)  & 0 \\
				0 & z^{-\frac{2\nu-1}{4}} & 0 & 0 \\
				0 & 0 & z^{-\frac{2\nu-1}{4}} & 0 \\
				0 & 0 & -\delta_{\nu}(z) & z^{\frac{2\nu-1}{4}}
			\end{pmatrix} \nonumber \\
			&\quad\times \begin{pmatrix}
				z^{-\frac{2\nu-1}{4}} & -\delta_{\nu}(z) & \left(\gamma-1\right)e^{-\nu\pi \mathrm{i}}\delta_{\nu}(z) & 0 \\
				0 & z^{\frac{2\nu-1}{4}} & 0 & 0 \\
				0 & 0 & z^{\frac{2\nu-1}{4}} & 0 \\
				0 & 0 & \delta_{\nu}(z) & z^{-\frac{2\nu-1}{4}}
			\end{pmatrix} \nonumber \\
			&= \widecheck R(0) M_0(0)\lim_{z \to 0}\begin{pmatrix}
				1 & 0& \gamma e^{-\nu\pi \mathrm{i}}\delta_{\nu}(z) z^{\frac{2\nu-1}{4}}+\eta \left(\frac{z}{s}\right) & 0 \\
				0 & 1 & 0 & 0 \\
				0 & 0 & 1 & 0 \\
				0 & 0 & 0 & 1
			\end{pmatrix}.
		\end{align}
		
		As a consequence, it follows from \eqref{eq:widecheckRexpansion} and the definitions of $p_k(s)$ and $q_k(s)$, $k=5,\ldots,12$, in \eqref{def:p5} and 
         \eqref{def:q5} that as $s \to 0^+$, if $\nu\neq \frac{1}{2}$,
		\begin{align}\label{eq:p5q5zero}
			\sum_{k=5}^{8} p_k(s) \frac{\partial q_k(s)}{\partial \beta} &= \frac{1-2\nu}{2}\begin{pmatrix}
				0 & 1 & 0 & 0 \end{pmatrix} M_0(0)^{-1}\left(I+\Boh(s^{\frac{2\nu+1}{2}})\right) \frac{\partial \left(I+\Boh(s^{\frac{2\nu+1}{2}})\right)M_0(0)}{\partial \beta} \begin{pmatrix}
				0 \\ 1 \\ 0 \\0
			\end{pmatrix}  \nonumber\\
			&= \frac{1-2\nu}{2}\begin{pmatrix}
				0 & 1 & 0 & 0 \end{pmatrix} M_0(0)^{-1}\Boh(s^{\frac{2\nu+1}{2}}) M_0(0) \begin{pmatrix}
				0 \\ 1 \\ 0 \\0
			\end{pmatrix}\
			= \Boh(s^{\frac{2\nu+1}{2}}),
		\end{align}
		and
		\begin{align}\label{eq:p9q9zero}
			\sum_{k=9}^{12} p_k(s) \frac{\partial q_k(s)}{\partial \beta} &= \begin{pmatrix}
				0 & 0 & \frac{1-2\nu}{2} & 0 \end{pmatrix} M_0(0)^{-1}\left(I+\Boh(s^{\frac{2\nu+1}{2}})\right) \frac{\partial \left(I+\Boh(s^{\frac{2\nu+1}{2}})\right)M_0(0)}{\partial \beta} \nonumber\\
			&\quad \times \begin{pmatrix}
				\gamma e^{-\nu \pi \mathrm{i}}\delta_{\nu}(z) z^{\frac{2\nu-1}{4}}+\eta \left(\frac{z}{s}\right)\\ 0 \\ 1 \\0
			\end{pmatrix}  \nonumber\\
			&=\frac{1-2\nu}{2}\begin{pmatrix}
				0 & 0 & 1 & 0 \end{pmatrix} M_0(0)^{-1}\Boh(s^{\frac{2\nu+1}{2}}) M_0(0) \begin{pmatrix}
				\gamma e^{-\nu\pi \mathrm{i}}\delta_{\nu}(z) z^{\frac{2\nu-1}{4}}+\eta \left(\frac{z}{s}\right) \\ 0 \\ 1 \\0
			\end{pmatrix}\nonumber\\
			&= \Boh(s^{\frac{2\nu+1}{2}});
		\end{align}
		if $\nu =\frac{1}{2}$,
		\begin{align}\label{eq:p5q5zero'}
			\sum_{k=5}^{8} p_k(s) \frac{\partial q_k(s)}{\partial \beta} &= -\frac{1}{2\pi}\begin{pmatrix}
				0 & 1 & 0 & 0 \end{pmatrix} M_0(0)^{-1}\left(I+\Boh(s^{\frac{2\nu+1}{2}})\right) \frac{\partial \left(I+\Boh(s^{\frac{2\nu+1}{2}})\right)M_0(0)}{\partial \beta} \begin{pmatrix}
				1 \\ 0 \\ 0 \\0
			\end{pmatrix}  \nonumber\\
			&= -\frac{1}{2\pi}\begin{pmatrix}
				0 & 1 & 0 & 0 \end{pmatrix} M_0(0)^{-1}\Boh(s^{\frac{2\nu+1}{2}}) M_0(0) \begin{pmatrix}
				1 \\ 0 \\ 0 \\0
			\end{pmatrix}
			= \Boh(s^{\frac{2\nu+1}{2}}),
		\end{align}
		and
		\begin{align}\label{eq:p9q9zero'}
			\sum_{k=9}^{12} p_k(s) \frac{\partial q_k(s)}{\partial \beta} &=-\frac{1}{2\pi} \begin{pmatrix}
				0 & 0 & 1 & 0 \end{pmatrix} M_0(0)^{-1}\left(I+\Boh(s^{\frac{2\nu+1}{2}})\right) \frac{\partial \left(I+\Boh(s^{\frac{2\nu+1}{2}})\right)M_0(0)}{\partial \beta} \begin{pmatrix}
				\ii (\gamma -1)\\ 0 \\ 0 \\-1
			\end{pmatrix}  \nonumber\\
			&= -\frac{1}{2\pi}\begin{pmatrix}
				0 & 0 & 1 & 0 \end{pmatrix} M_0(0)^{-1}\Boh(s^{\frac{2\nu+1}{2}}) M_0(0) \begin{pmatrix}
				\ii (\gamma -1)\\ 0 \\ 0 \\-1
			\end{pmatrix}
			 =\Boh(s^{\frac{2\nu+1}{2}}).
		\end{align}
		This completes the proof of Proposition \ref{th:pq}.
	\end{proof}
\begin{remark}
From the proof of Proposition \ref{th:pq}, it is also straightforward to derive large and small $s$ asymptotics of the functions $p_k(s)$ and $q_k(s)$, $k=1,\ldots,12$. We will not go into the details here. 
\end{remark}

	\section{Proofs of main results}\label{sec:proof}
	\subsection{Proof of Theorem \ref{th:F-H}}
 On account of \eqref{eq:derivativeinsX-2} and \eqref{eq:JMU}, it is easily seen that
\begin{equation}\label{9.1}
\frac{\ud}{\ud s} F(s;\gamma)=H(s),
\end{equation}
which leads to the integral representation of $F$ claimed in \eqref{eq:F-H2} after integrating on both sides of the  \eqref{9.1}. 

We next establish asymptotics of the Hamiltonian 
$H$ for large and small $s$.
	\subsubsection*{Asymptotics of $H(s)$ as $s \to +\infty$ for $\gamma=1$}
 Combining \eqref{eq:derivativeinsX-2} and \eqref{9.1} results in
	\begin{align}
		H(s)=-\frac{1}{2 \pi \ii}\lim_{z \to s} \left(X(z)^{-1}X'(z)\right)_{31}, \quad z \in \Omega_1^{(s)}.
	\end{align}
	By inverting the transformations $X \to T \to S$ in \eqref{def:XtoT} and \eqref{def:TtoS}, it follows from the above formula that
	\begin{align}\label{ds-F-1}
		H(s) &=-\frac{1}{2 \pi \ii s} \lim_{z \to 1}\left(T(z)^{-1}T'(z)\right)_{31}\nonumber\\
		&= -\frac{1}{2 \pi \ii s} \lim_{z \to 1} \left(\diag \left(e^{s^{\frac 34}g_1(z)-\tau \sqrt{sz}}, e^{s^{\frac 34}g_2(z)+ \tau s \sqrt{sz}}, e^{-s^{\frac 34}g_1(z)-\tau \sqrt{sz}}, e^{-s^{\frac 34}g_2(z) + \tau \sqrt{sz}} \right)\right.\nonumber\\
		&\quad\left. \times S(z)^{-1}S'(z)\diag \left(e^{-s^{\frac 34}g_1(z)+\tau \sqrt{sz}}, e^{-s^{\frac 34}g_2(z)- \tau\sqrt{sz}}, e^{s^{\frac 34}g_1(z)+\tau \sqrt{sz}}, e^{s^{\frac 34}g_2(z) -\tau\sqrt{sz}} \right)\right)_{31}\nonumber\\
  &=-\frac{1}{2 \pi \ii s} \lim_{z \to 1}\left(S(z)^{-1}S'(z)\right)_{31},
	\end{align}
	where the limits are taken from $\Omega_1^{(1)}$.
	
To evaluate the above limit, we note from \eqref{def:R} that if $z$ is close to $1$,
	\begin{align}\label{eq:SinvSder}
		& S(z)^{-1}S'(z)\nonumber\\
		&= P^{(1)}(z)^{-1} R(z)^{-1} R'(z) P^{(1)}(z) + P^{(1)}(z)^{-1} (P^{(1)})'(z) \nonumber\\
		&=\mathscr{A}_{1}(z)^{-1}\mathscr{B}_{1}(s^{\frac{3}{2}} f_1(z))^{-1}E_1(z)^{-1} R(z)^{-1} R'(z) E_1(z)\mathscr{B}_{1}(s^{\frac{3}{2}} f_1(z))\mathscr{A}_{1}(z)\nonumber\\
		&\quad +\mathscr{A}_{1}(z)^{-1}\mathscr{B}_{1}(s^{\frac{3}{2}} f_1(z))^{-1}E_1(z)^{-1} E_1'(z)\mathscr{B}_{1}(s^{\frac{3}{2}} f_1(z)\mathscr{A}_{1}(z) \nonumber\\
		&\quad+s^{\frac{3}{2}} f_1'(z)\mathscr{A}_{1}(z)^{-1}\mathscr{B}_{1}(s^{\frac{3}{2}} f_1(z))^{-1}\mathscr{B}_{1}'(s^{\frac{3}{2}}f_1(z))\mathscr{A}_{1}(z) + \mathscr{A}_{1}(z)^{-1}\mathscr{A}_{1}'(z),
	\end{align}
	where $E_1$ is defined in \eqref{def:E1},
	\begin{equation}\label{def:B1}
		\mathscr{A}_{1}(z) = \begin{pmatrix}
			e^{s^{\frac34}g_1(z)} & 0 & 0 & 0\\
			0 &1 & 0 & 0\\
			0 & 0 & e^{-s^{\frac34}g_1(z)} &0\\
			0 & 0 & 0 & 1
		\end{pmatrix}
	~~ 
	\textrm{and}
        ~~
		\mathscr{B}_{1}(z)=
		\begin{pmatrix}
			\Phi^{(\Bes)}_{11}(z) & 0 & -\Phi^{(\Bes)}_{12}(z) & 0\\
			0 & 1 & 0 & 0\\
			-\Phi^{(\Bes)}_{21}(z) & 0 & \Phi^{(\Bes)}_{22}(z)\\
			0 & 0 & 0 & 1
		\end{pmatrix}.
	\end{equation}
 Recall the following properties of the modified Bessel functions $I_0$ and $K_0$ (cf. \cite[Chapter 10]{DLMF}):
	\begin{align}
		I_0(z) & = \sum_{k=0}^{\infty} \frac{(z/2)^{2k}}{(k!)^2},\label{eq:I0}\\
		K_0(z)&=-\left(\ln \left(\frac{z}{2}\right) + \gamma_E\right)I_0(z) + \Boh(z^2), \qquad z \to 0,\label{eq:K0}
	\end{align}
	where $\gamma_E$ is the Euler's constant, we see from \eqref{def:B1} and \eqref{def:Bespara} that
	\begin{equation}\label{eq:evaBinBder}
		\lim_{z \to 0} \left(\mathscr{B}_{1}(z)^{-1} \mathscr{B}_{1}'(z)\right)_{31} = -\frac{\pi \ii}{2}.
	\end{equation}
 By using the local behavior of $E_1(z)$ in \eqref{def:localE1} and the explicit expression of $R_1'(1)$ in \eqref{eq:R1'1}, we have
	\begin{align}
	& \lim_{z \to 1} \left(E_1(z)^{-1} R(z)^{-1} R'(z) E_1(z)\right)_{31} 
 \nonumber  \\
  & = \lim_{z \to 1} \left(E_1(z)^{-1} \left(\frac{R_1'(z)}{s^{3/4}} + \Boh(s^{-3/2})\right) E_1(z) \right)_{31}\nonumber\\
		&= -\frac{\left(4\alpha^2-1\right) \pi \ii \left(\frac{1}{2} +\frac{ \tilde{s}}{\sqrt{s}}\right) }{8 \left(\frac{1}{2} +\frac{ \tilde{s}}{\sqrt{s}}+\tau s^{-1/4}\right)}+\frac{3\pi\ii  \left(\frac{1}{2} +\frac{ \tilde{s}}{\sqrt{s}}\right) }{32\left(1-\frac{ \tilde{s}}{2\sqrt{s}}\right)}+\Boh(s^{-\frac{3}{4}})
	\end{align}
	and
	\begin{equation}
		\lim_{z \to 1}\left( E_1(z)^{-1} E_1'(z)\right)_{31} =0.
	\end{equation}
To proceed, note that for an arbitrary $4 \times 4$ matrix $\msf M$,  we have, as $s \to +\infty$,
	\begin{align}
		\lim_{z \to 1} \left(\mathscr{A}_{1}(z)^{-1} \msf M \mathscr{A}_{1}(z)\right)_{31} &= \msf M_{31},
	\end{align}
	and
	\begin{equation}\label{eq:BinvMB}
		\lim_{z \to 0} \left(\mathscr{B}_{1}(z)^{-1} \msf M \mathscr{B}_{1}(z)\right)_{31}  = \msf M_{31}.
	\end{equation}

A combination of \eqref{eq:SinvSder} and \eqref{eq:evaBinBder}--\eqref{eq:BinvMB} yields
	\begin{align}\label{eq:SS'11}
		\lim_{z \to 1} \left(S(z)^{-1}S'(z)\right)_{31} &= \frac{\pi \ii}{4}\left(1 - \frac{2\tilde{s}}{\sqrt{s}}\right)^2 s^{\frac{3}{2}} - \frac{(4\alpha^2-1)\pi \ii}{8}+\frac{3\ii \pi}{64}+\Boh(s^{-\frac{1}{4}}).
	\end{align}
This, together with \eqref{ds-F-1}, gives us the first formula in \eqref{asy:H}. 
	
	\subsubsection*{Asymptotics of $H(s)$ as $s \to +\infty$ for $0\leq \gamma<1$}
	If $\gamma \in [0, 1)$, we turn to use the relation \eqref{eq:JMU}, which can be written as
	\begin{equation}\label{eq:HXRL}
		H(s) = -\frac{\gamma}{2 \pi \ii} \begin{pmatrix}
			0 & 0 & 1 & 0
		\end{pmatrix}  X_{R,1}(s) \begin{pmatrix}
			1\\0\\0\\0
		\end{pmatrix},
	\end{equation}
	where $X_{R,1}(s)$ is given in \eqref{eq: X-expand-s} .
	
	From \eqref{eq:X-near-s} and \eqref{eq: X-expand-s}, it follows that
	\begin{align}
		X_{R,1}(s) = \frac{X_{R,0}(s)^{-1}}{s} \lim_{z \to 1 ~\textrm{from}~ \Omega_1^{(1)}} \left[X(sz) \begin{pmatrix}
			1 & 0 & \frac{\gamma}{2 \pi \ii} \ln{(sz-s)} & 0\\
			0 & 1 & 0 & 0\\
			0 & 0 & 1 & 0\\
			0 & 0 & 0 & 1
		\end{pmatrix}\right]'.
	\end{align}
	Note that (see \eqref{eq:XR0})
	\begin{align}
			\begin{pmatrix}
				0 & 0 & 1 & 0
			\end{pmatrix} X_{R,0}(s)^{-1} &= \begin{pmatrix}
				0 & 0 & e^{-\tau s} & 0
			\end{pmatrix} \what{\Upsilon}_0^{-1}\what E_1(1)^{-1} \what R(1)^{-1} \diag (s^{\frac 18}, s^{\frac 18}, s^{-\frac 18}, s^{-\frac 18})\nonumber \\
			& \quad \times \operatorname{diag}\left(\begin{pmatrix}
				1 & -1 \\ 1 & 1 \end{pmatrix} ,\begin{pmatrix} 1 & 1 \\ -1 & 1 \end{pmatrix}\right)\operatorname{diag}\left(s^{-\frac{1}{4}}, s^{\frac{1}{4}},s^{-\frac{1}{4}}, s^{\frac{1}{4}}\right) ,
	\end{align}		
	and by tracing back the transformations $X \to \what{T} \to \what S \to \what R$ in \eqref{def:XToTgamma}, \eqref{def:hatS} and \eqref{def:hatR}, we have
	\begin{align}
		&\lim_{z \to 1 ~\textrm{from}~ \Omega_1^{(1)}} \left[X(sz) \begin{pmatrix}
			1 & 0 & \frac{\gamma}{2 \pi \ii} \ln{(z-s)} & 0\\
			0 & 1 & 0  & 0\\
			0 & 0 & 1 & 0\\
			0 & 0 & 0 & 1
		\end{pmatrix}\right]'\begin{pmatrix}
			1\\0\\0\\0
		\end{pmatrix} \nonumber\\
		&= \frac{1}{2}\diag (s^{\frac 14}, s^{-\frac 14}, s^{\frac 14}, s^{-\frac 14})\operatorname{diag}\left(\begin{pmatrix}
			1 & -1 \\ 1 & 1 \end{pmatrix} ,\begin{pmatrix} 1 & 1 \\ -1 & 1 \end{pmatrix}\right)\diag (s^{-\frac 18}, s^{-\frac 18}, s^{\frac 18}, s^{\frac 18})\nonumber\\
		&\quad  \times \lim_{z \to 1 ~\textrm{from}~ \Omega_1^{(1)}} \left[ \what R(z) \what P^{(1)} (z) \diag(e^{- \theta_1(\sqrt{sz}) + \tau \sqrt{sz}},e^{- \theta_2(\sqrt{sz}) - \tau \sqrt{sz}},e^{ \theta_1(\sqrt{sz}) + \tau \sqrt{sz}},e^{\theta_2(\sqrt{sz}) - \tau \sqrt{sz}}) \right]'\nonumber\\
       & \quad \times
       \begin{pmatrix}
			1\\0\\0\\0
       \end{pmatrix}
       \nonumber \\
		&=\frac{1}{2}\diag (s^{\frac 14}, s^{-\frac 14}, s^{\frac 14}, s^{-\frac 14})\operatorname{diag}\left(\begin{pmatrix}
			1 & -1 \\ 1 & 1 \end{pmatrix} ,\begin{pmatrix} 1 & 1 \\ -1 & 1 \end{pmatrix}\right) \diag (s^{-\frac 18}, s^{-\frac 18}, s^{\frac 18}, s^{\frac 18})\nonumber\\
		& \quad \times \left(\what R'(1) \what E_1 (1)\what{\Upsilon}_0 e^{\tau \sqrt{s}} + \what R(1) \what E_1' (1)\what{\Upsilon}_0 e^{\tau \sqrt{s}}\right. \left. +s^{\frac 34} \what f_1'(1) \what R(1) \what E_1(1)\what{\Upsilon}_0 \what{\Upsilon}_1 e^{\tau \sqrt{s}}\right.\nonumber \\
  &\left.\qquad + \frac{\tau \sqrt{s} }{2}\what R(1) \what E_1(1)\what{\Upsilon}_0e^{\tau \sqrt{s}}\right)\begin{pmatrix}
			1\\0\\0\\0
		\end{pmatrix},
	\end{align}
where $\what{\Upsilon}_0$ is defined in \eqref{def:whatUpsilon} and $\what{\Upsilon}_1$ is a $4\times 4$ matrix related to $\Upsilon_1$ in \eqref{eq:H-expand-2} with $(\what{\Upsilon}_1)_{31}=\frac{ \beta \pi \ii e^{-\beta \pi \ii} }{\sin(\beta \pi )}$. It then follows from the above three equations that
	\begin{align}\label{HR1}
		-\frac{\gamma}{2 \pi \ii} \begin{pmatrix}
			0 & 0 & 1 & 0
		\end{pmatrix} X_{R,1}(s) \begin{pmatrix}
			1\\0\\0\\0
		\end{pmatrix}&=-\frac{\gamma}{2 \pi \ii s} \left(\what{\Upsilon}_0^{-1}\what E_1(1)^{-1} \what R(1)^{-1}\what R'(1) \what E_1 (1)\what{\Upsilon}_0\right.\nonumber\\
		&\quad\left.+\what{\Upsilon}_0^{-1}\what E_1(1)^{-1}\what E_1' (1)\what{\Upsilon}_0+s^{\frac 34} \what f_1'(1) \what{\Upsilon}_1 +\frac{\tau \sqrt{s}}{2}  I \right)_{31}.
	\end{align}
	To calculate the first term in the bracket, we substitute \eqref{eq:hatE11}, \eqref{def:whatUpsilon} and \eqref{eq:hatR-1} into the above equation and obtain
	\begin{align}\label{term1}
		\left(\what{\Upsilon}_0^{-1}\what E_1(1)^{-1} \what R(1)^{-1}\what R'(1) \what E_1 (1)\what{\Upsilon}_0\right)_{31}=\Boh(s^{-\frac 14}).
	\end{align}
	
	From the explicit expression of $\what{\Upsilon}_0$ in \eqref{def:whatUpsilon} and $\what E_1(1)^{-1}\what E_1' (1)$ in \eqref{eq:hatE1'1}, one has
	\begin{align}\label{eq:hatUpsinv}
		\left(\what{\Upsilon}_0^{-1}\what E_1(1)^{-1}\what E_1' (1)\what{\Upsilon}_0\right)_{31}&=-\frac{7 \beta e^{-\beta \pi \ii}}{8} \Gamma (1+\beta) \Gamma (1-\beta)  \nonumber \\
		&\quad-\frac{\ii}{8} \left(\frac{\Gamma(1+\beta)^2}{\what {\textsf{a}}^2} + \what {\textsf{a}}^2 \Gamma(1-\beta)^2 e^{-2 \beta \pi \ii} \right)
		+ \Boh(s^{-\frac{1}{2}}),
	\end{align}
	where $\what {\textsf{a}}$ is given in \eqref{def:whata}. Since $\beta$ is purely imaginary, we have
	\begin{align}
		\Gamma (1+\beta) \Gamma (1-\beta) = |\Gamma (1+\beta)|^2=\frac{\beta \pi}{\sin (\beta \pi)}.
	\end{align}
This, together with \eqref{eq:hatUpsinv}, implies that
	\begin{align}\label{term2}
		\left(\what{\Upsilon}_0^{-1}\what E_1(1)^{-1}\what E_1' (1)\what{\Upsilon}_0\right)_{31}=-\frac{7 \beta^2 \pi e^{-\beta \pi \ii}}{8 \sin (\beta \pi)} -\frac{\beta \pi \ii e^{-\beta \pi \ii}}{4\sin (\beta \pi)} \cos \left(2 \vartheta (s)\right)+ \Boh(s^{-1}),
	\end{align}
	where $\vartheta (s)$ is given in \eqref{def:theta}. Also note from \eqref{eq:hatf1} that
	\begin{align}\label{term3}
		\left(s^{\frac 34} \what f_1'(1) \what{\Upsilon}_1\right)_{31} = \frac{ \beta \pi \ii e^{-\beta \pi \ii}}{\sin (\beta \pi)}(s^\frac{3}{4}-\tilde{s}s^\frac{1}{4}),
	\end{align}
	and
	\begin{align}\label{term4}
		\left(\tau s I \right)_{31} = 0.
	\end{align}
Substituting \eqref{def:beta}, \eqref{term1} and \eqref{term2}--\eqref{term4} into \eqref{HR1}, it is readily seen from \eqref{eq:HXRL} that
	\begin{equation}\label{9.29}
		H(s) =\beta \ii  s^{-\frac{1}{4}}-\beta \ii \tilde{s} s^{-\frac{3}{4}}-\frac{7 \beta^2 }{8}s^{-1}-\frac{\beta \ii }{4} \cos  \left(2 \vartheta (s)\right)s^{-1}+\Boh({s^{-\frac{5}{4}}}), \quad s\to +\infty, 
	\end{equation}
as required.	
	
\subsubsection*{Asymptotics of $H(s)$ as $s \to 0^+$}
	From \eqref{eq:H}, it follows that 
	\begin{align}\label{eq:sH}
		sH(s) &= -\frac{\gamma}{2 \pi \ii} \begin{pmatrix}
			0 & 0 & 1 & 0
		\end{pmatrix}X_{R,0}(s)^{-1} \left[sA_0+A_2(s)\right] X_{R,0}(s) \begin{pmatrix}
			1\\0\\0\\0
		\end{pmatrix}.
	\end{align}
	From \eqref{def:A0} and \eqref{def:R00}, we have
	\begin{align}
		\begin{pmatrix}
			0 & 0 & 1 & 0
		\end{pmatrix}X_{R,0}(s)^{-1} sA_0 X_{R,0}(s) \begin{pmatrix}
			1\\0\\0\\0
		\end{pmatrix}		=\Boh(s^{\frac{2\nu+1}{2}}).
	\end{align}
	From \eqref{def:A2}, \eqref{def:R00} and \eqref{def:L00}, we have if $\nu \neq 1/2$,
	\begin{align}
		&\begin{pmatrix}
			0 & 0 & 1 & 0
		\end{pmatrix}X_{R,0}(s)^{-1} A_2(s) X_{R,0}(s) \begin{pmatrix}
			1\\0\\0\\0
		\end{pmatrix}	\nonumber\\
		&= \frac{1-2\nu}{2}
		\begin{pmatrix}
			0 & 0 & 1 & 0
		\end{pmatrix}X_{R,0}(s)^{-1}X_{L,0}(s)\begin{pmatrix}
			0 & 0 & 0 & 0\\
			0 & 1 & 0 & 0\\
			0 & 0 & 1 & 0\\
			0 & 0 & 0 & 0
		\end{pmatrix}  X_{L,0}^{-1}(s) X_{R,0}(s) \begin{pmatrix}
			1\\0\\0\\0
		\end{pmatrix} \nonumber \\
		&=\frac{1-2\nu}{2}
		\begin{pmatrix}
			0 & 0 & s^{\frac{2\nu-1}{4}} & 0
		\end{pmatrix}\left(I+\Boh(s)\right) \begin{pmatrix}
			0 & 0 & \Boh(1) & 0\\
			0 & 1 & 0 & 0\\
			0 & 0 & 1 & 0\\
			0 & 0 & 0 & 0
		\end{pmatrix}\left(I+\Boh(s)\right) \begin{pmatrix}
			s^{\frac{2\nu-1}{4}}\\0\\0\\0
		\end{pmatrix}\nonumber \\ 
		&=\Boh (s^{\frac{2\nu+1}{2}});
	\end{align}
	and if $\nu = 1/2 $,
	\begin{align}\label{eq:nuhalf}
		&\begin{pmatrix}
			0 & 0 & 1 & 0
		\end{pmatrix}X_{R,0}(s)^{-1} A_2(s) X_{R,0}(s) \begin{pmatrix}
			1\\0\\0\\0
		\end{pmatrix}		\nonumber\\
		&= -\frac{1}{2\pi}
		\begin{pmatrix}
			0 & 0 & 1 & 0
		\end{pmatrix}X_{R,0}(s)^{-1}X_{L,0}(s)\begin{pmatrix}
			0 & 1 & \ii(\gamma-1) & 0\\
			0 & 0 & 0 & 0\\
			0 & 0 & 0 & 0\\
			0 & 0 & -1 & 0
		\end{pmatrix} X_{L,0}^{-1}(s) X_{R,0}(s) \begin{pmatrix}
			1\\0\\0\\0
		\end{pmatrix} \nonumber \\
		&=-\frac{1}{2\pi}
		\begin{pmatrix}
			0 & 0 & 1 & 0
		\end{pmatrix}\left(I+\Boh(s)\right) \begin{pmatrix}
			0 & 1 & \ii(\gamma-1) & 0\\
			0 & 0 & 0 & 0\\
			0 & 0 & 0 & 0\\
			0 & 0 & -1 & 0
		\end{pmatrix}\left(I+\Boh(s)\right) \begin{pmatrix}
			1\\0\\0\\0
		\end{pmatrix}=\Boh(s).
	\end{align}
A combination of \eqref{eq:sH}--\eqref{eq:nuhalf} gives us $sH(s)=\Boh(s^{\frac{2\nu+1}{2}})$ as $s\to 0^+$, which is \eqref{eq:H-0}. 

This finishes the proof of Theorem \ref{th:F-H}.
\qed

	\subsection{Proof of Theorem \ref{th:1}}
 Large $s$ asymptotics of $F(s; 1)$ in \eqref{ds-F} follows directly from \eqref{asy:H} and the integral representation of $F$ in \eqref{eq:F-H2}. To derive the large $s$ asymptotics of $F(s; \gamma)$ with $\gamma \in [0, 1)$, we first show 
	\begin{equation}\label{tauindep}
		\frac{\partial}{\partial \tau} F(s;\gamma,\tilde{s},\nu, \tau) = \Boh(s^{-\frac 14}), \quad s\to +\infty.
	\end{equation}
	Tracing back the transformations $X \to \what{T} \to \what{S} \to \what{R}$ in \eqref{def:XToTgamma}, \eqref{def:hatS} and \eqref{def:hatR}, we obtain that for $z \in \mathbb{C}\setminus\{D(0,\varepsilon)\cup D(1,\varepsilon)\}$,
	\begin{multline}
		X(sz)=X_0\diag{\left(s^{\frac 18}, s^{-\frac 38}, s^{\frac 38}, s^{-\frac 18}\right)} \what R(z) \what N(z)
		\\
		\times \diag{\left(e^{-\theta_1(\sqrt{sz})+\tau \sqrt{sz}}, e^{-\theta_2(\sqrt{sz})- \tau \sqrt{sz}}, e^{\theta_1(\sqrt{sz})+\tau
				\sqrt{sz}},e^{\theta_2(\sqrt{sz}) - \tau \sqrt{sz}} \right)}.
	\end{multline}
	Taking $z \to \infty$ and comparing the coefficient of $\Boh (1/z)$ term on both sides of the above formula, we have
	\begin{equation}\label{def:X11}
		X^{(1)} = s X_0 \diag{\left(s^{\frac 18}, s^{-\frac 38}, s^{\frac 38}, s^{-\frac 18}\right)} \left(\what R^{(1)}+\what N^{(1)}\right)\diag{\left(s^{-\frac 18}, s^{\frac 38}, s^{-\frac 38}, s^{\frac 18}\right)},
	\end{equation}
	where $\what N^{(1)}$ and $\what R^{(1)}$ are given in \eqref{eq:asyhatN} and \eqref{eq:asy:hatR}, respectively.
	In view of \eqref{eq:asyhatN1}, it is readily seen that 
	\begin{equation}\label{def:asyN1}
		\what N^{(1)}_{21}=\what N^{(1)}_{43}=2\beta^2.
	\end{equation}
	By \eqref{eq:asy:hatR}, \eqref{eq:estJhatR3}, \eqref{eq:hatR1} and \eqref{eq:hatR2}, one has
	\begin{align}
		\what R^{(1)} = s^{-\frac 14} \Res_{\zeta =0}\what J^{(0)}_1(\zeta) + s^{-\frac 12} \Res_{\zeta =0}(\what R_{1,-}(\zeta) \what J^{(0)}_1(\zeta)+\what J^{(0)}_2(\zeta)) +\Boh(s^{-\frac 34}).
	\end{align}
	Substituting \eqref{eq:hatR1}, the expressions of $\what J^{(0)}_1$ and $\what J^{(0)}_2$ in \eqref{def:hatJ01} and \eqref{def:hatJ02} into the above formula gives us
	\begin{align}
		\what R^{(1)}_{21} & = -2\beta(M^{(1)}_{13}+M^{(1)}_{14}) s^{-\frac 14} \nonumber\\
		&\quad+ \left(4\beta^2(M^{(1)}_{33}+M^{(1)}_{34}-M^{(1)}_{11}+M^{(1)}_{12}+(M^{(1)}_{13}+M^{(1)}_{14})^2)+(M^{(1)}_{11}-M^{(1)}_{12})\right)s^{-\frac 12}
  \nonumber \\
  &\quad+\Boh (s^{-\frac 34} ), \label{eq:hatR143asy}\\
  \what R^{(1)}_{43} & = -2\beta(M^{(1)}_{13}+M^{(1)}_{14}) s^{-\frac 14} \nonumber\\
		&\quad+4\beta^2(M^{(1)}_{33}+M^{(1)}_{34}-M^{(1)}_{11}+M^{(1)}_{12}+(M^{(1)}_{13}+M^{(1)}_{14})^2)-(M^{(1)}_{33}+M^{(1)}_{34}))s^{-\frac 12}
  \nonumber \\
  &\quad +\Boh (s^{-\frac 34} ).
  \label{eq:hatR121asy}
	\end{align}
Combining \eqref{def:X11}, \eqref{def:asyN1}, \eqref{eq:hatR121asy} and \eqref{eq:hatR143asy}, it follows that
	\begin{align}
		X^{(1)}_{21} & = 2 \beta^2 s^{\frac 12}  -2\beta(M^{(1)}_{13}+M^{(1)}_{14}) s^{\frac 14} \nonumber\\
		&\quad+4\beta^2(M^{(1)}_{33}+M^{(1)}_{34}-M^{(1)}_{11}+M^{(1)}_{12}+(M^{(1)}_{13}+M^{(1)}_{14})^2)
  \nonumber\\
		&\quad + M^{(1)}_{11}-M^{(1)}_{12} +\Boh (s^{-\frac 14} ), \label{eq:asyX111} 
       \\
		X^{(1)}_{43} & =  2 \beta^2 s^{\frac 12}  -2\beta(M^{(1)}_{13}+M^{(1)}_{14}) s^{\frac 14} \nonumber \\
		&\quad+ 4\beta^2(M^{(1)}_{33}+M^{(1)}_{34}-M^{(1)}_{11}+M^{(1)}_{12}+(M^{(1)}_{13}+M^{(1)}_{14})^2)
        \nonumber \\
		&\quad-(M^{(1)}_{33}+M^{(1)}_{34})+\Boh (s^{-\frac 14} ).
	\end{align}
	Substituting the above equations and \eqref{def:M2} into \eqref{eq:derivativeins-tau} yields \eqref{tauindep}.
	
As a consequence of \eqref{tauindep}, we conclude that	the first few terms till the constant term in the large $s$ asymptotics of $F(s; \gamma)$ are independent of $\tau$. We thus simply take $\tau=0$. By \eqref{eq:M11M12} and 
\cite[Corollary 1]{Del13}, we arrive at \eqref{eq:zerotau}, 
which implies $d=k$ from \eqref{def:d} and \eqref{def:k}. Thus, by \ref{def:C}, 
	\begin{align}\label{eq:Cnvs0}
		C_{\nu}(\tilde{s},0)&=\frac{4}{9} \beta \ii \bigg[10 (M^{(1)}_{31}+M^{(1)}_{32}) + (4\beta^2-6)(M^{(1)}_{13}+M^{(1)}_{14})M^{(1)}_{12}+ 36\beta^2  (M^{(1)}_{13}+M^{(1)}_{14})^3 \nonumber \\
		&\quad+ 8\beta^2 (M^{(1)}_{31}-M^{(1)}_{32})-(10+8\beta^2)(M^{(1)}_{13}+M^{(1)}_{14})^2 (M^{(1)}_{13}-M^{(1)}_{14})  \nonumber\\
  &\quad-(5 + 4 \beta^2)  (M^{(2)}_{13}+M^{(2)}_{14}) 
		\bigg]+\frac{16\beta^3\ii}{3}M^{(1)}_{14}
		\left(2M^{(1)}_{12} +(M^{(1)}_{14})^2-(M^{(1)}_{13})^2\right).
	\end{align}
When $\tau=0$, similar to the proof of \cite[Lemma 5]{Del13}, we can derive the following relations through direct calculations:
\begin{align}
	&M^{(1)}_{12}=\frac{(M^{(1)}_{13})^2-(M^{(1)}_{14})^2+\tilde{s}}{2},\\
	&M^{(2)}_{13}-M^{(1)}_{31}=M^{(1)}_{12}(M^{(1)}_{13}+M^{(1)}_{14})+\tilde{s}M^{(1)}_{13},\\
	&M^{(2)}_{14}=M^{(1)}_{12}(M^{(1)}_{13}+M^{(1)}_{14}),\\
	&M^{(1)}_{32}=-2M^{(1)}_{12}(M^{(1)}_{13}+M^{(1)}_{14})-\tilde{s}M^{(1)}_{14}+\ii\nu.
\end{align}
Inserting these relations into \eqref{eq:Cnvs0} gives us
\begin{align}\label{def:C0}
	C_{\nu}(\tilde{s},0)&=\frac{4}{9} \beta \ii \bigg[5M^{(1)}_{31}-28(M^{(1)}_{13}+M^{(1)}_{14})^2 (M^{(1)}_{13}-M^{(1)}_{14})-\tilde{s}(13M^{(1)}_{13}+18M^{(1)}_{14})+10\ii\nu\bigg]\nonumber\\
	&\quad +\frac{8}{9} \beta^3 \ii\bigg[2M^{(1)}_{31}+18(M^{(1)}_{13}+M^{(1)}_{14})^3-(M^{(1)}_{13}+M^{(1)}_{14})^2 (M^{(1)}_{13}-M^{(1)}_{14})\nonumber\\
 &\qquad\qquad\qquad+\tilde{s}(M^{(1)}_{13}+13M^{(1)}_{14})-4\ii\nu\bigg].
\end{align}
	
 To proceed, we note from \eqref{eq:differentialH1} that
	\begin{align}\label{eq:int0sH}
		\int_0^s H(t) \ud t & =\int_0^s  \left(\sum_{k=1}^{12}p_k(t)q_k'(t) -H(t)\right) \ud t +  [tH(t)]_{t=0}^s.
	\end{align}
	The integral on the right hand side of the above equation can be evaluated with the aid of \eqref{eq:differential:gamma}, where also holds if we replace $\gamma$ by $\beta$. By integrating both sides of \eqref{eq:differential:gamma} with respect to $s$, it is readily seen that
	\begin{multline}
		\frac{\partial}{\partial \beta} \int_0^s  \left(\sum_{k=1}^{12}p_k(t)q_k'(t) -H(t)\right) \ud t 
		\\ = \sum_{k=1}^{12}\left(p_k(s)\frac{\partial}{\partial \beta}q_k(s)-p_k(0)\frac{\partial}{\partial \beta}q_k(0) \right) =\sum_{k=1}^{12} p_k(s)\frac{\partial}{\partial \beta}q_k(s),
	\end{multline}
where the second equality follows from small $s$ asymptotics of $\sum p_k(s) \frac{\partial q_k(s)}{\partial \beta}$ established in Proposition \ref{th:pq}. Inserting the above formula into \eqref{eq:int0sH}, we arrive at
	\begin{align}\label{eq:asy:H3}
		\int_0^s H(t) \ud t = \int_0^{\beta}\sum_{k=1}^{12}p_k(s)\frac{\partial}{\partial \beta'}q_k(s) \ud \beta'+[tH(t)]_{t=0}^s.
	\end{align}
By using the results in Proposition \ref{th:pq} and large $s$ asymptotics of $H$ in \eqref{asy:H}, we have
	\begin{align}\label{9.49}
		\int_0^s H(t) \ud t&=\int_0^{\beta}\bigg[ 2\beta' \ii \frac{\partial \vartheta(s)}{\partial \beta'}+\frac{10}{3}\beta'+\frac{\ii}{3}s^{3/4}-3\ii \tilde{s} s^{1/4}+C_{\nu}(\tilde{s},0)+\frac{\partial(\beta' \ii \cos(2\theta(s))-\beta'^2)}{4\partial \beta'}
  \nonumber\\
  &\qquad\qquad+\Boh\left(s^{-\frac{1}{4}}\right)\bigg]\ud \beta'
		+\beta \ii  s^{\frac{3}{4}}-\beta \ii \tilde{s} s^{\frac{1}{4}}-\frac{7 \beta^2 }{8}-\frac{\beta \ii }{4} \cos  \left(2 \vartheta (s)\right)+\Boh({s^{-\frac{1}{4}}})
		\nonumber\\
		&=\frac{4}{3}\beta \ii  s^{\frac{3}{4}}-4\beta \ii \tilde{s} s^{\frac{1}{4}}+\frac{13}{24}\beta^2+\int_0^{\beta}\left(2\beta' \ii \frac{\partial \vartheta(s)}{\partial \beta'}+C_{\nu}(\tilde{s},0)\right)\ud \beta'+\Boh({s^{-\frac{1}{4}}}).
	\end{align}
Recall the definition of $\vartheta(s)$ and $C_{\nu}(\tilde{s},0)$ in \eqref{def:theta} and \eqref{def:C0}, we have
	\begin{align}\label{eq:intetheta}
		\int_0^{\beta}2  \beta'\ii \frac{\partial \vartheta(s)}{\partial \beta'} \ud \beta'&=\int_0^{\beta}2  \beta'\ii \left(\frac{\partial}{\partial \beta'}\arg \Gamma(1+\beta) + \ii \frac{3}{4} \ln s + \ii \ln8\left(1-\frac{\tilde{s}}{\sqrt{s}}\right)\right) \ud \beta'\nonumber\\
		&=\ln G(1+\beta)G(1-\beta)-\frac{3}{4}\beta^2 \ln s -\beta^2 \ln 8+\Boh(s^{-1}),
	\end{align}
 and 
	\begin{align}\label{def:D}
	& \int_0^{\beta} C_{\nu}(\tilde{s},0)\ud \beta' \nonumber \\
		& = \frac{2}{9} \beta^2 \ii \bigg[5M^{(1)}_{31}-28(M^{(1)}_{13}+M^{(1)}_{14})^2 (M^{(1)}_{13}-M^{(1)}_{14})-\tilde{s}(13M^{(1)}_{13}+18M^{(1)}_{14})+10\ii\nu\bigg]\nonumber\\
		&\quad +\frac{2}{9} \beta^4 \ii\bigg[2M^{(1)}_{31}+18(M^{(1)}_{13}+M^{(1)}_{14})^3-(M^{(1)}_{13}+M^{(1)}_{14})^2 (M^{(1)}_{13}-M^{(1)}_{14})\nonumber\\
  &\qquad\qquad\qquad+\tilde{s}(M^{(1)}_{13}+13M^{(1)}_{14})-4\ii\nu\bigg].
	\end{align}
By setting 
\begin{align}
	&D_{\nu,1}=\frac{2}{9}  \ii \bigg[5M^{(1)}_{31}-28(M^{(1)}_{13}+M^{(1)}_{14})^2 (M^{(1)}_{13}-M^{(1)}_{14})-\tilde{s}(13M^{(1)}_{13}+18M^{(1)}_{14})+10\ii\nu\bigg],  \label{def:D1} \\
	&D_{\nu,2}=\frac{2}{9}  \ii\bigg[2M^{(1)}_{31}+18(M^{(1)}_{13}+M^{(1)}_{14})^3-(M^{(1)}_{13}+M^{(1)}_{14})^2 (M^{(1)}_{13}-M^{(1)}_{14})\nonumber\\
&\qquad\qquad +\tilde{s}(M^{(1)}_{13}+13M^{(1)}_{14})-4\ii\nu\bigg],\label{def:D2}
\end{align}
we obtain the second formula in \eqref{ds-F} from \eqref{9.49}--\eqref{def:D2}. We emphasize that here the parameter $\tau=0$ in the matrix $M^{(1)}$, although not explicitly expressed.   

This completes the proof of Theorem \ref{th:1}. \qed

	\subsection{Proof of Corollary \ref{coro1}}
	It is readily seen that, as $x \to 0$,
	\begin{align}
		\mathbb{E} \left(e^{-2\pi x N(s)}\right) = 1-2 \pi \mathbb{E}(N(s)) x + 2 \pi^2 \mathbb{E}(N(s)^2) x^2 + \Boh(x^3).
	\end{align}
	Then we have
	\begin{align}\label{eq:ENs}
		F(s; 1-e^{-2\pi x}) &=\ln  \mathbb{E} \left(e^{-2\pi x N(s)}\right)\nonumber\\
		&=-2 \pi \mathbb{E}(N(s)) x + 2 \pi^2 \Var(N(s)) x^2 + \Boh(x^3), \qquad x \to 0,
	\end{align}
	where $F$ ie defined in \eqref{def:Fnotation}. In view of \eqref{ds-F}, we have
	\begin{align}\label{eq:FENs}
		F(s; 1-e^{-2\pi x}) &= -2 \pi \mu (s) x + 2 \pi^2\left(\sigma (s)^2 - \frac{-\frac{13}{24} + \ln 8-D_{\nu,1}}{2 \pi^2}\right)x^2 \nonumber\\
		&\quad+D_{\nu,2}x^4+ \ln (G(1+\ii x)G(1-\ii x)) + \Boh(s^{-\frac 14}), \quad s\to +\infty,
	\end{align}
	where the functions $\mu (s)$ and $\sigma (s)^2$ are defined in \eqref{def:mu-sigma}, $D_{\nu,1}$ and $D_{\nu,2}$ are defined in \eqref{def:D1} and \eqref{def:D2}.
	
	Note that
	\begin{align}
		G(1+z) = 1 + \frac{\ln (2\pi)-1}{2} z+\left(\frac{(\ln (2\pi)-1)^2}{8}-\frac{1 + \gamma_E}{2}\right) z^2 + \Boh(z^3), \qquad z \to 0,
	\end{align}
	where $\gamma_E$ is Euler's constant, we then obtain \eqref{def:EN} and \eqref{def:VarN} from \eqref{eq:ENs} and \eqref{eq:FENs}. Note that the additional factors $\ln s$ and $(\ln s)^2$ appear in the error terms due to the derivative with respect to $x$.
	
	Finally, since $\sigma (s)^2 \to +\infty$ for large positive $s$, it is easily shown that
	\begin{align}
		\mathbb{E}\left(e^{t \cdot \frac{N(s) - \mu (s)}{\sqrt{\sigma (s)^2}}}\right) \to e^{\frac{t^2}{2}}, \qquad s \to +\infty,
	\end{align}
	which implies the convergence of $\frac{N(s) - \mu (s)}{\sqrt{\sigma (s)^2}}$ in distribution to the normal law $\mathcal{N} (0,1)$. The upper bound \eqref{ub} follows directly from a combination of \cite[Lemma 2.1 and Theorem 1.2]{CC21}, \eqref{eq:ENs} and \eqref{eq:FENs}.
	
	This completes the proof of Corollary \ref{coro1}.
	\qed
	\section*{Acknowledgements}
Luming Yao was partially supported by National Natural Science Foundation of China under grant number 12401316. Lun Zhang was partially supported by National Natural Science Foundation of China under grant numbers 12271105, 11822104, and ``Shuguang Program'' supported by Shanghai Education Development Foundation and Shanghai Municipal Education Commission.
	
	\begin{appendices}

\section{The  Bessel parametrix}
\label{A}
The  Bessel parametrix is the unique solution of the following RH problem.
		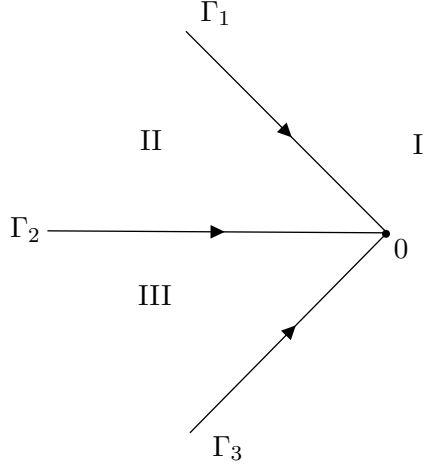
\begin{figure}[t]
			\centering

\tikzset{every picture/.style={line width=0.5pt}} 

\begin{tikzpicture}[x=0.75pt,y=0.75pt,yscale=-1,xscale=1]

\draw    (100,130) -- (269.5,131) ;
\draw [shift={(188.55,130.52)}, rotate = 180.34] [fill={rgb, 255:red, 0; green, 0; blue, 0 }  ][line width=0.08]  [draw opacity=0] (7.14,-3.43) -- (0,0) -- (7.14,3.43) -- cycle    ;
\draw   (169.12,29.61) -- (269.5,131) ;
\draw [shift={(221.98,83.01)}, rotate = 225.29] [fill={rgb, 255:red, 0; green, 0; blue, 0 }  ][line width=0.08]  [draw opacity=0] (7.14,-3.43) -- (0,0) -- (7.14,3.43) -- cycle    ;
\draw     (269.5,131) -- (171.12,231.61) ;
\draw [shift={(224.01,177.52)}, rotate = 134.36] [fill={rgb, 255:red, 0; green, 0; blue, 0 }  ][line width=0.08]  [draw opacity=0] (7.14,-3.43) -- (0,0) -- (7.14,3.43) -- cycle    ;
\draw  [fill={rgb, 255:red, 0; green, 0; blue, 0 }  ,fill opacity=1 ][line width=0.75]  (267.5,131.5) .. controls (267.5,130.67) and (268.17,130) .. (269,130) .. controls (269.83,130) and (270.5,130.67) .. (270.5,131.5) .. controls (270.5,132.33) and (269.83,133) .. (269,133) .. controls (268.17,133) and (267.5,132.33) .. (267.5,131.5) -- cycle ;

\draw (281,80) node [anchor=north west][inner sep=0.75pt]   [align=left] {I};
\draw (271.5,133) node [anchor=north west][inner sep=0.75pt]   [align=left] {0};
\draw (145,78) node [anchor=north west][inner sep=0.75pt]   [align=left] {II};
\draw (144,156) node [anchor=north west][inner sep=0.75pt]   [align=left] {III};
\draw (175,13) node [anchor=north west][inner sep=0.75pt]   [align=left] {$\Gamma_1$};
\draw (80,122) node [anchor=north west][inner sep=0.75pt]   [align=left] {$\Gamma_2$};
\draw (181,232) node [anchor=north west][inner sep=0.75pt]   [align=left] {$\Gamma_3$};

\end{tikzpicture}
			\caption{The jump contours $\Gamma_j$, $j=1,2,3$, and the domains I--III in the RH problem for $\Phi_{\alpha}^{(\text{Bes})}$.}
			\label{figure-Bessel}
		\end{figure}
		\begin{rhp}\label{rhp:Bessel}
			\hfill
			\begin{itemize}
				\item [\rm(a)] $\Phi_{\alpha}^{(\Bes)}(z)$ is defined and analytic in $\mathbb{C} \setminus (\cup_{j=1}^3 \Gamma_j \cup\{0\})$, where the contours $\Gamma_j$, $j=1,2,3$, are shown in Figure \ref{figure-Bessel}.
				\item [\rm(b)] For $z \in \cup_{j=1}^3 \Gamma_j$, we have
				\begin{equation}\label{eq:jump:Bessel}
					\Phi_{\alpha,+}^{(\Bes)}(z) = \Phi_{\alpha, -}^{(\Bes)}(z) \begin{cases}
						\begin{pmatrix}
							1 & 0\\
							e^{\alpha \pi \ii} & 1
						\end{pmatrix}, \quad & z \in \Gamma_1,\\
						\begin{pmatrix}
							0 &1\\
							-1 & 0
						\end{pmatrix}, \quad & z \in \Gamma_2,\\
						\begin{pmatrix}
							1 & 0\\
							e^{-\alpha \pi \ii} & 1
						\end{pmatrix}, \quad & z \in \Gamma_3,
					\end{cases}
				\end{equation}
    where $\alpha>-1$. 
				\item [\rm(c)] As $z \to \infty$, we have
				\begin{multline}\label{eq:infty:Bessel}
					\Phi_{\alpha}^{(\Bes)}(z) = \frac{(\pi^2 z)^{-\sigma_3/4}}{\sqrt{2}} \begin{pmatrix}
						1 & \ii\\
						\ii & 1
					\end{pmatrix}
					\left(I + \frac{1}{8z^{1/2}}\begin{pmatrix}
						-1-4\alpha^2 & -2 \ii\\
						-2 \ii & 1+4\alpha^2
					\end{pmatrix} + \Boh\left(\frac{1}{z}\right) \right)e^{z^{1/2}\sigma_3}.
				\end{multline}
				\item [\rm(d)] As $z \to 0$, we have
				\begin{equation}
					\Phi_{\alpha}^{(\Bes)}(z) = \begin{cases}
						\Boh (|z|^{\frac{\alpha}{2}})
      \quad & \alpha <0,\\
						\Boh (\ln|z|)
     \quad & \alpha =0,\\
						\Boh\begin{pmatrix}
							|z|^{\frac{\alpha}{2}} & |z|^{-\frac{\alpha}{2}}\\
							|z|^{\frac{\alpha}{2}} & |z|^{-\frac{\alpha}{2}}
						\end{pmatrix}, \quad & \textrm{$\alpha >0$ and $z \in \rm{I}$},\\
						\Boh\begin{pmatrix}
							|z|^{-\frac{\alpha}{2}} & |z|^{-\frac{\alpha}{2}}\\
							|z|^{-\frac{\alpha}{2}} & |z|^{-\frac{\alpha}{2}}
						\end{pmatrix}, \quad & \textrm{$\alpha >0$ and $z \in \rm{II} \cup \rm{III}$},
					\end{cases}
				\end{equation}
				where the domains $\rm{I}$--$\rm{III}$ are illustrated in Figure \ref{figure-Bessel}.
			\end{itemize}
		\end{rhp}
By \cite{KMVV04}, we have 
\begin{equation}\label{def:Bespara}
			\Phi_{\alpha}^{(\Bes)}(z) = \begin{cases}
				\begin{pmatrix}
					I_\alpha(z^{\frac12}) & \frac{\ii}{\pi} K_\alpha(z^{\frac12})\\
					\pi \ii z^{\frac12}I_\alpha'(z^{\frac12}) & -z^{1/2}K_\alpha'(z^{\frac12})
				\end{pmatrix}, &\quad z \in \rm{I},\\
				\begin{pmatrix}
					I_\alpha(z^{\frac12}) & \frac{\ii}{\pi} K_\alpha(z^{\frac12})\\
					\pi \ii z^{1/2}I_\alpha'(z^{\frac12}) & -z^{1/2}K_\alpha'(z^{\frac12})
				\end{pmatrix}\begin{pmatrix}
					1 & 0\\
					-1 & 1
				\end{pmatrix},& \quad z \in \rm{II},\\
				\begin{pmatrix}
					I_\alpha(z^{\frac12}) & \frac{\ii}{\pi} K_\alpha(z^{\frac12})\\
					\pi \ii z^{\frac12}I_\alpha'(z^{\frac12}) & -z^{1/2}K_\alpha'(z^{\frac12})
				\end{pmatrix}\begin{pmatrix}
					1 & 0\\
					1 & 1
				\end{pmatrix}, &\quad z \in \rm{III},
			\end{cases}
		\end{equation}
		where $I_\alpha(z)$ and $K_\alpha(z)$ denote the modified Bessel functions of order $\alpha$ (cf. \cite[Chapter 10]{DLMF}) and the principle branch is taken for $z^{1/2}$. 
		
		\section{The confluent hypergeometric parametrix}\label{sec:CHF}
The confluent hypergeometric parametrix is the unique solution of the following RH problem.

		\subsection*{RH problem for $\Phi^{(\CHF)}$}
		\begin{description}
			\item(a)   $\Phi^{(\CHF)}(z)=\Phi^{(\CHF)}(z;\beta)$ is analytic in $\mathbb{C}\setminus \{\cup^6_{j=1}\widehat\Sigma_j\cup\{0\}\}$, where the contours $\widehat\Sigma_j$, $j=1,\ldots,6,$ are indicated in Figure \ref{fig:jumps-Phi-C}.
			
			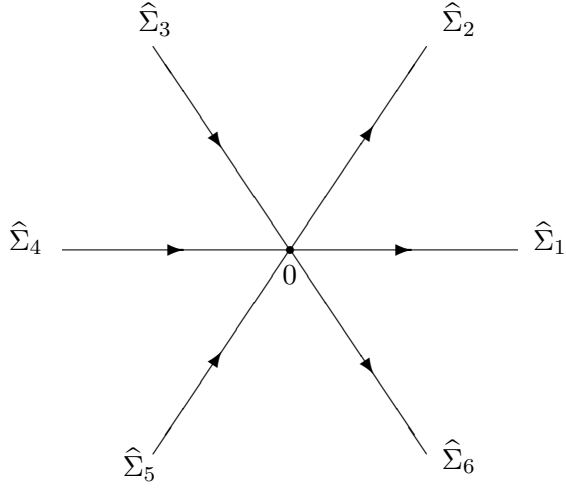
\begin{figure}[h]
				\begin{center}
					\setlength{\unitlength}{1truemm}
					\begin{picture}(100,70)(-5,2)
						\put(40,40){\line(-2,-3){18}}
						\put(40,40){\line(-2,3){18}}
						\put(40,40){\line(-1,0){30}}
						\put(40,40){\line(1,0){30}}
						\put(40,40){\line(2,-3){18}}
						\put(40,40){\line(2,3){18}}
						
						\put(30,55){\thicklines\vector(2,-3){1}}
						\put(25,40){\thicklines\vector(1,0){1}}
						\put(55,40){\thicklines\vector(1,0){1}}
						\put(30,25){\thicklines\vector(2,3){1}}
						\put(50,25){\thicklines\vector(2,-3){1}}
						\put(50,55){\thicklines\vector(2,3){1}}
						
						
						\put(39,35.5){$0$}
						\put(72,40){$\widehat \Sigma_1$}
						\put(60,69){$\widehat \Sigma_2$}
						\put(20,69){$\widehat \Sigma_3$}
						\put(3,40){$\widehat \Sigma_4$}
						\put(18,10){$\widehat \Sigma_5$}
						\put(60,11){$ \widehat \Sigma_6$}

						%
						
						\put(40,40){\thicklines\circle*{1}}
					\end{picture}
					\caption{The jump contours of the RH problem for $\Phi^{(\CHF)}$.}
					\label{fig:jumps-Phi-C}
				\end{center}
			\end{figure}
			
			\item(b) $\Phi^{(\CHF)}$ satisfies the jump condition
			\begin{equation}\label{CHFJumps}
				\Phi^{(\CHF)}_+(z)=\Phi^{(\CHF)}_-(z) \widehat J_j(z), \quad z \in \widehat\Sigma_j,\quad j=1,\ldots,6,
			\end{equation}
			where
			\begin{equation*}
				\widehat J_1(z) = \begin{pmatrix}
					0 &   e^{-\beta \pi \ii} \\
					-  e^{\beta \pi \ii} &  0
				\end{pmatrix}, \qquad \widehat J_2(z) = \begin{pmatrix}
					1 & 0 \\
					e^{ \beta \pi \ii } & 1
				\end{pmatrix}, \qquad
				\widehat J_3(z) = \begin{pmatrix}
					1 & 0 \\
					e^{ -\beta\pi \ii} & 1
				\end{pmatrix},                                                         
			\end{equation*}
			\begin{equation*}
				\widehat J_4(z) = \begin{pmatrix}
					0 &   e^{\beta\pi \ii} \\
					-  e^{-\beta\pi \ii} &  0
				\end{pmatrix}, \qquad
				\widehat J_5(z) = \begin{pmatrix}
					1 & 0 \\
					e^{- \beta\pi \ii} & 1
				\end{pmatrix},\qquad
				\widehat J_6(z) = \begin{pmatrix}
					1 & 0 \\
					e^{\beta\pi \ii} & 1
				\end{pmatrix}.
			\end{equation*}

			\item(c) As $z\to \infty$, we have 
			\begin{multline}\label{CHF at infinity}
				\Phi^{(\CHF)}(z)=\left(I +\frac{\Phi_1^{(\CHF)}}{z} +\Boh(z^{-2})\right) z^{-\beta \sigma_3}e^{-\frac{\ii z}{2}\sigma_3}
				\left\{\begin{array}{ll}
					I, & ~0< \arg z <\pi,
					\\
					\begin{pmatrix}
						0 &   -e^{\beta\pi \ii} \\
						e^{-\beta\pi \ii } &  0
					\end{pmatrix}, &~ \pi< \arg z<\frac{3\pi}{2},
					\\
					\begin{pmatrix}
						0 &   -e^{-\beta\pi \ii} \\
						e^{\beta\pi \ii} &  0
					\end{pmatrix}, & -\frac{\pi}{2}<\arg z<0,
				\end{array}\right.
			\end{multline}
			where 
			\begin{align}\label{def:CHF1}
				\Phi_1^{(\CHF)}=\begin{pmatrix}
					\beta^2 \ii & -\frac{\Gamma(1-\beta)}{\Gamma(\beta)} e^{-\beta \pi \ii}\ii \\
					\frac{\Gamma(1+\beta)}{\Gamma(-\beta)} e^{\beta \pi \ii}\ii &-\beta^2 \ii
				\end{pmatrix}.
			\end{align}
			\item(d) As $z\to 0$, we have $\Phi^{(\CHF)}(z)=\Boh(\ln |z|)$.
		\end{description}
		
By \cite{IK}, it follows that the above RH problem can be solved explicitly in terms of the confluent hypergeometric functions. Moreover, as $z \to 0$, we have
		
		\begin{equation}\label{eq:H-expand-2}
			\Phi^{(\CHF)}(z) e^{-\frac{\beta \pi \ii}{2} \sigma_3} = \Upsilon_0\left( I+ \Upsilon_1z+\Boh(z^2) \right) \begin{pmatrix} 1 & -\frac{\gamma}{2\pi \ii} \ln (e^{-\frac{\pi \ii}{2}}z) \\
				0 & 1  \end{pmatrix},
		\end{equation}
		for $z$ belonging to the region bounded by the rays $\widehat \Sigma_2$ and $\widehat \Sigma_3$, where $\gamma=1-e^{2\beta \pi \ii}$,
		\begin{align}\label{eq:H-expand-coeff-0}
			\Upsilon_0
			=\begin{pmatrix} \Gamma\left(1-\beta\right) e^{-\beta \pi \ii} &\frac{1}{\Gamma(\beta)} \left( \frac{\Gamma'\left(1-\beta\right)}{\Gamma\left(1-\beta\right)} +2\gamma_{\textrm{E}} \right) \vspace{5pt} \\
				\Gamma\left(1+\beta\right) & -\frac{e^{\beta \pi \ii}}{\Gamma(-\beta)} \left( \frac{\Gamma'\left(-\beta\right)}{\Gamma\left(-\beta\right) } +2\gamma_{\textrm{E}}\right) \end{pmatrix}
		\end{align}
		with $\gamma_{\textrm{E}}$ being the Euler's constant,
		and
		\begin{equation}\label{eq:H-expand-coeff-1}
			(\Upsilon_1)_{21}=\frac{ \beta \pi \ii \, e^{-\beta \pi \ii} }{\sin(\beta \pi )}.
		\end{equation}
	\end{appendices}
	
\end{document}